\documentclass[11pt,reqno]{amsart}
\title[$\mathrm{X}$-evolution flows and $\mathrm{X}$-foliations]
{From green mutation to $\mathrm{X}$-evolution: \\ Flows and foliations
on cluster complexes}% of 2-Calabi-Yau categories}

\author{Yu Qiu}
\address{Qy: Yau Mathematical Sciences Center and Department of Mathematical Sciences, Tsinghua University, 100084 Beijing, China.
    \&
Beijing Institute of Mathematical Sciences and Applications, Yanqi Lake, Beijing, China}
\email{yu.qiu@bath.edu}

\author{Liheng Tang}
\address{Tl: Qiuzhen College and Department of Mathematical Sciences, Tsinghua University, 100084 Beijing, China.}
\email{tang-lh20@mails.tsinghua.edu.cn}

\thanks{Qy is supported by National Natural Science Foundation of China (Grant No. 12425104).}
\dedicatory{Dedicated to Thomas Br\"{u}stle on the occasion of his sixtieth birthday}
\usepackage[a4paper]{geometry}
\usepackage{float}
\usepackage[usenames,dvipsnames]{xcolor}
\usepackage{subfigure}

\usepackage[colorlinks, linkcolor=blue,anchorcolor=Periwinkle,
   citecolor=red,urlcolor=Emerald]{hyperref}
\usepackage{amsmath}           % AMS 的主包
\usepackage{amssymb}           % AMS 符号包，会自动加载 amsfonts
\usepackage{latexsym}          % 几个额外的特殊符号，包括 \Box
\usepackage{bookmark}
\usepackage{enumitem}
\usepackage{array}
\usepackage{graphicx}
\usepackage[bottom]{footmisc}
\usepackage{cleveref}
\usepackage{transparent}

\usepackage{tikz}
\usetikzlibrary{matrix}
\usepackage{url}

\usetikzlibrary{decorations.markings,cd}
\usetikzlibrary{arrows}
\usetikzlibrary{decorations.pathreplacing,decorations.pathmorphing,shadings,fadings,calc}
\usepackage{extarrows,stmaryrd}

\tikzset{->-/.style={decoration={  markings,  mark=at position #1 with
    {\arrow{>}}},postaction={decorate}}}
\tikzset{-<-/.style={decoration={  markings,  mark=at position #1 with
    {\arrow{<}}},postaction={decorate}}}
\tikzcdset{arrow style=tikz, diagrams={>=stealth}}

\newcolumntype{L}{>{$}l<{$}} % math-mode version of "l" column type

\usepackage{tikz}
\usetikzlibrary{matrix, calc, arrows,cd,
decorations.pathreplacing, decorations.markings, decorations.pathmorphing,
backgrounds,fit,positioning,shapes.symbols,chains,shadings,fadings}

\usepackage{mathabx}

\usepackage{marvosym}

\theoremstyle{plain}
\newtheorem{theorem}{Theorem}[section]
\newtheorem{lemma}[theorem]{Lemma}
\newtheorem{corollary}[theorem]{Corollary}
\newtheorem{proposition}[theorem]{Proposition}

\newtheorem{lemdef}[theorem]{Definition/Lemma}
\newtheorem{conjecture}[theorem]{Conjecture}

\newtheorem{convention}[theorem]{Convention}

\theoremstyle{definition}
\newtheorem{definition}[theorem]{Definition}
\newtheorem{example}[theorem]{Example}
\newtheorem{remark}[theorem]{Remark}

\numberwithin{equation}{section}
\newcommand\hua{\mathcal}
\newcommand\NN{\mathbb{N}}
\newcommand\ZZ{\mathbb{Z}}

\newcommand\RR{\mathbb{R}}

\newcommand\<{\langle}
\renewcommand\>{\rangle}

\renewcommand{\setminus}{\smallsetminus}
\renewcommand{\emptyset}{\varnothing}

\newcommand{\Proj}{\operatorname{Proj}}

\newcommand\Add{\operatorname{add}}
\newcommand\Ind{\operatorname{Ind}}
\newcommand\Sim{\operatorname{Sim}}
\newcommand\Hom{\operatorname{Hom}}
\newcommand\End{\operatorname{End}}
\newcommand\Ext{\operatorname{Ext}}
\newcommand\Irr{\operatorname{Irr}}
\newcommand{\Int}{\operatorname{Int}}

\newcommand\Br{\operatorname{Br}} % braid group
\newcommand\Grot{\operatorname{K}} % Grothendieck group

\newcommand{\h}{\operatorname{\hua{H}}} %heart
\newcommand{\C}{\operatorname{\hua{C}}} % cluster category
\newcommand{\D}{\operatorname{\hua{D}}} % triangulated category
\newcommand{\per}{\operatorname{per}} % perfect derived category

\newcommand\Aut{\operatorname{Aut}} % automorphism group
\newcommand{\EG}{\operatorname{EG}} % oriented exchange graph of hearts
\newcommand{\SEG}{\operatorname{SEG}} % oriented exchange graph of siling
\newcommand{\ST}{\operatorname{ST}}  % spherical twist group
\newcommand{\CEG}{\operatorname{CEG}} % oriented cluster exchange graph
\newcommand{\uCEG}{\underline{\CEG}} % unoriented cluster exchange graph

\newcommand\ceg{\operatorname{\hua{CEG}}}

\renewcommand{\k}{\mathbf{k}}

\newcommand{\Cone}{\operatorname{Cone}}

\newcommand{\vsource}{\otimes}
\newcommand{\vsink}{\odot}
\newcommand{\vertx}{\bullet}

\newcommand{\on}[1]{\operatorname{#1}}

\newcommand\class{\mathfrak{Q}}
\usepackage{pifont}%\Ding

\def\nn{node{$\bullet$}}

\newcommand\surf{\mathbf{S}}  % marked surface

\def\Y{\mathcal{Y}}
\def\be{\begin{equation}}
\def\ee{\end{equation}}

\def\wt{\mathbf{w}}

\def\CCx{\on{Cpx}}   %cluster complex
\def\cc{\on{cell}}  %cell of pcto
\def\Star{\on{Star}}    %the topological star
\def\link{\on{Link}}    %the topological link
\def\susp{\Sigma}       %the suspension
\def\Serre{\mathbb{S}}  %the serre functor
\def\SS{\mathbf{S}}    %the Sphere

\def\ww{W}
\def\vv{V}
\def\uu{U}

\def\ctc{\mathcal{\vv}}   %the cluster tilting subcategory
\def\cts{\mathbf{\vv}}    %the cluster tilting set
\def\cto{\mathbf{\vv}_\oplus} %the cluster tilting object
\def\Red{\operatorname{Red}}%CY reduction
\newcommand{\App}[3]{\on{App}^{\text{#1}}_{#2}(#3)}%approximation

\def\pvd{\on{pvd}}%perfectly valued category
\newcommand{\DynA}[1]{A_{#1}}%Dynkin and Euclidean diagram

\newcommand{\DynE}[1]{E_{#1}}
\newcommand{\EucA}[2]{\widetilde{A_{#1,#2}}}
\newcommand{\EucD}[1]{\widetilde{D_{#1}}}
\newcommand{\EucE}[1]{\widetilde{E_{#1}}}
\def\ar{\on{AR}}%ARquiver
\def\pl{\mathbb{P}^1}%projective line
\def\coh{\mathrm{coh}}%coherent sheaves

\def\source{X[1]}
\def\sink{X}
\def\ueot{\Psi^{\downarrow}_{\sink}}%universal ext on the top
\def\ueby{\Psi^{\uparrow}_{\sink}}%universal ext by
\def\ue{\Psi^{?}_{\sink}}
\def\Xevo{\ccpt{X}[1]}

\newcommand{\dflow}[1]{\on{Evo}^{\downarrow}_{#1}}%the downward flow direction
\newcommand{\uflow}[1]{\on{Evo}^{\uparrow}_{#1}}

\def\Traj{L}
\def\num{m}

\newcommand{\Dflow}[1]{\widetilde{\on{Evo}}^{\downarrow}_{#1}}
\newcommand{\Uflow}[1]{\widetilde{\on{Evo}}^{\uparrow}_{#1}}

\def\ori{\mathrm{Tr}^{\uparrow}_{\sink}}%the origin
\def\des{\mathrm{Tr}^{\downarrow}_{\sink}}%the destiny
\newcommand{\pctc}[1]{\mathcal{#1}}%partial cluster tilting subcategory
\newcommand{\pcts}[1]{\mathbf{#1}}%partial cluster tilting set
\newcommand{\pcto}[1]{\mathbf{#1}_\oplus} %the cluster tilting object
\newcommand{\ccpt}[1]{\mathrm{#1}}%the point in ccx
\newcommand{\real}[1]{\underline{#1}}%realize pcts as points
\newcommand{\dapp}[1]{\pcts{#1}^{\on{apr}}}%downward approx
\def\Ps{\on{Ps}}
\newcommand{\sqd}[3]{\on{Sqd}_{#1,#2,#3}}%the squid algebra
\def\vect{\on{Vect}}%Vector bundle
\def\tube{\on{Tube}}%Tube
\def\rgd{\circledR} %rigid
\def\Psp{\on{Ps}_{+}^\rgd}%section plus
\def\Psm{\on{Ps}_{-}^\rgd}%section minus
\def\uat{\underline{\tau}}%inverse tau
\newcommand{\mut}[2]{\pcts{#1}_{\hat{#2}}}%mutations/codimension-cell

\def\DQ{\D(Q)}
\def\CQ{\C(Q)}
\def\Qminus{Q\backslash\{0\}}%Q minus vertex 0
\def\plpqr{\pl_{p,q,r}}%the wpl with weight p,q,r
\def\II{\on{ITV}}%the interval
\def\xevo{\ccpt{\sink}}%general \sink
\def\invsum{\frac{1}{p}+\frac{1}{q}+\frac{1}{r}}%inverse sum
\newcommand{\proj}[1]{\on{pr}_{\link(#1)}}%projection

\def\con{\sigma}
\def\fan{\Delta}

\begin{document}
%=========================================================
\begin{abstract}
For any point $\mathrm{X}$ in the cluster complex $\mathrm{Cpx}(\mathcal{C})$ of a 2-Calabi-Yau category $\mathcal{C}$, we introduce $\xevo$-evolution flow on $\mathrm{Cpx}(\mathcal{C})$. We show that such a flow induces a piecewise linear one-dimensional $\mathrm{X}$-foliation with two singularities, the unique sink $\mathrm{X}$ and the unique source $\mathrm{X}[1]$.
Moreover, we show that evolution flows on cluster complexes are continuous refinement/generalization of green mutations on cluster exchange graphs.

For the cluster category of a Dynkin or Euclidean quiver $Q$,
we prove that the $\mathrm{X}$-foliation is compact or semi-compact, for various choices of $\mathrm{X}$. As an application, we show that $\mathrm{Cpx}(\mathcal{C})$ is spherical (Dynkin case) or contractible (Euclidean case). As a byproduct, we show that the fundamental group of the cluster exchange graph of $Q$ is generated by squares and pentagons.

\end{abstract}
\maketitle

\tableofcontents \addtocontents{toc}{\setcounter{tocdepth}{1}}
\setlength\parindent{0pt}
\setlength{\parskip}{5pt}

%=========================================================
%=========================================================
\section{Introduction}
%=========================================================
\subsection{Motivation}\
%=========================================================
%=========================================================
\paragraph{\textbf{Cluster complexes}}\

Arc (in general, curve) complexes are crucial tools in the study of mapping class groups of surfaces, cf. e.g. \cite{FM}. Motivated by the cluster theory,
Fomin-Shapiro-Thurston \cite{FST} studied cluster complexes,
a generalization of arc complexes, for a marked surface $\surf$.
Such simplicial complexes are naturally associated with cluster algebras in general.
When the marked surface $\surf$ has no punctures,
the cluster complex $\CCx(\surf)$ coincides with the arc complex $A(\surf)$ of $\surf$.

Hatcher \cite{Hat} proved that the arc complexes are usually contractible, except for two simple cases--
in which cases, they are spherical (=homotopy to some spheres).
Our original motivation is to give a categorification of his proof (to calculate the homotopy type of cluster complexes) that applies to all Dynkin and Euclidean cases.
We fail to do precisely that but end up with a novel purely representation theoretical approach.
The tool developed in the process seems to be more interesting than the application itself.
More precisely, we construct a continuous family of flows/foliations on cluster complexes,
which provides a dynamical tool in the study of cluster theory.

%=========================================================
\paragraph{\textbf{Green mutation}}\

Green mutation was introduced by Keller, cf. \cite{Ke2}, with motivation coming from cluster theory and Donaldson-Thomas (DT) theory.
It can be defined combinatorially using Fomin-Zelevinsky's quiver mutation (for the principal extension of the initial quiver).
Alternatively, green mutation can be interpreted as an orientation of the cluster exchange graph
with a unique source and a unique sink.
The existence of a green-to-red mutation sequence (=source-to-sink directed path) plays a key role in proving many results.
For instance, such a sequence provides a formula to calculate DT-invariants using products of quantum dilogarithms.
Surprisingly, our flows/foliations on cluster complexes turn out to be are continuous refinement/generalization of (discrete) green mutations on cluster exchange graphs.
%This was a conjecture in the first version of the paper but now we have a precise statement. using tropical duality.

%=========================================================
\subsection{$\xevo$-evolution flows and $\xevo$-foliations}\
%=========================================================

Initiated by \cite{BMRRT} (cf. the survey \cite{Ke2}),
additive categorification of cluster algebras provides a powerful tool that successfully solves many conjectures in cluster theory.
For us, the cluster complex $\CCx(\C)$ will be defined
for a 2-Calabi-Yau category $\C$ admitting clusters (=cluster tilting sets) and
the cells $\pcts{\vv}$ of $\CCx(\C)$ correspond to partial clusters (cf. \Cref{conv}).
Note that any rigid object $V=\oplus_{i=0}^k V_{i}^{d_i}$, for non-isomorphic indecomposables $V_i$'s,
can be realized as a point
$
    \real{V}=\left( {\sum_{i=0}^k d_i V_i} \right) \big/ \left( {\sum_{i=0}^k d_i} \right)
$
in $k$-cell $\pcts{\vv}=\{V_i\}_{i=0}^k$ of $\CCx(\C)$.

%=========================================================
\paragraph{\textbf{$\sink$-evolving triangles}}\

Fix a rigid indecomposable $\sink$ in $\C$, i.e. a vertex (0-cell) in $\CCx(\C)$.
We say a triangle
\begin{gather}\label{eq:evo-tri0}
    \sink \xrightarrow{g} \ww \to \uu \xrightarrow{f} \source
\end{gather}
is a \emph{$\sink$-evolving triangle} if $\Ext^1(\uu,\ww)=0$ and $\Add(\ww)\cap\Add(\uu)=0$ (cf. \Cref{def:evo-tri}).

For a $k$-cell $\pcts{\vv}$, its \emph{downward (resp. upward) $\sink$-evolving triangle} is such a triangle satisfying,
in addition, that $f$ (resp. $g$) is the right (resp. left) approximation of $\source$ (resp. $\sink$) in $\Add(\pcto{\vv})$.
We will focus on the downward version from now on (and they are compatible in certain sense, cf. \Cref{lem:non-split}).

Note that for a top-cell $\pcts{\vv}$ (i.e. when $\pcts{\vv}$ is a cluster),
its downward and upward $\sink$-evolving triangles coincide and,
also coincide with the triangle that defines the index/$g$-vector for $\source$,
with respect to $\pcto{\vv}$, studied in \cite{DK} (cf. \cite{KR}).

%The point is that instead of studying the indices of all rigid indecomposables with respect to a fixed cluster, we study the indices of a fixed rigid indecomposable with respect to all clusters.

%=========================================================
\paragraph{\textbf{A $\CCx(\C)$-family of $\xevo$-evolution flows}}\

\begin{definition}[\Cref{def:flow}]
Let $\ccpt{\vv}$ be an (interior) point in $\CCx(\C)$ of a $k$-cell $\pcts{\vv}$
with the downward $\sink$-evolving triangle \eqref{eq:evo-tri0}.
Define the (downward) $\sink$-evolution flow at $\ccpt{\vv}$ to be
\begin{gather}\label{def:dflow-med0}
    \dflow{\sink}(\ccpt{\vv})=
    \begin{cases}
      \real{\ww}-\ccpt{\vv}, & \mbox{if $\uu=0\Leftrightarrow \ww=\sink$}, \\
      \real{\ww}-\real{\uu}, & \mbox{if $\ww\ne0\ne \uu$}, \\
      \ccpt{\vv}-\real{\uu}, & \mbox{if $\ww=0 \Leftrightarrow \uu=\source $}.
    \end{cases}
\end{gather}

\end{definition}
By definition, the only singularities of the $\sink$-evolution flow is at $\sink$ and $\source$,
in the sense that $\dflow{\sink}(\ccpt{\vv})=0$ if and only if $\ccpt{\vv}$ is $\sink$ or $\source$.

The $\sink$-evolutions flows, for any vertex $\sink$ of $\CCx(\C)$
can be extended to $\xevo$-evolution flows, for any point $\xevo=\sum_{i=0}^k c_i\sink_i$ in $\CCx(\C)$,
by linear combination in the obvious way:
\[
    \dflow{\xevo}=\sum_{i=0}^k c_i \dflow{\sink_i},
\]
despite that well-definedness is not so obvious (cf. \Cref{lem:well-defined}).
Similar as above, the only singularities of the $\xevo$-evolution flow is at $\xevo$ and $\Xevo$ (cf. \Cref{lem:sing}).

%=========================================================
\paragraph{\textbf{Inducing $\xevo$-foliations}}\

One can think of the $\xevo$-evolution flows are like lights:
$\Xevo$ is the light source and $\xevo$ is the light sink;
each top-cell in the cluster complex is a kind of medium.
When propagating within the interior of a top-cell, the $\xevo$-evolution flow has a fixed direction.

Refraction happened during cell-crossing.

The first result about general $\xevo$-evolution flows is the following, where
we require a mild condition (cf. the beginning of \Cref{sec:evoCC}) on the 2-Calabi-Yau category $\C$.%, which apply to, say the cluster category of any Jacobi-finite non-degenerate quiver with potential.

\begin{theorem}[\Cref{thm:foliation}]
For any point $\xevo$ in $\CCx(\C)$, the $\xevo$-evolution flow induces a piecewise linear one-dimensional $\xevo$-foliation with two singularities, the unique sink $\xevo$ and the unique source $\Xevo$.
\end{theorem}

We expect that the topological shape of the $\xevo$-foliation is like
electric field lines of two equipotential point charges with opposite signs shown in \Cref{fig:2shapes}.
In the left picture, one has the compact/spherical case:
the charges are in the poles of a sphere and the foliations are longitudes.
In the right picture, one has the semi-compact/Euclidean case:
some flows go/come from infinity.
\begin{figure}[hbt]\centering\makebox[\textwidth][c]{
  \includegraphics[height=8cm]{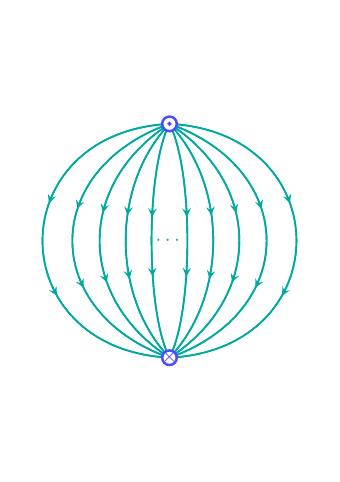}\hskip -1cm
  \includegraphics[height=8cm]{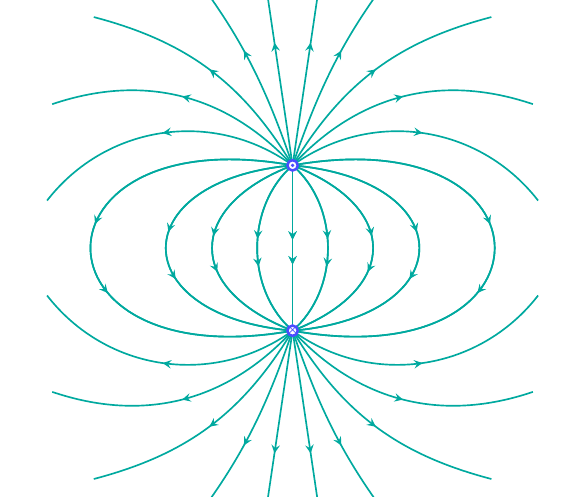}
  }
\caption{Two types of $\sink$-foliation shapes: compact and semi-compact} \label{fig:2shapes}
\end{figure}

When the $\xevo$-foliation is compact, i.e. every $\xevo$-leaf is compact,
then the flow induces a contraction from $\CCx(\C)\setminus\{\Xevo\}$ to $\{\xevo\}$
(and dually a contraction from $\CCx(\C)\setminus\{\xevo\}$ to $\{\Xevo\}$).
If $\xevo=\sink$ is a vertex, the $\sink$-trace induces:
\begin{itemize}
\item a deformation retract from $\CCx(\C)\setminus\{\xevo,\Xevo\}$ to $\CCx(\C\backslash\xevo)$, and
\item a homotopy equivalence
\begin{gather} \label{eq:inductive}
    \CCx(\C)\simeq \Sigma\CCx(\C\backslash\sink),
\end{gather}
where $\C\backslash\sink$ is the Calabi-Yau reduction of $\C$ by $\sink$, cf. \Cref{sec:CTS}.
\end{itemize}

%=========================================================
\paragraph{\textbf{Green mutations are special $\xevo$-evolutions}}\

The philosophies of the terminology `evolution' is the following.
\begin{itemize}
\item Mutations of (partial) clusters are local (on the level of `genes of individuals').
\item $\xevo$-evolutions are global (on the level of `species') and
the chosen point $\xevo$ controls the direction of mutation (i.e. playing the role of `environment' in the evolution).
\end{itemize}
In fact, we show that Keller's green mutation is induced by special $\xevo$-evolutions,
which was a conjecture in the previous version of the paper.

Recall that given an (initial) cluster $\pcts{Y}$, Keller \cite{Ke2} introduces the green mutation using combinatorics of the principal extension of the associated quivers.
In other words, the initial cluster $\pcts{Y}$ induces a preferable direction of the mutation (the green ones).
Such a direction is equivalent to an orientation of the cluster exchange graph, cf. \cite[\S~9]{KQ1}..
In the general setting, the signs of c-vectors determine such this direction.

Categorically, when identifying the cluster exchange graph with an interval in the exchange graphs of hearts in the corresponding 3-Calabi-Yau categories the orientation can be easily interpreted as simple backward tilting,
cf. \cite[Thm.~1.1]{Q12}.
We prove the following theorem, which indicates that
$\xevo$-evolution flow on cluster complexes is a continuous generalization/refinement of
green mutation on cluster exchange graphs.

\begin{theorem}[\Cref{thm:green}]
Suppose $\C$ is the cluster category of a Jacobi-finite non-degenerate quiver with potential.
Let $\xevo$ be a \emph{generic} point in $\CCx(\C)$, i.e. an interior point of a top-cell.
Then the downward $\xevo$-evolution flow induces an orientation on the unoriented cluster exchange graph,
given by the green mutation with respect to $\pcts{\sink}[1]$.
\end{theorem}

For instance, \Cref{fig:A3-F-F*} and \Cref{fig:A12-F-F*} show how the theorem works
for the $A_3$ case and the $\EucA{1}{2}$ case respectively.

%=========================================================
\subsection{Case study and applications}\

We investigate the case when $\C=\CQ$ is the cluster category associated with a Dynkin or an Euclidean quiver, which demonstrates all the philosophies/phenomena above.

\begin{theorem}[\Cref{thm:Dynkin}, \Cref{thm:Ess} and \Cref{thm:E-charge}]
Let $\C=\CQ$.
\begin{itemize}
  \item The $\sink$-foliation is compact, for any Dynkin quiver $Q$ with any $\sink$ in $\Ind\CQ$.
  \item The $\sink$-foliation is compact, for any Euclidean quiver $Q$ with any rigid regular simple $\sink$.
  \item The $\sink$-foliation on $\CQ$ is semi-compact, for an affine type A quiver $Q$ and any rigid indecomposable $\sink$ not in the tubes.
\end{itemize}
\end{theorem}

%=========================================================
\paragraph{\textbf{Homotopy type of cluster complexes}}\

The first application is that we can calculate the homotopy type of the cluster complexes in these cases. Depending on the choice of $\sink$, we obtain many inductive homotopy relations between various cluster complexes.

In the Dynkin case, let $\sink$ is in the $\tau$-orbit of the projective $P_0$, for some vertex $0$ of $Q$.
Then $\CQ\backslash\sink\cong\C(\Qminus)$ and hence \eqref{eq:inductive} becomes
\begin{gather}
    \CCx(\CQ)\simeq \Sigma\CCx(\C(\Qminus)),
\end{gather}
where $\Qminus$ is the quiver obtained from $Q$ by deleting vertex $0$.

In the Euclidean case, suppose that $Q$ is connected and derived equivalent to the weighted projective line $\pl_{p,q,r}$
for some $p,q,r\in\ZZ_+$ with $1/p+1/q+1/r>1$.
% We have $\CQ\cong\C(\pl_{p,q,r})$.
Take $\sink$ to be any rigid indecomposable in the bottom of a fat tube, say the one of rank $r\ge2$.
Then $\CQ\backslash\sink\cong\C(\pl_{p,q,r-1})$ and hence \eqref{eq:inductive} becomes
\begin{gather}
    \CCx( \C(\pl_{p,q,r}) )\simeq \Sigma\CCx( \C(\pl_{p,q,r-1}) ).
\end{gather}

In either case, we can inductively deduce the following.

\begin{corollary}
In particular, $\CCx(\CQ)$ is homotopy to an $(n-1)$-sphere if $Q$ is a Dynkin quiver with $n$ vertices
and is contractible if $Q$ is an Euclidean quiver.
\end{corollary}

Previous results concerning homotopy types of arc/cluster complexes include:
\begin{itemize}
  \item Hatcher \cite{Har} for arc complexes of marked surfaces (gentle type);
  \item Fomin-Shapiro-Thurston \cite{FST} for cluster complexes of marked surface
     (equipped punctures with extra $\ZZ_2$-symmetry, i.e. skew-gentle type);
  \item Igusa-Orr-Todorov-Weyman \cite{IOTW} for cluster complexes of Dynkin quivers.
\end{itemize}

%=========================================================
\paragraph{\textbf{Fundamental groups of cluster exchange graphs}}\

Let $\uCEG(\C)$ be the (unoriented) cluster exchange graph of $\C$,
which is the graph dual to the top two dimensional cells of $\CCx(\C)$.
It is well-known that there are squares and pentagons in $\uCEG(\C)$, cf. \Cref{sec:ceg}.
In many scenario, one expects that these squares and pentagons are generators of $\pi_1\uCEG(\C)$.
Previous results include:
\begin{itemize}
  \item Hatcher and Fomin-Shapiro-Thurston, as above
  for the marked surfaces case;
  \item Qiu \cite{Q1} (cf. \cite{QW}) for Dynkin quivers.
\end{itemize}
The $\xevo$-evolution flows provide a uniform approach to attack such a question
and we prove the following.

\begin{corollary}\label{cor:pi1}
For a Dynkin/Euclidean quiver $Q$,
$\pi_1\uCEG(Q)$ is generated by squares and pentagons.
\end{corollary}

%=========================================================
\subsection{Further studies}\label{sec:further}\
%=========================================================

%=========================================================
%The first interesting observation is about
%the relationship between $\xevo$-foliation and orientations of the cluster exchange graph $\uCEG(\C)$.
%
%Recall that after choosing an initial cluster $\pcts{\vv}=\{\vv_i\}_{i=0}^{n-1}$,
%there is an induced orientation on $\uCEG(\C)$,
%making $\pcts{\vv}$ and $\pcts{\vv}[1]$ to be the unique source/sink (depending on the convention).
%We will call it the forward/backward $\pcts{\vv}$-orientations (which are opposite to each other) of $\uCEG(\C)$.
%The oriented edges is defined by forward/backward simple tilting of the corresponding hearts
%in the associated 3-Calabi-Yau category, see \cite[\S9]{KQ1} for details.
%Alternatively, such an orientation can be combinatorially defined via green mutation (cf. \cite{Ke2} and \cite[Thm.~1.1]{Q12}).
%
%\begin{conjecture}[Green mutations are special $\xevo$-evolutions]\label{conj2}\
%Let $\xevo=( \sum_{i=0}^{n-1} \vv_i )/n$ be the center of the top-cell $\cc(\pcts{\vv})$,
%for any cluster $\pcts{\vv}=\{\vv_i\}_{i=0}^{n-1}$.
%Then upward/downward $\xevo$-foliation induces an orientation on the unorientated cluster exchange graph,
%which coincides with the forward/backward $\pcts{\vv}$-orientation.
%\end{conjecture}
%
%In other words, $\xevo$-foliations are continuous refinements of green mutations.
%For instance, \Cref{fig:A3-F-F*} confirms this conjecture for the $A_3$ case
%and \Cref{fig:A12-F-F*} confirms this conjecture for the $\uCEG(\EucA{1}{2})$ case.

A general expectation that $\xevo$-evolution flows have directions (down toward $\xevo$ and up toward $\Xevo$)
and hence there should be no cycles. This is compatible with the acyclic property of the green mutation.
%Note that the previous conjecture, if holds, provides evidence for the acyclicity conjecture below.
\begin{conjecture}[Acyclicity]\label{conj1}\
Any $\xevo$-foliation is acyclic.
\end{conjecture}

For Dynkin and Euclidean case, we have more specific predictions.
\begin{conjecture}
\label{conj-DE}\
Let $\C=\CQ$ be the cluster category of an acyclic quiver.
\begin{itemize}
  \item If $Q$ is a Dynkin quiver, then any $\xevo$-foliation is compact.
  \item If $Q$ is an Euclidean quiver, then any $\xevo$-foliation is semi-compact.
  Moreover, when $\xevo=\sink$ is a vertex, then
    the $\sink$-foliation is compact if and only if $\sink$ in some tube.
\end{itemize}
\end{conjecture}
Not that in the Dynkin case, the acyclicity of $\xevo$-foliation is equivalent to compactness of $\xevo$-foliation,
due to finiteness.

Following \Cref{cor:pi1},
we expect the following for the cluster braid groups (see \cite[\S2]{KQ2} for details).
Recall that the cluster groupoid $\ceg(\C)$ of $\C$ is an enhancement of $\uCEG(\C)$,
obtained from which by doubling of $\uCEG(\C)$ (=replacing each unoriented edges by a 2-cycle)
and gluing oriented square/pentagon/hexagons for each unoriented square/pentagons.
One can regard it as a 2-dim CW-complex.
Denote by $\Br(Q)$ the Artin braid group associated to $Q$ (cf. \cite[App.~C]{Q24}).
We will prove the following conjecture in a follow-up work in the near future.

\begin{conjecture}\label{conj3}
For any Dynkin/Euclidean quiver $Q$, there is canonical isomorphism
\begin{equation}\label{eq:Br}
  \pi(\ceg(Q),\pcts{\vv}_Q)\cong\Br(Q),
\end{equation}
sending the standard generators to the standard ones.
Here, $\pcts{\vv}_Q$ is the initial cluster associated to $Q$.
\end{conjecture}
Note that such a conjecture holds for Dynkin quiver, essentially due to \cite{QW}
but was reformulated in such a form in \cite[Thm.~2.16]{KQ2}.
A variation of the conjecture also holds for braid twist groups; see \cite[Thm.~3.17]{KQ2}, which covers type $\widetilde{A_{p,q}}$ case.
%A recent development \cite{Q24} that generalizes results in \cite{KQ2} to the punctured case,
%that essentially covers the case for type $\widetilde{D}$ (cf. Appendix~C there) --
%the explicit statement will be stated in the forthcoming article \cite{QZ25}.

%=========================================================
%\paragraph{\textbf{Further directions}}
%
%As our original motivation is coming from arc complexes,
%it is natural to study the case when $\C=\C(\surf)$ is the cluster category of a marked surface $\surf$, in the sense of \cite{FST}.
%
%Another interesting direction is to investigate the induced structure of $\xevo$-evolution flows/$\xevo$-foliations on
%the corresponding scattering diagrams and $g$-vector fans.
%Indeed, the $\sink$-evolution triangles are originally studied with motivation coming from indices/$g$-vectors
%and there is a close relationship between cluster complexes and scattering diagrams/$g$-vector fans.
%
%Finally, as cluster complexes are originally from cluster algebras, one would wonder about the decategorification problems,
%i.e. the cluster algebraic interpretation of $\xevo$-evolution flows and/or $\xevo$-foliations.

%=========================================================
\subsection*{Acknowledgments}
%=========================================================
The appexdix of this paper is the bachelor thesis of Tl under the supervision of Qy
(which will be submitted separately).
We would like to thank He Ping for proofreading the early versions and pointing out references.
Also thanks to Fu Changjian for inspiring conversations.
Tl would like to thank Qy for his patient guidance and thoughtful insights, which greatly encouraged him to complete this paper.

%=========================================================
%=========================================================
\section{Cluster complexes}\label{sec:cc}
%=========================================================
%=========================================================

%=========================================================
\subsection{Preliminaries}\label{sec:pre}\
%=========================================================

For simplicity, we assume $\k$ is an algebraically-closed field.
All categories considered will be $\k$-linear.
The triangulated shift functor will be denoted by $[1]$.

Let $\C$ be a triangulated category.
A functor $\Serre:\C\to\C$ is called a \emph{Serre functor} if there exists a functorial isomorphism
\[
    \Hom(Y,Z)\cong D\Hom(Z,\Serre Y),
\]
for any objects $Y,Z$ in $\C$.
The triangulated category $\C$ is called \emph{$2$-Calabi-Yau}, if it admits a Serre functor $\Serre$ such that $\Serre\cong[2]$.

Assume $\hua{Y} \subset \C$ be a full subcategory of $\C$, then its left and right perpendicular categories are the full subcategories with following indecomposables.
\begin{gather*}
    \hua{Y}^\perp=\{Z\in\C:\Hom_{\C}(Y,Z)=0~\forall Y\in \hua{Y}\},\\
    {^\perp}\hua{Y}=\{Z\in\C:\Hom_{\C}(Z,Y)=0~\forall Y\in \hua{Y}\}.
\end{gather*}

The \emph{right $\hua{Y}$-approximation} of an object $C\in\C$ is a morphism $f: Y\to C$ together with an object $Y\in \hua{Y}$ such that $\Hom(Y',f)$ is an epimorphism for any object $Y'\in\hua{Y}$.
A \emph{minimal right $\hua{Y}$-approximation} is a right $\hua{Y}$-approximation that is also \emph{right minimal}, i.e. for each $g\in\End(Y)$ such that $f\circ g=f$, we have $g$ is an isomorphism. Dually, we can define the left and minimal left $\hua{Y}$-approximation.

Moreover, we say the full subcategory $\hua{Y}$ is \emph{contravariantly (resp. covariantly) finite} if every object in $\C$ admits a right(resp. left) $\hua{Y}$-approximation. If $\hua{Y}$ is both contravariantly and covariantly finite, then $\hua{Y}$ is called \emph{functorially finite}.

\begin{definition}
Let $\Y$ be a functorially finite subcategory of $\C$. For any object $C\in\C$, we denote its right minimal $\hua{Y}$-approximation by $\App{R}{\hua{Y}}{C}$ and its left $\hua{Y}$-approximation by $\App{L}{\hua{Y}}{C}$.
For simplicity, we will write $\App{?}{Y}{C}$ for $\App{?}{\Add(Y)}{C}$, where $?\in\{L,R\}$.
\end{definition}

%=========================================================
\subsection{Cluster tilting subcategories and mutation}\label{sec:CTS}\
%=========================================================

In this subsection, we recall the terminologies about cluster tilting theory from \cite{BMRRT,KR}.

\begin{definition}\cite[\S5.1]{KR}\label{def:ctc}
A subcategory $\ctc$ of $\C$ is called a \emph{partial cluster tilting subcategory} if the following conditions hold:
\begin{itemize}
    \item $\ctc$ is functorially finite;
    \item $\ctc$ is rigid, that is, $\Ext^1(\ctc,\ctc)=0$;
\end{itemize}
If in addition, $\ctc={^\perp}\ctc[1]=\ctc^{\perp}[-1]$ holds, then $\ctc$ is called \emph{cluster tilting subcategory}.
\end{definition}

\begin{remark}\label{rem:cts/o}
We will use the following alternative equivalent terminologies.
An \emph{cluster tilting set} $\cts=\{\vv_1,\ldots, \vv_n\}$ in $\C$ is an $\Ext$-configuration (cf.\cite[\S2]{BMRRT}), that is, a maximal collection of non-isomorphic indecomposable objects such that
\[
\Ext^1_{\C}(\vv_i,\vv_j)=0
\]
for all $1\le i,j\le n$.
The object $\cto=\bigoplus_{i=1}^n \vv_i$ is called \emph{a cluster tilting object} for a cluster tilting set $\{\vv_{1},\ldots,\vv_{n}\}$.

The partial cluster tilting set/object is a subset/direct summand of the cluster tilting one.
\end{remark}

\begin{definition}\label{def:mut}
Let $\cts=\{\vv_1,\ldots,\vv_n\}$ be a cluster tilting set of a 2-Calabi-Yau category $\C$.
The \emph{forward mutation $\mu_i^\sharp$} at the object $\vv_i$ is an operation that produces another cluster tilting set $\mu_i(\cts)$ by replacing $Y_i$ by
\[
    \vv_i^\sharp=\Cone(\vv_i\overset{f}{\to}\bigoplus_{j\ne i}\Irr_{\Add(\pcto{\vv})}(\vv_i,\vv_j)^*\otimes \vv_j),
\]
where $\Irr_{\Add(\pcto{\vv})}(\vv_i,\vv_j)$ is the space of irreducible maps from $\vv_i$ to $\vv_j$
in the additive category $\Add\pcto{\vv}$.
Similarly, the backward mutation $\mu_i^\flat$ replaces $\vv_i$ by
\[
    \vv_i^\flat=\Cone(\bigoplus_{j\ne i}\Irr_{\Add(\pcto{\vv})}(\vv_j,\vv_i)\otimes \vv_j\overset{g}{\to} \vv_i)[-1].
\]
\end{definition}

The 2-Calabi-Yau property implies that $\vv_i^\sharp=\vv_i^\flat$ and hence $\mu_i^\sharp=\mu_i^\flat$
(i.e., mutation is an involution).
Thus, we can drop $\sharp$ and $\flat$ in the notations.
The so-called \emph{unoriented cluster exchange graph} $\uCEG(\C)$ is the graph whose vertices are clusters and
whose edges are mutations. Such a graph is $n$-regular,
meaning that each vertex is $n$-valent.

%=========================================================
\subsection{Calabi-Yau reduction}\label{sec:Red}\
%=========================================================

We recall the Calabi-Yau reduction introduced in \cite[\S4]{IY}. Let $\C$ be a triangulated category with a Serre functor $\Serre$. Suppose that $\hua{\sink}\subset\hua{Y}$ are two subcategories of $\C$ with auto-equivalence $\Serre\circ[-2]$, such that:
\begin{itemize}
    \item $\hua{Y}$ is extension-closed, i.e. $\hua{Y}*\hua{Y}\subset\hua{Y}$,
    \item $(\hua{\sink}*\hua{Y}[1])\cap^\bot\hua{\sink}[1],~(\hua{Y}[-1]*\hua{\sink})\cap\hua{\sink}[-1]^\bot\subset\hua{Y}$,
\end{itemize}

Denote by $\hua{Y}/\hua{\sink}$ the quotient category
with respect to the ideals consisting of morphisms that factor through objects in $\hua{\sink}$.
Under the assumptions as above, \cite[Prop.~2.6]{IY} constructs a functor $\<1\>\colon\hua{Y}/\hua{\sink}\to\hua{Y}/\hua{\sink}$:
\begin{gather}\label{def:shift}
Y\xrightarrow{\alpha} \sink_Y\to Y\<1\>\to Y[1],
\end{gather}
where $\alpha$ is a minimal left $\hua{\sink}$-approximation of $Y$, such that
\begin{enumerate}
    \item $(\hua{Y}/\hua{\sink}, \<1\>)$ forms a triangulated category and the quotient
    $\Red_{\hua{\sink}}:\hua{Y}\to\hua{Y}/\hua{\sink}$ is a triangle functor,
    \item $\hua{Y}/\hua{\sink}$ has a Serre functor $\Serre_2\circ\< 2\>$. Moreover, $\hua{Y}/\hua{\sink}$ is $2$-Calabi-Yau if $\C$ is.
    \item The functor $\Red_{\hua{\sink}}:\hua{Y}\to\hua{Y}/\hua{\sink}$ induces a one-to-one correspondence between cluster tilting subcategories of $\C$ containing $\hua{\sink}$ and cluster tilting subcategories of $\hua{Y}/\hua{\sink}$.
\end{enumerate}
In particular, fixing an rigid indecomposable $\sink$ in $\C$ we will take
$\hua{Y}={^\perp}(\source)$, $\hua{\sink}=\Add(\sink)$
and denote by $\C\backslash \sink={^\perp}(\source)/\sink$ the quotient category. Hence we have the reduction functor for $\sink$:
\begin{gather}\label{eq:redS}
\Red_{\sink}:{^\perp}(\source)\to \C\backslash\sink.
\end{gather}
Noticing that ${^\perp}(\sink[2])=(\sink)^\perp$ by the 2-Calabi-Yau condition,
we have another reduction functor for $\source$:
\begin{gather}\label{eq:redS1}
\Red_{\source}:(\sink){^\perp}\to \C\backslash \source.
\end{gather}

%=========================================================
\subsection{Cluster complexes}\label{sec:ccx}\
%=========================================================

Let $\C$ be the 2-Calabi-Yau category with cluster tilting sets.

\begin{definition}\label{def:CCx}
The \emph{cluster complex} $\CCx(\C)$ is a clique complex for the compatibility relation on the ground set whose simplices are partial cluster tilting sets, and whose maximal simplices are cluster tilting sets.

This is an equivalent analogue of \cite[\S7]{FST}.
\end{definition}

In the case when $\C=\C(Q)$ for an acyclic quiver $Q$, we will write $\CCx(Q)$ for $\CCx(\C(Q))$ for short. Similarly for $\CCx(\plpqr)$. (cf. \Cref{sec:qc})

\begin{remark}\label{rem:dual}
Note that the top-cells of the cluster complex $\CCx(\C)$ are precisely the cluster tilting sets and the codimension-1 cells are precisely the mutations. As a result, the 1-skeleton of the dual complex of $\CCx$ is precisely the cluster exchange graph introduced in \Cref{def:mut}.
For instance, \Cref{fig:A2} shows the duality between $\CCx(A_2)$ and $\uCEG(A_2)$,
where we use the geometric model (the pentagon) to demonstrate,
that arcs/diagonals are indecomposable objects in $\C(A_2)$ and triangulations are clusters.

\begin{figure}[hb]\centering\makebox[\textwidth][c]{
  \includegraphics[width=12cm]{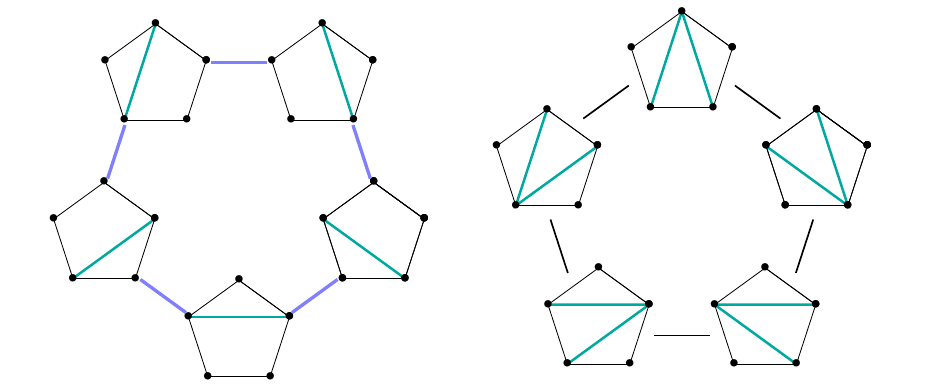}}\vskip -.2cm
\caption{The cluster complex $\CCx(A_2)$ VS the exchange graph $\uCEG(A_2)$} \label{fig:A2}
\end{figure}

\end{remark}

\begin{example}\label{ex:AR}
When $\C=\CQ$, the connected component
\[
    \ar(\tau^\ZZ\Proj \k Q) \cong \ZZ Q^{\on{op}}
\]
in the AR-quiver of $\CQ$ can be naturally embedded into the cluster complex $\CCx(\CQ)$, where the indecomposables are 0-cells and multi-arrows between two indecomposables become a 1-cell connecting the corresponding 0-cells.

For instance:
\begin{itemize}
  \item Let $Q$ be a $A_1$ quiver. The AR-quiver of $\C(\DynA{1})$ consists of two distinct points. Hence the cluster complex $\CCx(\DynA{1})=\SS^0$.

  \item Let $Q$ be a $A_3$ quiver with a bivalent sink vertex $0$ and two source vertices $1,2$. The left picture of \Cref{fig:A3-AR-CCx} shows the AR-quiver of $\C(A_3)$.
      The right picture of which shows the cluster complex $\C(A_3)$ with
      the dual graph (in black) being the cluster exchange graph $\uCEG(A_3)$
      (which is an associahedron).

  \item Let $Q$ be a $D_4$ quiver with a trivalent sink vertex $0$ and three source vertices $1,2,3$. There is a $C_3$-symmetry on $Q$ permutating the source vertices. \Cref{fig:D4-AR-CC} shows the AR-quiver of $\C(D_4)$ and the cluster complex $\C(D_4)$.
\end{itemize}

\begin{figure}[hb]\centering\makebox[\textwidth][c]{
  \includegraphics[width=15cm]{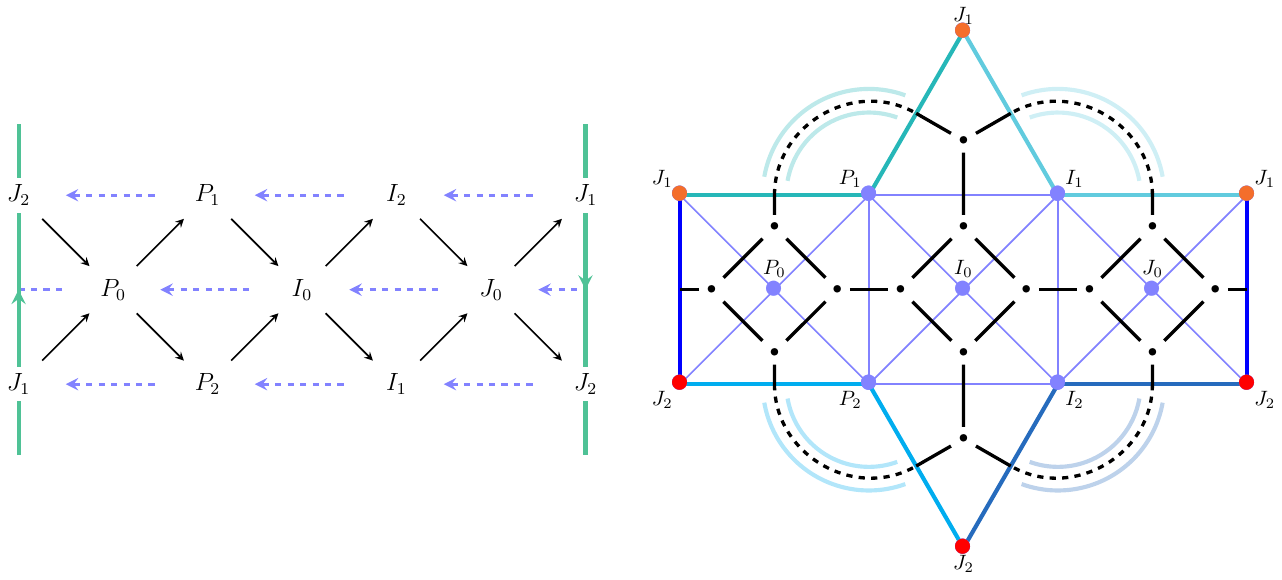}
}
\caption{The AR-quiver and cluster complex $\C(A_3)$} \label{fig:A3-AR-CCx}
\end{figure}

\begin{figure}[ht]\centering\makebox[\textwidth][c]{
  \includegraphics[width=15cm]{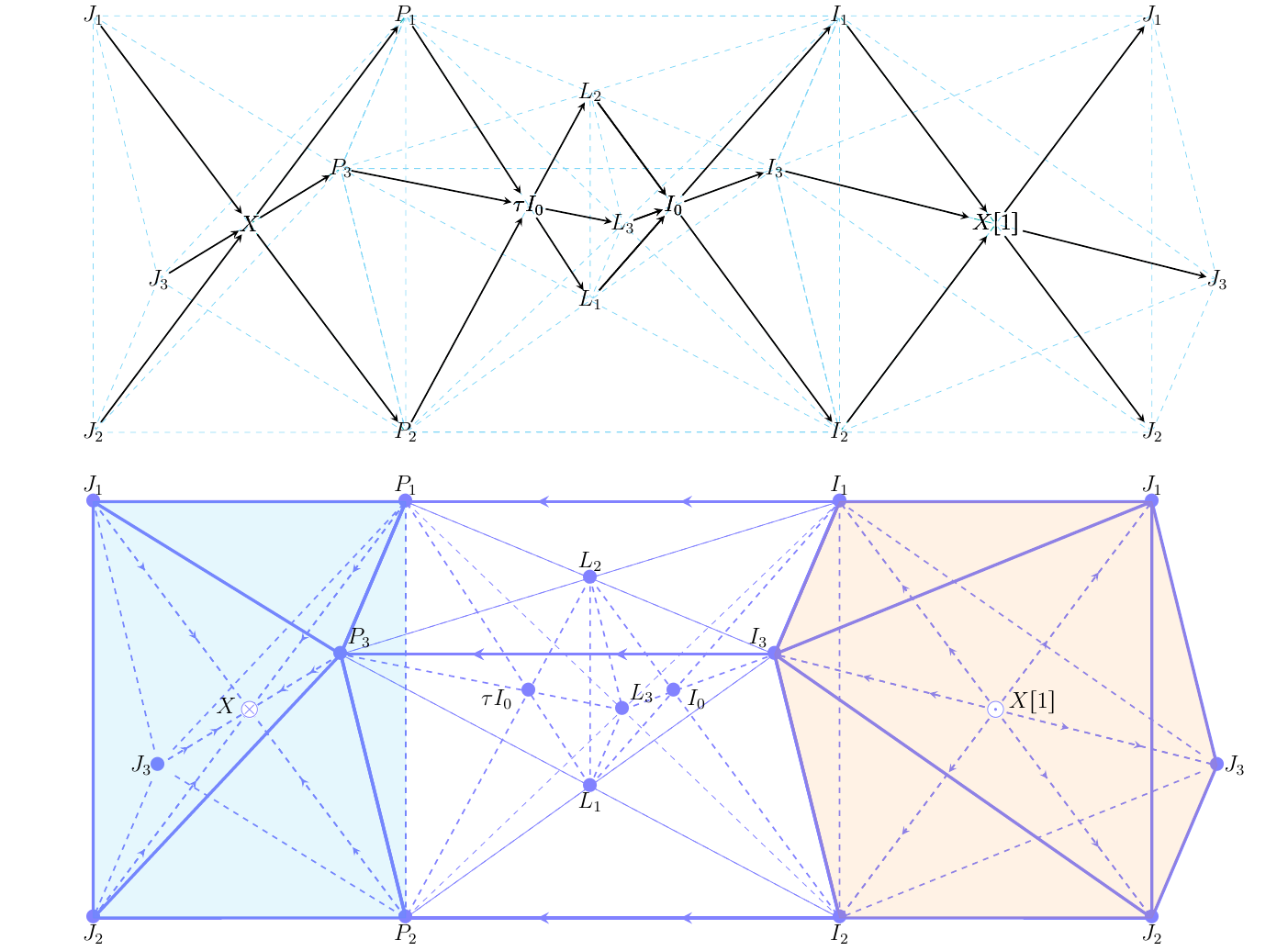}\quad}
\caption{The AR-quiver and part of cluster complex $\C(D_4)$} \label{fig:D4-AR-CC}
\end{figure}
\end{example}

Let $K\subset\CCx$ be a $k$-simplex in a complex $\CCx$. We recall some of the standard terminologies from the theory of simplicial complexes.
\begin{itemize}
\item The \emph{interior} of a $k$-simplex $[x_0,\ldots,x_k]$ is
\begin{gather}\label{def:interior}
[x_0,\ldots,x_k]^\circ=\{\sum_{i=0}^n c_i x_i: c_i\ne0,~~\forall 0\le i\le k \}.
\end{gather}
Note that it is the interior in the usual sense except for $k=0$, where $[x_0]^\circ=\{x_0\}=[x_0]$.

\item The \emph{$k$-skeleton} $\CCx^{\le k}$ of $\CCx$ consists of all interior of simplices $\sigma$ with dimension at most $k$:
\[
    \CCx^{\le k}=\{\sigma^\circ\subset\CCx:\dim\sigma\le k\},
\]
and dually
\[
\CCx^{\ge k}=\{\sigma^\circ\subset\CCx:\dim\sigma\ge k\}.
\]

\item The \emph{star $\Star(K)$} of $K$ in $\CCx$ consists of all simplices that containing $K$:
\[
    \Star(K)=\{\sigma\subset\CCx:K\subset\sigma\}.
\]

\item The \emph{link $\link(K)$} of $K$ in $\CCx$ consists of all faces of simplices in the star that do not intersect $K$:
\[
    \link(K)=\{\sigma\subset\Star(K):\sigma\cap K=\emptyset\}.
\]
\item The \emph{join} of two simplicial complexes $\CCx_1$ and $\CCx_2$ is the simplicial complex consisting of the join of their simplices:
\[
  \CCx_1*\CCx_2=\{K_1*K_2:K_i\subset\CCx_i,i=1,2\},
\]
where the join of simplices are $[v_1,\ldots,v_k]*[w_1,\ldots,w_l]=[v_1,\ldots,v_k,w_1,\ldots,w_l]$.
\end{itemize}

We will usually take $K$ be a vertex $\sink$ in a cluster complex $\CCx(\C)$,
which is a rigid indecomposable in $\C$.
Moreover, we know that
\begin{itemize}
    \item $\Star(\sink)$ consists of all partial cluster tilting sets perpendicular to $\source$, that is,
    \[
        \Star(\sink)=\{ \pcts{\vv}:\pcts{\vv} \subset{^\perp}(\sink[1]) \}.
    \]
    \item $\link(\sink)$ consists of all partial cluster tilting sets perpendicular to $\source$ but not containing $\sink$, that is,
    \[
        \link(\sink)=\{ \pcts{\vv}:\sink\notin\pcts{\vv} \subset{^\perp}(\sink[1]) \}.
    \]
\end{itemize}
By Calabi-Yau reduction, $\C\backslash\sink$ is still 2-Calabi-Yau and the (partial) clusters are preserved by the reduction functor. In other words, we can naturally identify:
\begin{gather}\label{eq:red}
\begin{cases}
\CCx(\C\backslash\sink)\simeq\link(\sink),\\
\CCx(\C\backslash\source)\simeq\link(\source).
\end{cases}
\end{gather}
Furthermore, the join of two cluster complexes can be realized as the cluster complex of the 2-Calabi-Yau category $\C_1\sqcup\C_2$:
\begin{gather}\label{eq:union}
    \CCx(\C_1\sqcup\C_2)=\CCx(\C_1)*\CCx(\C_2).
\end{gather}
For instance, the join with $\SS^0$ (two disjoint vertices) is the equivalent to the suspension.
So we have
\[
    \Sigma\CCx(\C)=\SS^0*\CCx(\C)=\CCx(\C_{A_1})*\CCx(\C)=\CCx(\C_{A_1}\sqcup\C).
\]

\begin{example}\label{ex:ccx-D4}
Continue with the $D_4$ example above.
Fix a rigid object $\sink=P_0$.
In the lower picture of \Cref{fig:D4-AR-CC},
the blue and orange octahedrons are the stars $\Star(\sink)$ and $\Star(\source)$;
the surface (=8 faces) of which are $\link(\sink)$ and $\link(\source)$,
respectively.

Moreover, the interval $[\sink,\source]$ (cf. \Cref{sec:int}) corresponds to the white prism in $\CCx(D_4)$ showing in the lower picture of \Cref{fig:D4-AR-CC}, which consists of 19 top-cells (tetrahedrons):
\begin{itemize}
  \item 4 of which are $C_3$-invariant showing in \Cref{fig:D4-evo1}.
  \item 3 of which with edges $P_iI_i$, for $1\le i\le 3$, e.g. the green one in \Cref{fig:D4-evo2}.
  \item 6 of which with edges $P_iP_j$ or $I_iI_j$, for $1\le i,j \le 3$, e.g. the blue one in \Cref{fig:D4-evo2}.
  \item 6 of which with edges $P_0P_i$ or $I_0I_i$, for $1\le i\le 3$, e.g. the orange one in \Cref{fig:D4-evo2}.
\end{itemize}
\end{example}

\begin{figure}[thb]\centering\makebox[\textwidth][c]{
  \includegraphics[height=7cm]{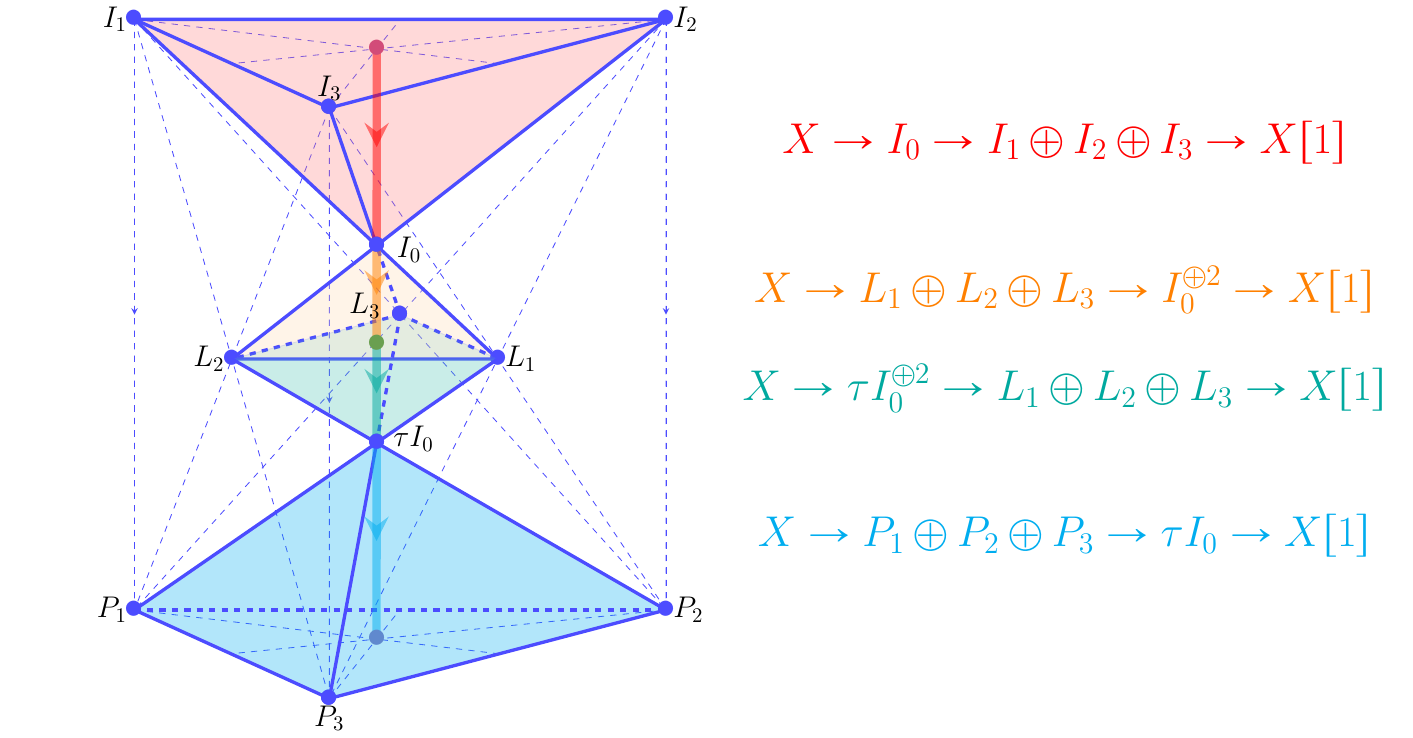}
}
\caption{Four tops cells with $\sink$-evolving triangles in $\CCx(D_4)$}
\label{fig:D4-evo1}
\end{figure}

\begin{figure}[thb]\centering\makebox[\textwidth][c]{
  \includegraphics[height=7cm]{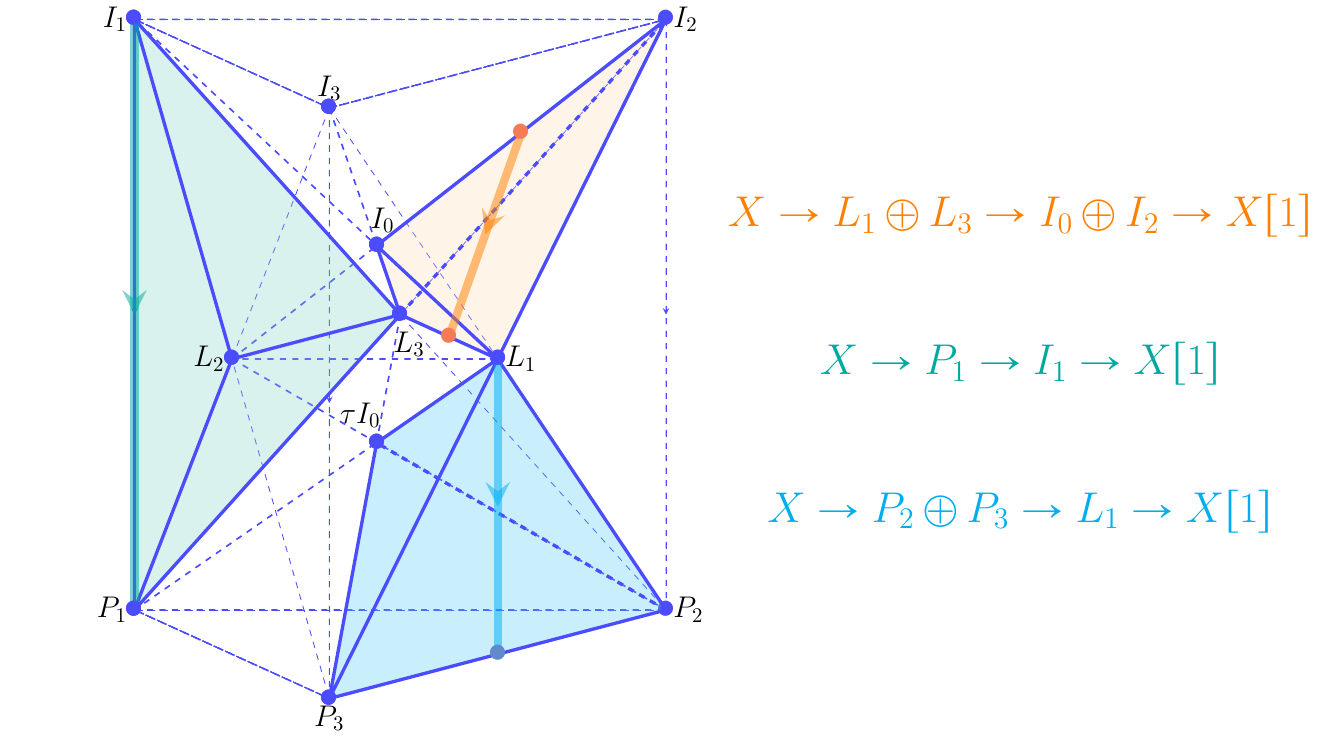}
}
\caption{Three tops cells with $\sink$-evolving triangles in $\CCx(D_4)$}
\label{fig:D4-evo2}
\end{figure}

%=========================================================
%=========================================================
\section{\texorpdfstring{$\xevo$-evolution flows}{X-evolution flows}}\label{sec:evoCC}
%=========================================================
%=========================================================

Throughout, we fix a 2-Calabi-Yau triangulated category $\C$, admitting a cluster tilting object. Moreover, we impose a mild condition on $\C$ for simplicity, that
\begin{itemize}
  \item[$\bigstar$] Any rigid objects in $\C$ is a partial cluster.
\end{itemize}
So basically, we require functorial finiteness for rigid indecomposable objects.
For instance, $\Hom$-finite Krull-Schmidt condition implies Condition~$\bigstar$,
cf. \cite[Prop.~4.2]{AS}.
In particular, if $\C$ is the cluster category of a Jacobi-finite non-degenerate quiver with potential (say an acyclic quiver with zero potential) or quiver with potential from a triangulated marked surface, then Condition~$\bigstar$ holds.
%=========================================================
\subsection{$\sink$-evolving triangles}\label{sec:evo}\
%=========================================================
Fix a rigid indecomposable $\sink$ in $\C$.
\begin{convention}\label{conv}
Throughout, we will use the following notations of partial cluster objects/subcategories/sets:
\begin{itemize}
    \item $\vv_i$ denotes a rigid indecomposable in $\C$, i.e. a vertex in $\CCx(\C)$.
    \item $\pcts{\vv}=\{\vv_i\}_{i=0}^k$ denotes a partial cluster tilting set in $\C$, which is also a simplex of $\CCx(\C)$, cf. \Cref{def:CCx}.
    \item $\pcto{\vv}=\oplus_{i=0}^k \vv_i$ denotes a partial cluster tilting objects in $\C$.
    \item $\pctc{\vv}=\Add(\pcto{\vv})$ denotes a partial cluster tilting subcategories in $\C$.
    \item $\ccpt{\vv}=\sum_{i=0}^k \alpha_i \vv_i$ denotes a point of $\CCx(\C)$.
    \item $\cc(\pcto{\vv})=\pcts{\vv}$ denotes the cell corresponding to $\pcto{\vv}$ in $\CCx(\C)$.
\end{itemize}
and with a slight abuse of terminology, we call all of them \emph{partial clusters}.

Any rigid object $V=\oplus_{i=0}^k V_i^{\oplus d_i}$ can be realized as a point
\[
    \real{V}=\frac{\sum_{i=0}^k d_i V_i}{\sum_{i=0}^k d_i}
\]
in $\CCx(\C)$, where $V_i$ are non-isomorphic indecomposable objects.
\end{convention}

Next, we introduce our core notion, the $\sink$-evolving triangles, which is closely related to the notion of indices in \cite{DK},
cf. \cite{IY} and \cite{Pla}.

\begin{lemdef}\label{def:evo-tri}
A triangle
\begin{gather}\label{eq:evo-tri}
    \sink \xrightarrow{g} \ww \to \uu \xrightarrow{f} \source
\end{gather}
is called a \emph{pre-$\sink$-evolving triangle}, if it satisfies the following equivalent conditions:
\begin{enumerate}
    \item $\uu$ is a partial cluster and $f$ is a right $\Add(\uu)$-approximation of $\source$;
    \item $\Ext^1(\uu,\ww)=0$;
    \item $\ww$ is a partial cluster and $g$ is a left $\Add(\ww)$-approximation of $\sink$.
\end{enumerate}
It is an \emph{$\sink$-evolving triangle}, if it satisfies the following equivalent conditions in addition:
\begin{enumerate}
    \item $f$ is right minimal;
    \item $\Add(\ww)\cap\Add(\uu)=0$;
    \item $g$ is left minimal.
\end{enumerate}
In particular, $\Add(\ww\oplus \uu)$ is a partial cluster for $\ww, \uu$ in an $\sink$-evolving triangle.
\end{lemdef}

\begin{proof}
The first set of equivalences is the 2-Calabi-Yau version of the Wakamatsu's Lemma \cite[Lem.~2.1.13]{GT}.
We only show that $1^\circ$ is equivalent to $2^\circ$. The equivalence of $2^\circ$ and $3^\circ$ is similar.

Applying $\Hom(\uu,-)$ to the triangle $\eqref{eq:evo-tri}$, we have:
\begin{equation}\label{les1}
\begin{aligned}
\Hom(\uu,\uu)\xrightarrow{\Hom(\uu,f)}&\Hom(\uu,\source)\to\Hom(\uu,\ww[1])\to\Hom(\uu,\uu[1])\\
    \to&\Hom(\uu,\sink[2])\xrightarrow{\Hom(\uu,g[2])}\Hom(\uu,\ww[2]).
\end{aligned}
\end{equation}

$1^\circ \Rightarrow 2^\circ$:
Note that the condition $1^\circ$ is equivalent to $\Hom(\uu,\uu[1])=0$ and $\Hom(\uu,f)$ is surjective, which implies that $\Hom(\uu,\ww[1])=0$.

$2^\circ \Rightarrow 1^\circ$:
$\Ext^1(\uu,\ww)=0$ implies that $\Hom(\uu,f)$ is surjective.
Applying $\Hom(\sink,-)$ and $\Hom(-,\ww)$ to the triangle $\eqref{eq:evo-tri}$ we have
\[
\begin{cases}
    \Hom(\sink,\ww)\to\Hom(\sink,\uu)\to\Hom(\sink,\source)=0.\\
    \Hom(\ww,\ww)\to\Hom(\sink,\ww)\to\Hom(\uu[-1],\ww)=0.
\end{cases}
\]

Combining with the following communicative diagram:
\[
\begin{tikzcd}
& {\Hom(\ww,\ww)} \arrow[r] \arrow[d,two heads] & {\Hom(\ww,\uu)} \arrow[d,"{\Hom(g,\uu)}"] \\
& {\Hom(\sink,\ww)} \arrow[r,two heads] & {\Hom(\sink,\uu)},
\end{tikzcd}
\]
we conclude that $\Hom(g,\uu)$ is surjective.
By the 2-Calabi-Yau condition, the surjectivity of $\Hom(g,\uu)$ is equivalent to the injectivity of $\Hom(\uu,g[2])$.
Back to $\eqref{les1}$, $\Hom(\uu,\ww[1])=0$ implies that $\Hom(\uu,\uu[1])=0$.

The second set of equivalences was proofed in \cite[Lem.~3.4 and Lem~3.5]{Pla}.
\end{proof}

There are two \emph{trivial} $\sink$-evolving triangles:
\begin{gather}\label{eq:trivial}\begin{cases}
    \sink \to \sink \to 0 \to \source,\\
    \sink \to 0 \to \source \to \source.
\end{cases}\end{gather}
We will call the other $\sink$-evolving triangles non-trivial (where $\ww\ne0\ne \uu$).

\begin{example}
In the cluster category $\C( \widetilde{A_{1,1} } )=\D^b(\pl)\big/\Serre_2$
with rigid object $\sink=\mathcal{O}(1)$, all the $\sink$-evolving triangles are
\iffalse
\begin{equation}\label{eq:evo-A1,1}
\begin{cases}
  \mathcal{O}(1) \to \mathcal{O}(m)^{\oplus m} \to \mathcal{O}(m+1)^{\oplus (m-1)}
    \to  \mathcal{O}(1)[1], & m\ge1\\
  \mathcal{O}(1) \to \mathcal{O}(-m)^{\oplus m}[1] \to \mathcal{O}(1-m)^{\oplus (m+1)}[1]
    \to  \mathcal{O}(1)[1], & m\ge0.
\end{cases}
\end{equation}
\fi
\[
\mathcal{O}(1) \to \mathcal{O}(m)^{\oplus |m|} \to \mathcal{O}(m+1)^{\oplus |m-1|}
\to  \mathcal{O}(1)[1], \quad m\in\ZZ.
\]
Note that
$$\mathcal{O}(1) \to \mathcal{O}_x \to \mathcal{O}[1] \to \mathcal{O}(1)[1]$$
is not an $\sink$-evolving triangle.

Note that in this case,
by identifying $\mathcal{O}(m)$ with integer $m\in\ZZ\subset\RR$ for any $m\in\ZZ$,
the cluster complex $\CCx(\widetilde{A_{1,1}})$ can be identified with $\RR$.
\end{example}

\begin{lemma}\label{pp:complete}
Let \eqref{eq:evo-tri} be an $\sink$-evolving triangle and $\vv$ an object.
Then
\begin{enumerate}[label=(\alph*)]
  \item $\Ext^1(\vv,\ww)=0$ implies that $f$ is still a minimal right $\Add(\vv\oplus\uu)$-approximation.
  \item $\Ext^1(\uu,\vv)=0$ implies that $g$ is still a minimal left $\Add(\ww\oplus\vv)$-approximation.
\end{enumerate}
\end{lemma}

\begin{proof}
For $(a)$, applying $\Hom(\vv,-)$ to \eqref{eq:evo-tri}, we obtain
\begin{equation}\label{les2}
\Hom(\vv,Z)\xrightarrow{\Hom(\vv,f)}\Hom(\vv,\source)\to\Hom(\vv,\ww[1])=0.
\end{equation}
Since \eqref{eq:evo-tri} is $\sink$-evolving, $\Hom(\uu,f)$ is surjective and hence $\Hom(\vv\oplus\uu,f)$ is surjective, that is, $f$ is a right $\Add(\vv\oplus\uu)$-approximation.
The minimality is also inherited (using the second criterion).
Dually, we have $(b)$.
\end{proof}

\begin{definition}
An $\sink$-evolving triangle \eqref{eq:evo-tri} is called a \emph{downward $\sink$-evolving triangle} for $\pcts{\vv}$, if $f$ is the right minimal $\pctc{\vv}$-approximation of $\source$, that is, $\uu=\App{R}{\pctc{\vv}}{\source}$.

Dually, \eqref{eq:evo-tri} is called a \emph{upward $\sink$-evolving triangle} for $\pcts{\vv}$, if $g$ is the left minimal $\pctc{\vv}$-approximation of $\sink$, that is, $\ww=\App{L}{\pctc{\vv}}{\sink}$.
\end{definition}

\begin{remark}
By \Cref{pp:complete}, an $\sink$-evolving triangle \eqref{eq:evo-tri} is the downward/upward $\sink$-evolving triangle
for any partial cluster $\pcts{\vv}$ containing $\ww\oplus \uu$.

In particular, the downward and upward $\sink$-evolving triangles of a partial cluster $\pcts{\vv}$ may coincide,
e.g. when $\pcts{\vv}$ is a cluster.
In such a case, we call it the $\sink$-evolving triangle for $\pcts{\vv}$.
\end{remark}

Note that by the functorially finiteness, the downward/upward $\sink$-evolving triangle for any $\pcts{\vv}$ exists.
By the minimality, it is moreover unique. But two different partial clusters may share the same downward/upward $\sink$-evolving triangle.

The first property of downward/upward $\sink$-evolving triangles is the following.
\begin{lemma}\label{lem:last}
Let \eqref{eq:evo-tri} be a downward or upward $\sink$-evolving triangle for $\pcts{\vv}$.
Then $\Ext^1(\pcto{\vv},\ww)=0$ or $\Ext^1(\uu,\pcto{\vv})=0$, respectively.
In particular, $\ww\oplus\pcto{\vv}\oplus \uu$ is a partial cluster.
\end{lemma}
\begin{proof}
We only deal the downward case.
Applying $\Hom(\pcto{\vv}, -)$ to \eqref{eq:evo-tri}, we have
\[
    \Hom(\pcto{\vv},\uu)\xrightarrow{\Hom(\vv,f)}\Hom(\pcto{\vv},\source)\to\Hom(\pcto{\vv},\ww[1])
    \to \Hom(\pcto{\vv}, \uu[1])=0.
\]
The last term is zero as $\uu$ is in the partial cluster $\pctc{\vv}$;
$\Hom(\vv,f)$ is surjective since $f$ is the right minimal $\pctc{\vv}$-approximation of $\source$.
Hence $\Ext^1(\pcto{\vv},\ww)=0$ as required.
\end{proof}

\iffalse
\begin{remark}\label{rem:formula}
For any $\pcto{\vv}=\oplus_{i=0}^k \vv_i$. By functorial finiteness, there exists both downward and upward $\sink$-evolution triangles for $\pctc{Y}$. When $\sink$ is a brick, They have the following explicit form:
\begin{gather}\label{eq:irr1}
    \sink\to X \to \oplus_{i=0}^k \vv_i\otimes\Irr_{\pctc{Y}}\Ext^1(\vv_i,\sink)\xrightarrow{f} \source,\\
    \sink\xrightarrow{g}\oplus_{i=0}^k \vv_i\otimes\Irr_{\pctc{Y}}\Ext^1(\source,\vv_i)\to Z\to \source. \label{eq:irr2}
\end{gather}
However, when $\sink$ is not a brick as in \Cref{ex:no}, we do not have such formulae.
\end{remark}
\fi

\begin{example}\label{ex:no}

Consider the geometric model of $\EucA{1}{3}$ (cf. \Cref{sec:GeoApq}) with a chosen triangulation corresponding to a cluster, showing in the left of \Cref{fig:A13-tube}.

\begin{figure}[th]\centering\makebox[\textwidth][c]{
  \includegraphics[width=15cm]{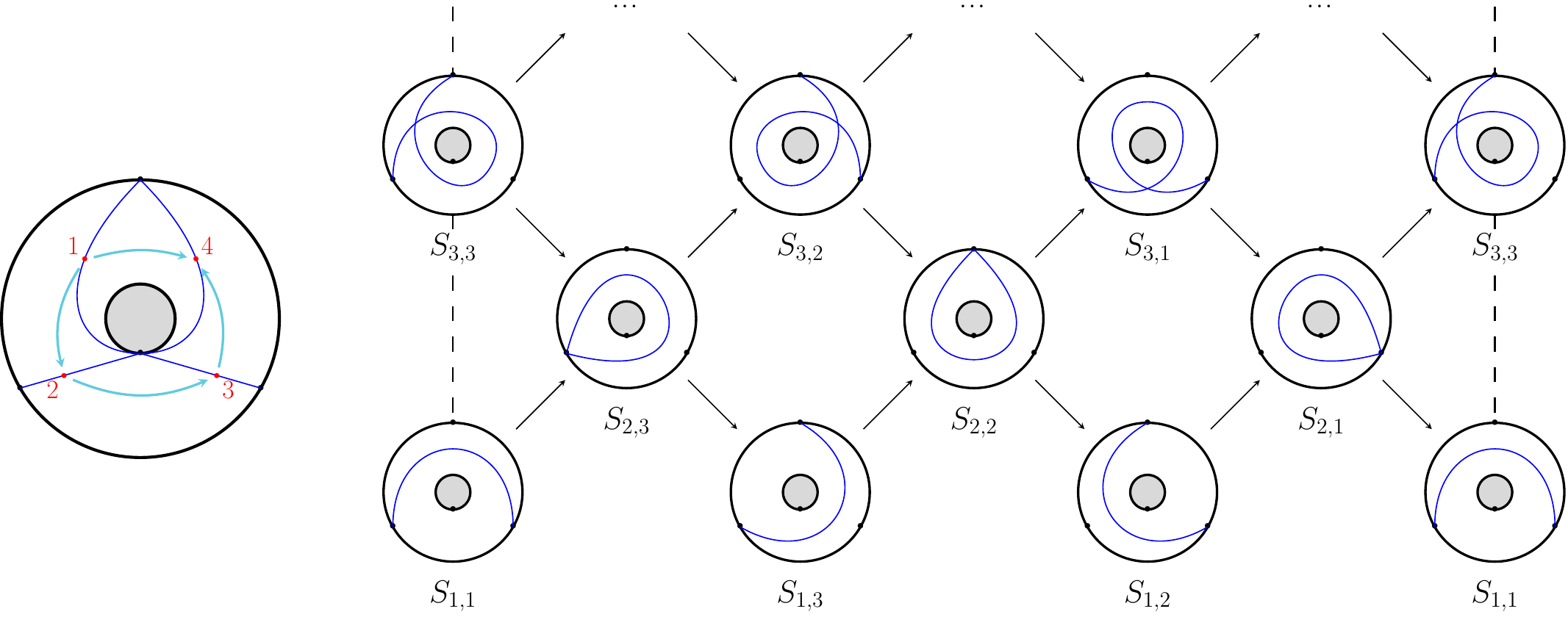}
}
\caption{A triangulation of $\EucA{1}{3}$ (left) and the AR-quiver of subcategory $\tube_{\infty}(\EucA{1}{3})$ (right)}
\label{fig:A13-tube}
\end{figure}

The the AR-quiver of the subcategory/tube $\tube_{0}(\EucA{1}{3})$ in $\C(\EucA{1}{3})$ is demonstrated in \Cref{fig:A13-tube}.
Then the upward/downward $\sink=S_{2,3}$-evolving triangle for $S_{2,2}$ are
\begin{gather*}
    S_{2,3}\to S_{2,2} \to S_{1,2}^{\oplus 2}\to S_{2,1},\\
    S_{2,3}\to  S_{1,3} ^{\oplus 2}\to S_{2,2}\to S_{2,1},
\end{gather*}
respectively.
The first one is also the downward $S_{2,3}$-evolving triangle for $S_{1,2}$.
Moreover, the downward $\sink=S_{1,2}$-evolving triangle for $S_{2,3}$ is
\begin{gather*}
    S_{1,2}\to S_{3,3} \to S_{2,3}\to S_{1,3}.
\end{gather*}
Note that $S_{1,3} ^{\oplus2}= \App{L}{ S_{1,3} }{S_{2,3}} $
but $S_{2,3}=\App{R}{ S_{2,3} }{ S_{1,3}  }$.
\end{example}

%=========================================================
\subsection{$\sink$-evolution flows}\label{sec:evo-flow}\
%=========================================================

For any cluster complex $\CCx(\C)$, we will introduce the \emph{(downward) $\xevo$-evolution flow} on $\CCx(\C)$ for an arbitrary point $\xevo\in\CCx(\C)$.

Firstly, consider the special case when $\xevo=\sink$ is a $0$-cell.

\begin{definition}\label{def:flow}
Let \eqref{eq:evo-tri} be a downward $\sink$-evolving triangle for a cell $\pcts{\vv}$.
we define the \emph{(downward) $\sink$-evolution flow} for any $\ccpt{\vv}\in\pcts{\vv}^\circ$ (cf. \eqref{def:interior}) as follows:
\begin{gather}\label{eq:dflow}
    \dflow{\sink}(\ccpt{\vv})=
    \begin{cases}
      \sink-\ccpt{\vv}, & \mbox{if $\ww=\sink$ and $\uu=0$}, \\
      \real{\ww}-\real{\uu}, & \mbox{if $\ww\ne0$ and $\uu\ne0$}, \\
      \ccpt{\vv}-\source, & \mbox{if $\ww=0$ and $\uu=\source$}.
    \end{cases}
\end{gather}
Dually, if \eqref{eq:evo-tri} is the upward $\sink$-evolving triangle for $\pcts{\vv}$,
define the \emph{inverse (upward) $\sink$-evolution flow} as follows:
\begin{gather}\label{eq:uflow}
    \uflow{\sink}(\ccpt{\vv})=
    \begin{cases}
      \ccpt{\vv}-\sink, & \mbox{if $\ww=\sink$ and $\uu=0$}, \\
      \real{\uu}-\real{\ww}, & \mbox{if $\ww\ne0$ and $\uu\ne0$}, \\
      \source-\ccpt{\vv}, & \mbox{if $\ww=0$ and $\uu=\source$}.
    \end{cases}
\end{gather}
\end{definition}

From the results in \Cref{sec:evo}, we have the following.
\begin{enumerate}
\item Note that as $\ww\oplus\pcto{\vv}\oplus \uu$ is still a partial cluster by \Cref{lem:last}, $\real{\ww}$ and $\real{\uu}$ are points in $\cc(\ww\oplus\pcto{\vv}\oplus\uu)$ (which contains the $k$-cell $\pcts{\vv}$).
    Thus the \emph{local $\sink$-evolution flows}
    \begin{gather}
        \label{eq:down-flow}
        \Dflow{\sink}(\ccpt{\vv},t)=\ccpt{\vv}+t\dflow{\sink}(\ccpt{\vv}),\\
    \label{eq:up-flow}
        \Uflow{\sink}(\ccpt{\vv},t)=\ccpt{\vv}+t\uflow{\sink}(\ccpt{\vv})
    \end{gather}
    are well-defined for small $t>0$.
\item Recall that the downward and upward $\sink$-evolution triangles coincide for a top-cell/cluster.
    Therefore, $\dflow{\sink}=-\uflow{\sink}$ in the interior of any top-cells.
\item There are only two singularities $\sink,\source$ of the $\sink$-evolution flow in the sense that $\uflow{\sink}=0=\dflow{\sink}$.
\end{enumerate}

\begin{definition}
Local $\sink$-evolution flows induce \emph{local $\sink$-leave}s:
\begin{itemize}
  \item a parallel linear foliation in any top-cell, that is not in $\Star(\sink)\bigcup\Star(\source)$, decomposing such a cell into disjoint one-dimensional intervals (non-degenerate local $\sink$-leaves) and points (degenerate local $\sink$-leaves).
  \item a radial linear foliation in any top-cell in $\Star(\sink)\bigcup\Star(\source)$,
  decomposing such a cell into one-dimensional intervals (non-degenerate local $\sink$-leaves), sharing a common endpoint $\sink$ or $\source$.
\end{itemize}
\end{definition}

%=========================================================
\subsection{A family of $\xevo$-evolution flows}\label{sec:evo-flow2}\
%=========================================================

Now we consider the general case.
\begin{definition}\label{def:Gflow}
Let $\xevo=\sum_{i=0}^k c_i \sink_i$ be an arbitrary point in $\CCx(\C)$, define the direction of $\xevo$-evolution flow to be the linear combination
\begin{gather}\label{eq:general}
\begin{cases}
  \dflow{\xevo}=\sum_{i=0}^k c_i \dflow{\sink_i},\\
  \uflow{\xevo}=\sum_{i=0}^k c_i \uflow{\sink_i}.
\end{cases}
\end{gather}
\end{definition}

\begin{lemma}\label{lem:well-defined}
The evolution flows are well-defined for any $\xevo$.
\end{lemma}

\begin{proof}
For any $\pcts{\vv}$, the downward partition of the index set induced by $\pcts{\vv}$
\begin{gather}\label{eq:d-part}
\{0,\ldots,k\}=I_0\sqcup I_1\sqcup J
\end{gather}
is determined by
\begin{itemize}
  \item $I_0=I_0(\pcts{\vv})\subset\{0,\ldots,k\}$ consists of $i_0$ such that the downward $\sink_{i_0}$-evolving triangle for $\pcts{\vv}$ is trivial with $\ww_{i_0}=\sink_{i_0}$ and $\uu_{i_0}=0$.
  \item $I_1=I_1(\pcts{\vv})\subset\{0,\ldots,k\}$ consists of $i_1$ such that the downward $\sink_{i_1}$-evolving triangle for $\pcts{\vv}$ is trivial with $\ww_{i_1}=0$ and $\uu_{i_1}=\sink_{i_1}[1]$.
  \item $J=\{0,\ldots,k\}\setminus(I_0\cup I_1)$ so that the downward $\sink_j$-evolving triangles for $\pcts{\vv}$
    \begin{gather}\label{eq:S_j-evo}
        \sink_j\xrightarrow{g_j} \ww_j\to \uu_j\xrightarrow{f_j} \sink_j[1]
    \end{gather}
     is nontrivial, cf. \eqref{eq:trivial}.
\end{itemize}
Note that $I_0\cap I_1=\emptyset$ is ensured by the rigidity of $\pcts{\vv}$.

Then the direction of $\Dflow{\xevo}$ at any $\ccpt{\vv}\in\pcts{\vv}^\circ$ is
\begin{align}\label{eq:direction}
&\qquad\dflow{\xevo}(\ccpt{\vv})=\sum_{i=0}^k c_{i}\dflow{\sink_{i}}(\ccpt{\vv})\\
\label{eq:decomposition}
&=\sum_{i_0\in I_0}c_{i_0}(\sink_{i_0}-\ccpt{\vv})+\sum_{j\in J}c_j(\real{\ww_j}-\real{\uu_j})+\sum_{i_1\in I_1}c_{i_1}(\ccpt{\vv}-\sink_{i_1}[1])\\
\label{eq:quasi-ratio}
&=\left[(\sum_{i_0\in I_0}c_{i_0} \sink_{i_0}+\sum_{j\in J}c_j\real{\ww_j})-(\sum_{j\in J}c_j\real{\uu_j}+\sum_{i_1\in I_1}c_{i_1} \sink_{i_1}[1])\right]-(\sum_{i_0\in I_0}c_{i_0}-\sum_{i_1\in I_1}c_{i_1})\ccpt{\vv}.
\end{align}

By \Cref{lem:last}, we have
\begin{itemize}
    \item $\Ext^1(\pcto{\vv}, \ww_j)=0=\Ext^1(\uu_j, \pcto{\vv})$ for any $j\in J$.
    \item $\Ext^1(\pcto{\vv}, \sink_{i_0})=0=\Ext^1(\sink_{i_1}[1],\pcto{\vv})$ for any
    $i_0\in I_0$ and $i_1\in I_1$.
\end{itemize}
Note that $\uu_j,\sink_{i_1}[1]\in\pcts{\vv}$ for any $j\in J$ and $i_i\in I_1$.
We further claim that
\begin{enumerate}
    \item $\Ext^1(\ww_j,\sink_{i_0})=0$.
    \item $\Ext^1(\sink_{i_1}[1],\ww_j)=0$.
    \item $\Ext^1(\ww_j,\uu_{j'})=0$ for any $j\ne j' \in J$.
    \item $\Ext^1(\ww_j,\ww_{j'})=0$ for any $j\ne j' \in J$.
\end{enumerate}
\textbf{For $1^\circ$}: Applying $\Hom(-,\sink_{i_0}[1])$ to \eqref{eq:S_j-evo}, we have
\[
0=\Hom(\uu_j,\sink_{i_0}[1])\to\Hom(\ww_j,\sink_{i_0}[1])\to\Hom(\sink_j,\sink_{i_0}[1])=0.
\]
\textbf{For $2^\circ$}: Applying $\Hom(\sink_{i_1},-)$ to \eqref{eq:S_j-evo}, we have
\[
\Hom(\sink_{i_1},\uu_j[-1])\twoheadrightarrow\Hom(\sink_{i_1},\sink_j)\to\Hom(\sink_{i_1},\ww_j)\to\Hom(\sink_{i_1},\uu_j)=0,
\]
where the surjection comes from the right $\pctc{\vv}$-approximativity of $\uu_{j}$ and $\sink_{i_1}[1]\in\pcts{\vv}$.
Hence $\Ext^1(\sink_{i_1}[1],\ww_j)=0$.

\textbf{For $3^\circ$}: Applying $\Hom(\uu_{j'},-)$ to \eqref{eq:S_j-evo}, we have
\[
\Hom(\uu_{j'},\uu_{j})\twoheadrightarrow\Hom(\uu_{j'},\sink_{j}[1])\to\Hom(\uu_{j'},\ww_{j}[1])\to\Hom(\uu_{j'},\uu_{j}[1])=0,
\]
where the surjection comes from the right $\pctc{\vv}$-approximativity of $\uu_{j}$ and $\uu_{j'}\in\pcts{\vv}$.
Hence $\Ext^1(\ww_{j},\uu_{j'})=0$.

\textbf{For $4^\circ$}:
Applying $\Hom(-,\sink_{j}[1])$ to the $\sink_j'$-evolving triangle for $\pcts{\vv}$, we have
\[
\Hom(\uu_{j'},\sink_{j}[1])\to\Hom(\ww_{j'},\sink_{j}[1])\to\Hom(\sink_{j'},\sink_{j}[1])=0.
\]
Combining with the exact sequence in $3^\circ$, we obtain the following diagram.
\[
\begin{tikzcd}
& {\Hom(\uu_{j'},\uu_j)} \arrow[r] \arrow[d, two heads] & {\Hom(\ww_{j'},\uu_j)} \arrow[d ,"{\Hom(\ww_{j'},f_j)}"] \\
& {\Hom(\uu_{j'},\sink_j[1])} \arrow[r, two heads] & {\Hom(\ww_{j'},\sink_j[1])}.
\end{tikzcd}
\]
Hence $\Hom(\ww_{j'},f_j)$ is surjective for any $j'\in J$.
Now applying $\Hom(\ww_{j'},-)$ to \eqref{eq:S_j-evo}, we have
\[
\Hom(\ww_{j'},\uu_{j})\twoheadrightarrow\Hom(\ww_{j'},\sink_{j}[1])\to\Hom(\ww_{j'},\ww_{j}[1])\to\Hom(\ww_{j'},\uu_{j}[1])=0.
\]
and deduce that $\Ext^1(\ww_{j'},\ww_j)=0$.

In summary,
\begin{gather}\label{eq:YY}
    \{ \sink_{i_0} \}_{i_0\in I_0}\bigcup
    \{\ww_j\}_{j\in J}\bigcup
    \pcts{\vv}\bigcup
    \{\uu_j\}_{j\in J}\bigcup
     \{ \sink_{i_1}[1] \}_{i_1\in I_1}
\end{gather}
    form a partial cluster $\pcts{Y}$.
If
\begin{equation}\label{eq:c}
    c\colon=\sum_{i_0\in I_0}c_{i_0}-\sum_{i_1\in I_1}c_{i_1}
\end{equation}
is nonzero, \eqref{eq:quasi-ratio} can be written in the form
\begin{equation}\label{eq:form1}
    \dflow{\xevo}(\ccpt{\vv})=c(\ccpt{Y}-\ccpt{\vv}),
\end{equation}
where
\begin{equation}\label{eq:W}
    \ccpt{Y}=\frac{1}{c}\left[(\sum_{i_0\in I_0}c_{i_0} \sink_{i_0}+\sum_{j\in J}c_j\real{\ww_j})-(\sum_{j\in J}c_j\real{\uu_j}+\sum_{i_1\in I_1}c_{i_1} \sink_{i_1}[1])\right]
\end{equation}
is a point in the simplicial coordinate for the cell $\pcts{Y}$ (which contains $\pcts{\vv}$). Hence the flow is well-defined.

On the other hand, if $c=0$,
then \eqref{eq:quasi-ratio} can be written in the form
\begin{equation}\label{eq:form2}
    \frac{\dflow{\xevo}(\ccpt{\vv})}{c_I+c_J}=\frac{1}{c_I+c_J}
    \big(\sum_{i_0\in I_0}c_{i_0} \sink_{i_0}+\sum_{j\in J}c_j\real{\ww_j}\big)
    -\frac{1}{c_I+c_J}
    \big(\sum_{j\in J}c_j\real{\uu_j}+\sum_{i_1\in I_1}c_{i_1} \sink_{i_1}[1]\big)
\end{equation}
for $c_I=\sum_{i_0\in I_0}c_{i_0}$ and $c_J=\sum_{j\in J}c_{j}$.
The two terms are points in the cell $\pcts{Y}$ and hence the flow is also well-defined.
\end{proof}

Similar to the vertex case,
there are only two singularities in the $\xevo$-evolution flow.

\begin{lemma}\label{lem:sing}
The $\xevo$-evolution flow has only two singularities $\xevo,\Xevo$, i.e., the following holds.
\begin{gather}\label{eq:flow=0}
    \dflow{\xevo}(\ccpt{\vv})=0
    \Longleftrightarrow \ccpt{\vv}\in\{\xevo,\Xevo\}
    \Longleftrightarrow \uflow{\xevo}(\ccpt{\vv})=0.
\end{gather}
\end{lemma}

\begin{proof}
Let $\pcts{\sink}=\{\sink_i\}_{i=0}^k$.
Take $\pcts{\vv}\subset\CCx(\C)$ to be any cell inducing downward partition \eqref{eq:d-part}.
If $J=\emptyset=I_1$ or equivalently $\pcts{\vv}\subset\Star(\pcts{\sink})$, then for any $\ccpt{\vv}\in\pcts{\vv}^\circ$ we have
\[
  \dflow{\xevo}(\ccpt{\vv})=\sum_{i=0}^k c_i \dflow{\sink_i}(\xevo)=\sum_{i=0}^k c_i(\sink_i-\ccpt{\vv})=\xevo-\ccpt{\vv}.\\
\]
If $I_0=\emptyset=J$ or equivalently $\pcts{\vv}\subset\Star(\pcts{\source})$, then for any $\ccpt{\vv}\in\pcts{\vv}^\circ$ we have
\[
  \dflow{\xevo}(\ccpt{\vv})=\sum_{i=0}^k c_i \dflow{\sink_i}(\xevo)=\sum_{i=0}^k c_i(\ccpt{\vv}-\sink_i[1])=\ccpt{\vv}-\Xevo.\\
\]
Hence $\xevo$ and $\Xevo$ are (the only) two singularities of $\dflow{\xevo}$ in $\Star(\pcts{\sink})\cup\Star(\pcts{\source})$.
We proceed to show that outside the two stars (so that $I_0,I_1\ne\emptyset$), there are no more singularities.

It is trivial when $J=\emptyset$ as there are both terms with positive/negative coefficients in \eqref{eq:quasi-ratio}.
Now suppose $J\ne\emptyset$.
Among $\oplus_{j\in J} g_j\colon \oplus_{j\in J} \sink_i \to \oplus_{i\in J} \ww_i$,
take a morphism $g_0\colon \sink_{j_0}\to \ww_0$ in the minimal left $\Add(\oplus_{j\in J}\ww_j)$-approximation of $\oplus_{j\in J} \sink_j$ for some indecomposable $\ww_0$.
We claim that
$\ww_0$ is not a summand of any $\uu_j$ or $\sink_{i_1}[1]$.

Indeed, since $\pcts{\sink}$ is rigid, $\ww_0\not\in\pcts{\sink}[1]$.
If $\ww_0\in\Add(\uu_j)$ for some $j\in J$, applying $\Hom(\sink_{j_0},-)$ to \eqref{eq:S_j-evo} we have
\[
\Hom(\sink_{j_0},\ww_j)\to \Hom(\sink_{j_0},\uu_j)\to \Hom(\sink_{j_0},\sink_j[1])=0,
\]
then $g_0\in\Hom(\sink_{j_0},\uu_j)$ factors through some $\widetilde{g}_0\in\Hom(\sink_{j_0},\ww_j)$, contradiction.
Therefore, both positive and negative coordinates exist in the first part of the vector, which implies that
$\dflow{\xevo}(\ccpt{\vv})\ne0$.

As a result, except for two special cases we have $\dflow{\xevo}\ne0$.
$\xevo$ and $\Xevo$ are the only two singularities of $\dflow{\xevo}$.
\end{proof}

\begin{definition}\label{rem:l-foliation}
The local $\xevo$-evolution flow induces local $\xevo$-leaves:
\begin{itemize}
  \item a parallel linear foliation in any top-cell such that $c=0$ (and then \eqref{eq:form2} holds), decomposing such a cell into one-dimensional intervals (non-degenerate local $\sink$-leaves) and points (degenerate local $\sink$-leaves).
  \item a radial linear foliation in any top-cell containing $\xevo$ or $\Xevo$.
    (i.e. in $\bigcap_{i=0}^k\Star(\sink_i)$ or $\bigcap_{i=0}^k\Star(\sink_i[1])$).
    So by deleting $\sink$ or $\source$, such a cell decomposes into one-dimensional intervals (non-degenerate local $\sink$-leaves).
  \item a quasi-radial linear foliation in any top-cell such that $c\ne0$ (and then \eqref{eq:form1} holds),
    decomposing such a cell into one-dimensional intervals (non-degenerate local $\sink$-leaves) and points (degenerate local $\sink$-leaves).
\end{itemize}
\end{definition}

Our next goal is to show that these local $\sink$-leaves glue well globally.

%=========================================================
%=========================================================
\section{\texorpdfstring{$\xevo$-foliations}{X-foliations}}\label{sec:foli}
%=========================================================
%=========================================================

%=========================================================
\subsection{$\xevo$-foliations}\label{sec:Gflow}\
%=========================================================

Let $\xevo$ be a fixed arbitrary point in $\CCx(\C)$.
We prove the following lemma, which basically says the downward and upward evolution flows are compatible (that they will induce the same foliation).

\begin{lemma}\label{lem:non-split}
Let $l$ be a nontrivial local $\xevo$-leaf in a cell $C_l$.
Let $\ccpt{\vv},\ccpt{\vv'}$ be two points in $l$ such that, along the downward flow,
$\ccpt{\vv}$ is not the sink and $\ccpt{\vv'}$ is not the source.
Then
\begin{equation}\label{eq:propotion}
    \RR_{>0}\cdot\dflow{\xevo}(\ccpt{\vv})=
    \RR_{<0}\cdot\uflow{\xevo}(\ccpt{\vv'}),
\end{equation}
in the sense that there is some $\lambda<0$ such that $\dflow{\xevo}(\ccpt{\vv})=\lambda\cdot\uflow{\xevo}(\ccpt{\vv'})$.

\end{lemma}
\begin{proof}
We may assume that $\ccpt{\vv}\in\pcts{\vv}^\circ$ is the source of $l$.
Let \eqref{eq:d-part} be the downward partition induced by $\pcts{\vv}$.
Then the direction $\dflow{\xevo}(\ccpt{\vv})$ can be expressed as \eqref{eq:quasi-ratio},
where $\uu_j\in\pcts{\vv}$ while $\ww_j$ is not necessarily for any $j\in J$.
For any $\ccpt{\vv'}\ne\ccpt{\vv}$ in $l$, we have
\begin{gather}\label{eq:vv'}
 \ccpt{\vv'}=\ccpt{\vv}+t\dflow{\xevo}(\ccpt{\vv}),
\end{gather}
where $t>0$ such that $\ccpt{\vv'}$ is in the interior of the cell $C_l$
, which is of the form $\pcts{Y}$ in \eqref{eq:YY}.
Combining \Cref{pp:complete} and \Cref{lem:well-defined}, for any $i\in\{0,\ldots, k\}$, the following triangles coincide.
\begin{itemize}
    \item The upward $\sink_i$-evolving triangle for $\pcts{Y}$.
    \item The downward $\sink_i$-evolving triangle for $\pcts{Y}$.
    \item The downward $\sink_i$-evolving triangle for $\pcts{\vv}$.
\end{itemize}

If $c=0$ for $c$ in \eqref{eq:c}, then $\dflow{\xevo}(\ccpt{\vv})=\dflow{\xevo}(\ccpt{\vv'})=-\uflow{\xevo}(\ccpt{\vv'})$.

If $c\ne0$, combining \eqref{eq:vv'} and \eqref{eq:form1} we have
\begin{equation}
  \dflow{\xevo}(\ccpt{\vv})= c(\ccpt{Y}-\ccpt{\vv})
    =c(\ccpt{Y}-\frac{\ccpt{\vv'}-tc\ccpt{Y}}{1-tc})
    =-(1-tc)^{-1}\uflow{\xevo}(\ccpt{\vv'}),
\end{equation}
for $Y$ in \eqref{eq:W}.
Hence, $\dflow{\xevo}(\ccpt{\vv})=\lambda\cdot\uflow{\xevo}(\ccpt{\vv'})$ for $\lambda=-(1-tc)^{-1}$.
If $1-tc<0$, then $\ccpt{\vv'}=(1-tc)\ccpt{\vv}+tc\ccpt{Y}$ will not be in the interior of $C_l=\pcts{Y}$, which is a contradiction.
Thus $\lambda<0$ as required.
\end{proof}

In conclusion, the downward/upward $\sink_j$-evolving triangles for any point $\ccpt{\vv}$ (other than the two singularities)
ensures that $\ccpt{\vv}$ is in two non-trivial local $\xevo$-leaves, which may coincide.
For instance, for an interior point of any top-cell $\ccpt{\vv}$,
since $\dflow{\sink}=-\uflow{\sink}$, we know that they are contained in exactly one non-degenerate local $\sink$-leaf.

Now, we can glue the local $\xevo$-leaves into global ones as consequences.

\begin{theorem}\label{thm:foliation}
The (downward/upward) $\xevo$-evolution flow induces a pairwise linear foliation on $\CCx(\C)$ with only two singularities $\xevo$ and $\Xevo$,
%the unique sink $\xevo$ and the unique source $\Xevo$.
In other words, there is a decomposition
\begin{gather}\label{eq:GS-foli}
    \CCx(\C)=\bigcup_{\text{$\xevo$-leaf $\Traj$}}\Traj
\end{gather}
such that each $\xevo$-leaf $L$ is a 1-dimensional manifold and any two $\xevo$-leaves can only intersect at their endpoint, which is either $\xevo$ or $\Xevo$. We call \eqref{eq:GS-foli} the \emph{$\xevo$-foliation} on $\CCx(\C)$.
\end{theorem}

\begin{proof}
Locally, The $\xevo$-evolution flow induces a linear foliation in each top-cell as in \Cref{rem:l-foliation}.
By \Cref{lem:non-split}, we can glue the non-degenerate local $\xevo$-leaves into global $\xevo$-leaves which are $1$-dimensional topological manifolds. Hence, we obtain a piecewise linear foliation on $\CCx(\C)$, the $\xevo$-foliation.
By \Cref{lem:sing}, the $\xevo$-foliation admits a unique sink $\xevo$ and a unique source $\Xevo$.
\end{proof}

Thus, we obtain a continuous $\CCx(\C)$-family of piecewise linear foliations on $\CCx(\C)$,
parameterized by its unique sink $\xevo$ and its unique source $\Xevo$.

Let $\Traj$ be an $\xevo$-leaf and $\pcts{\vv}$ be an cell.
We say $\Traj$ intersects $\pcts{\vv}$ nontrivially
if $l_{\pcts{\vv}}=\Traj\cap\pcts{\vv}$ is a non-degenerate local $\xevo$-leaf in $\pcts{\vv}$.
In this case, \eqref{eq:propotion} holds.

\begin{definition}\label{def:ccross}
We say there is a \emph{cell-crossing} of $\xevo$-leaf $L$ at $\ccpt{V}=\sum_{i=1}^k\alpha_i V_i$ if
$\ccpt{\vv}$ is a degenerate local $\xevo$-leaf.
In such a case, the downward $\xevo$-evolution flow will travel from a cell $\pcts{U}$ to a cell $\pcts{W}$ such that there are two non-degenerate local $\xevo$-leaves $l_{\pcts{U}}\cap l_{\pcts{W}}=\ccpt{V}$.
We call $\pcts{V}=\pcts{U}\cap\pcts{W}$ the \emph{cell-wall} of this cell-crossing.

\begin{figure}[ht]\centering
\begin{tikzpicture}[scale=.7]
\draw[thick](0,0)--(6,0);
\draw[very thick, Emerald, >=stealth,->-=.5](5,2) to (3,0);
\draw[very thick, Emerald, >=stealth,->-=.5](3,0) to ($(3,0)+(-120:2.828)$);
\draw (3,0)\nn
(.2,1) node[font=\large] {$\pcts{U}$}
(.2,-1) node[font=\large]  {$\pcts{W}$}
(6.6,0) node[font=\large]  {$\pcts{V}$}
(3.3,-.3) node {$\ccpt{V}$}
(5.3,2.3) node[Emerald] {$l_{\pcts{U}}$}
($(3,0)+(-120:3)$) node[Emerald] {$l_{\pcts{W}}$}
;
\end{tikzpicture}
    \caption{Refraction during cell-crossing at $\ccpt{V}$}
    \label{fig:refraction}
\end{figure}
\end{definition}

We will discuss the first property of cell-crossing as follows.

\begin{lemma}\label{lem:cwall}
Let $\xevo=\sum_{i=1}^k c_i\sink_i$ and $\pcts{\vv}$ is the cell-wall of a cell-crossing of $\xevo$-foliation.
If
\begin{gather}
\label{eq:uevoi}
    \sink_i\to \App{L}{\pctc{\vv}}{\sink_i}\to\uu_i\to\sink_i[1],\\
\label{eq:devoi}
    \sink_i\to\ww_i\to\App{R}{\pctc{\vv}}{\sink_i[1]}\to\sink_i[1]
\end{gather}
are the upward and downward $\sink_i$-evolving triangles for $\pcts{\vv}$.
Then
\begin{enumerate}
    \item \eqref{eq:uevoi} is also the downward $\sink_i$-evolving triangle for $\pcts{\uu}$.
    \item \eqref{eq:devoi} is also the upward $\sink_i$-evolving triangle for $\pcts{\ww}$.
    \item $\pcts{\vv}\subsetneqq\pcts{\uu}=\cc(\pcto{\vv}\oplus\bigoplus_{i=1}^n\uu_i)$ and $\pcts{\vv}\subsetneqq\pcts{\ww}=\cc(\bigoplus_{i=1}^n\ww_i\oplus\pcto{\vv})$.
\end{enumerate}
\end{lemma}

\begin{proof}
By \Cref{pp:complete}, $\Ext^1(\pcto{U},\App{L}{\pctc{V}}{X_i})=0$ implies $1^\circ$ and $2^\circ$ is similar.

If $\pcts{\vv}=\pcts{\uu}$, then $\ccpt{\vv}$ is in the interior of $l_{\pcts{\uu}}$ but not a non-degenerate local $\xevo$-leaf, which is a contradiction. Similar for $\pcts{\vv}\subsetneqq\pcts{\ww}$.
By the formula \eqref{eq:decomposition} for $\dflow{\xevo}$ and its dual, $3^\circ$ follows.
\end{proof}

%\begin{remark}\label{rem:family}
%Note that when choosing a cluster $\pcts{\vv}$, it induces an orientation of the (unoriented) cluster exchange graph $\uCEG(\C)$ (cf. \cite[\S9]{KQ1}, cf. \cite{Ke2,KQ2}),
%with a unique sink $\pcts{\vv}$ and a unique source $\pcts{\vv}[1]$.
%One can think that such an orientation is roughly a discrete dual flow on $\uCEG(\C)$.
%But the evolution flows provide much more information and are more powerful, cf. \Cref{sec:ceg}.
%Also, when choosing $\xevo=\real{\pcto{\vv}}$ to be the center of a top-cell $\pcts{\vv}$,
%the $\xevo$-folitation does not necessarily induce an orientation on $\uCEG(\C)$.
%However, we conjecture that it does and coincide with the orientation induced by $\xevo$,
%where the downward/upward directions of the flow become backward/forward directions of the mutation.
%See \Cref{fig:A3-F-F*} and \Cref{fig:A12-F-F*} for examples.
%%If it induces one, we guess it will match the one above.
%\end{remark}

%=========================================================
\subsection{Types of $\xevo$-foliations}\label{sec:type}\
%=========================================================

Since the global $\xevo$-leaves are $1$-dimensional topological manifolds, they must be one of the following:
\begin{itemize}
    \item \emph{compact}, i.e. $\Traj\simeq[0,1]$ with endpoint $\xevo$ and $\Xevo$.
    \item \emph{semi-compact}, i.e. $\Traj\simeq[0,1)$ with endpoint $\xevo$ or $\Traj\simeq(0,1]$ with endpoint $\Xevo$.
    \item \emph{non-compact}, i.e. $\Traj\simeq(0,1)$.
    \item \emph{closed}, i.e. $\Traj\simeq\SS^1$.
\end{itemize}

\begin{definition}
We say the $\xevo$-foliation on $\CCx(\C)$ is
\begin{itemize}
    \item \emph{compact} if any $\xevo$-leaf is compact;
    \item \emph{semi-compact} if any $\xevo$-leaf is compact or semi-compact,\\% but not all $\xevo$-leaves are compact;
    \item \emph{acyclic} if there is no closed $\xevo$-leaf.
\end{itemize}
\end{definition}
We expect that any $\xevo$-foliation is acyclic or even semi-compact.
By the realization \eqref{eq:red}, the compact $\sink$-foliations provide us with nice homotopy equivalences.

\begin{lemma}\label{lem:compact}
A compact $\xevo$-foliation induces the following contractions.
\begin{gather}\label{eq:contr}
\begin{cases}
\dflow{\xevo}:&\CCx(\C)\setminus\{\Xevo\}\xrightarrow{\simeq}\{\xevo\},\\
\uflow{\xevo}:&\CCx(\C)\setminus\{\xevo\}\xrightarrow{\simeq}\{\Xevo\}.
\end{cases}
\end{gather}
Moreover, if $\xevo=\sink$ is a 0-cell, then the $\sink$-foliation induces a deformation retraction from $\CCx(\C)\setminus\{\xevo,\Xevo\}$ to $\CCx(\C\backslash\sink)\cong\CCx(\C\backslash\source)$
and hence we have the following homotopy equivalence.
  \begin{gather}\label{eq:susp}
    \CCx(\C)\simeq \Sigma\CCx(\C\backslash\sink).
  \end{gather}
\end{lemma}
\iffalse
\begin{lemma}\label{lem:semi-compact}
A semi-compact $\xevo$-foliation induces a deformation retraction
\begin{gather}\label{eq:kernel}
\on{Evo}_{\xevo}:\CCx(\C)\xrightarrow{\simeq}\CCx^c(\C)=\bigcup_{\Traj~\text{compact}}\Traj.
\end{gather}
\end{lemma}
\fi
%=========================================================
\subsection{$\sink$-traces as Calabi-Yau reduction}\label{sec:uni.ext}\
%=========================================================

\begin{definition}\label{def:uext}
Let $\ww,\uu$ be objects in $\C$.
The \emph{irreducible extension of $\uu$ on the top of $\sink$} is defined to be
\[
    \ueot(\uu)=\Cone(\uu\xrightarrow{f}\App{L}{\source}{\uu})[-1]
\]
and the \emph{irreducible extension of $\ww$ by $\sink$} is defined to be
\[
    \ueby(\ww)=\Cone(\App{R}{\sink}{\ww}\xrightarrow{g} \ww).
\]
\end{definition}

Obviously,
\begin{gather}\label{eq:tri-ext}
    \begin{cases}
      \ueot(\ww)=\ww, & \text{if $\Ext^1(\ww,\sink)=0$}; \\
      \ueby(\uu)=\uu, & \text{if $\Ext^1(\source,\uu)=0$}.
    \end{cases}
\end{gather}
In particular, $\ueot(\source)=0=\ueby(\sink)$.

\begin{lemma}\label{lem:ie}
For any triangle
\begin{equation}\label{eq:ie}
    \sink^{\oplus a}\xrightarrow{g} \ww\to \uu\xrightarrow{f}\source^{\oplus a},
\end{equation}
we have
\begin{itemize}
    \item $f$ is left $\Add(\source)$-approximation $\iff$ $\Ext^1(\ww,\sink)=0$
    \item $f$ is moreover left minimal $\iff$ $\sink\not\in\Add(\ww)$
\end{itemize}
Dually,
\begin{itemize}
    \item $g$ is right $\Add(\sink)$-approximation $\iff$ $\Ext^1(\source,\uu)=0$
    \item $g$ is moreover right minimal $\iff$ $\source\not\in\Add(\uu)$
\end{itemize}
\end{lemma}

\begin{proof}
We only show the equivalences of $f$, the equivalences of $g$ are similar.

Applying $\Hom(-,\source)$ to the triangle \eqref{eq:ie}, we have
    \begin{gather*}
        \Hom(\source^{\oplus a},\source)\xrightarrow{\Hom(f,\source)}\Hom(\uu,\source)\to\Hom(\ww,\source)\to\Hom(\sink^{\oplus a},\source)=0.
    \end{gather*}
Then $\Hom(f,\source)$ is surjective if and only if
%, i.e. the right $\Add \source$-appproximativity of $f$
$\Ext^1(\ww,\sink)=0$,
which shows the first equivalence.
The second equivalence was proofed in \cite[Lem.~3.4 and Lem~3.5]{Pla}
\end{proof}

\begin{proposition}\label{prop:ie}
The irreducible extension has the following properties:
\begin{enumerate}
    \item $(\ueot)^2=\ueot\quad\text{and}\quad(\ueby)^2=\ueby$.
    \item $\ueot\ueby\ueot=\ueot\quad\text{and}\quad\ueby\ueot\ueby=\ueby$.
    \item
    $\ueot(\vv_1\oplus \vv_2)=\ueot(\vv_1)\oplus\ueot(\vv_2)$, and\\ $\ueby(\vv_1\oplus \vv_2)=\ueby(\vv_1)\oplus\ueby(\vv_2)$
    \item $\Ext^1(\ueot(\uu),\ueot(\uu'))=0$ for any $\Ext^1(\uu,\uu')=0$, and \\ $\Ext^1(\ueby(\ww),\ueby(\ww'))=0$ for any $\Ext^1(\ww,\ww')=0$.
    \item
    In any nontrivial $\sink$-evolving triangle \eqref{eq:evo-tri},\\
     $\ueot(\ww)=\ueot(\uu)\quad\text{and}\quad
    \ueby(\ww)=\ueby(\uu)$ .

\end{enumerate}
\end{proposition}

\begin{proof}
For each item, we only deal with the downward version and upward is similar.
\paragraph{\textbf{For $1^\circ$:}}
By \Cref{lem:ie}, $\Ext^1(\ueot(\uu),\sink)=0$ and
then $\eqref{eq:tri-ext}$ implies $1^\circ$.

\paragraph{\textbf{For $2^\circ$:}}
For any object $\uu$ in $\C$.
Consider the following irreducible extensions:
\begin{gather*}
    \sink^{\oplus a}\xrightarrow{g_1}\ueot(\uu)\to\ueby\ueot(\uu)\xrightarrow{f_1} \source^{\oplus a},\\
    \sink^{\oplus b}\xrightarrow{g_2}\ueot\ueby\ueot(\uu)\to\ueby\ueot(\uu)\xrightarrow{f_2} \source^{\oplus b}.
\end{gather*}

By \Cref{lem:ie}, we have
\[
    \Ext^1(\ueot\ueby\ueot(\uu),\sink)=0=\Ext^1(\source,\ueby\ueot(\uu)),
\]
hence $f_1$ is a left $\Add(\source)$-approximation and $g_2$ is a right $\Add(\sink)$-approximation. Note that $g_1$ and $f_2$ are minimal, so $b\le  a$ and $a\le b$. We have $a=b$ and thus $\ueot\ueby\ueot(\uu)=\ueot(\uu)$.

\paragraph{\textbf{For $3^\circ$:}}
Consider the irreducible extensions of $\vv_1,\vv_2$ and $\vv_1\oplus \vv_2$ as follows:
\begin{gather*}
    \sink^{a_i}\to \ueot(\vv_i)\to \vv_i\xrightarrow{f_i} \source^{a_i} \quad\mbox{for}\quad i=1,2,\\
    \sink^{a}\to \ueot(\vv_1\oplus \vv_2)\to \vv_1\oplus \vv_2\xrightarrow{f} \source^{a_i}.
\end{gather*}

Since $f$ is left minimal $\Add(\source)$-approximation of $\vv_1\oplus \vv_2$, we have $a\le a_1+a_2$.
Using the octahedra axiom, we obtain the octahedron:
\[
\begin{tikzcd}
&\sink^{a_1+a_2}\arrow[d]\arrow[r, equal]&\sink^{a_1+a_2}\arrow[d]&\\
\ueot(\vv_1\oplus \vv_2)\arrow[r]\arrow[d, equal]&\ueot(\vv_1)\oplus\ueot(\vv_2)\arrow[r]\arrow[d]&\sink^{(a_1+a_2)-a}\arrow[d]\arrow[r,"\sigma"]&\ueot(\vv_1\oplus \vv_2)[1]\arrow[d, equal]\\
\ueot(\vv_1\oplus \vv_2)\arrow[r]&\vv_1\oplus \vv_2\arrow[r]\arrow[d]&\source^a\arrow[d]\arrow[r]&\ueot(\vv_1\oplus \vv_2)[1]\\
&{\source^{a_1+a_2}}\arrow[r, equal]& {\source^{a_1+a_2}}&.
\end{tikzcd}
\]
By \Cref{lem:ie}, we have $\sigma=0$ and $\sink\not\in \Add(\ueot(\vv_1)\oplus\ueot(\vv_2))$.
Therefore, $a=a_1+a_2$ and $\ueot(\vv_1\oplus \vv_2)=\ueot(\vv_1)\oplus\ueot(\vv_2)$.

\paragraph{\textbf{For $4^\circ$:}}
Consider the irreducible extensions of $\uu,\uu'$:
\begin{gather*}
    \sink^{\oplus a}\to \ueot(\uu')\to \uu'\xrightarrow{f} \source^{\oplus a},\\
     \sink^{\oplus b}\xrightarrow{g} \ueot(\uu)\to \uu\to \source^{\oplus b}.
\end{gather*}
Applying $\Ext^1(\ueot(\uu),-)$ to the first triangle, we have
\begin{gather}\label{les3}
    \nonumber0\xlongequal{Lem.~\ref{lem:ie}}\Ext^1(\ueot(\uu),\sink^{\oplus a})\to\Ext^1(\ueot(\uu),\ueot(\uu'))\to\Ext^1(\ueot(\uu),\uu')\\
    \xrightarrow{\Ext^1(\ueot(\uu),f)}\Ext^1(\ueot(\uu),\source^{\oplus a}).
\end{gather}
Applying $\Ext^1(\sink^{\oplus b},-)$ to the first triangle and $\Ext^1(-,\uu')$ to the second, we have
\[
\begin{cases}
0\xlongequal{Lem.~\ref{lem:ie}}\Ext^1(\sink^{\oplus b},\ueot(\uu'))\to\Ext^1(\sink^{\oplus b},\uu')\to
    \Ext^1(\sink^{\oplus b},\source^{\oplus a})\\
0=\Ext^1(\uu,\uu')\to\Ext^1(\ueot(\uu),\uu')\to
    \Ext^1(\sink^{\oplus b},\uu'),
\end{cases}
\]
Combining with the following communicative diagram:
\[
\begin{tikzcd}
& {\Ext^1(\ueot(\uu),\uu')} \arrow[r, hook] \arrow[d, "{\Ext^1(\ueot(\uu),f)}"'] & {\Ext^1(\sink^{\oplus b},\uu')} \arrow[d, hook] \\
& {\Ext^1(\ueot(\uu),\source^{\oplus a})} \arrow[r] & {\Ext^1(\sink^{\oplus b},\source^{\oplus a})},
\end{tikzcd}
\]
we deduce that $\Ext^1(\ueot(\uu),f)$ is injective and thus
$\Ext^1(\ueot(\uu),\ueot(\uu'))=0$.

\paragraph{\textbf{For $5^\circ$:}}
Consider the irreducible extensions of $\ww$ and $\uu$:
\begin{gather*}
    \sink^a\to\ueot(\ww)\to \ww\xrightarrow{f_\ww}\source^a,\\
    \sink^b\to\ueot(\uu)\to \uu\xrightarrow{f_\uu}\source^b.
\end{gather*}
Combining the first triangle with \eqref{eq:evo-tri}, we obtain the following by octahedral axiom.
\[
\begin{tikzcd}
&\ueot(\ww)\arrow[d]\arrow[r, equal]&\ueot(\ww) \arrow[d]&\\
\sink\arrow[r]\arrow[d, equal]&\ww\arrow[r]\arrow[d,"f_\ww"]&\uu\arrow[d,"f'_\uu"]\arrow[r,"f"]&\source\arrow[d, equal]\\
\sink\arrow[r,"0"]&\source^{a}\arrow[r]\arrow[d]&\source^{a+1}\arrow[d]\arrow[r]&\source\\
&{\ueot(\ww)[1]}\arrow[r, equal]& {\ueot(\ww)[1]}&,
\end{tikzcd}
\]
By \Cref{lem:ie},
$\Ext^1(\ueot(\ww),\sink)=0$ and $\sink\not\in\Add(\ueot(\ww))$ implies that
    $f'_\uu$ is the minimal left $\Add(\source)$-approximation of $\uu$.
Hence $f_\uu=f'_\uu$ and $\ueot(\ww)=\ueot(\uu)$.
\end{proof}

\begin{remark}
Since the irreducible extensions preserve the direct sum and partial cluster by \Cref{prop:ie}, they can be extended to the cluster complex by
\[
\ue(\ccpt{\vv})=\sum_{i=0}^k\alpha_i\real{\ue(\vv_i)},
\]
where $?=\uparrow,\downarrow$.
\end{remark}

Unsurprisingly, the shift functor induces an equivalence $[1]:\C\backslash \sink\xrightarrow{\cong}\C\backslash \source$.
However, the irreducible extension provides a more natural way to identify the two subcategories.

\begin{corollary}\label{lem:1-1}
The irreducible extensions induce a triangle equivalence
\[
\ueby:\C\backslash \sink \overset{\cong}{\longleftrightarrow}\C\backslash \source:\ueot
\]
extending the identity on $\C\backslash \sink\cap\C\backslash \source$. Moreover, it induces an isomorphism between simplicial complexes:
\[
\ueby:\CCx(\C\backslash \sink) \overset{\cong}{\longleftrightarrow}\CCx(\C\backslash \source):\ueot.
\]
\end{corollary}

\begin{proof}
By \Cref{prop:ie}, $\ueot:\C\backslash \source\to\C\backslash \sink$ is an additive equivalence which is identity on $\C\backslash \sink\cap\C\backslash \source$. By $\eqref{def:shift}$, the following triangle
\begin{gather*}
\uu[-1]\to \sink\to \ueot(\uu)\to \uu.
\end{gather*}
and its shift imply that $\<1\>_{\C\backslash \sink}[-1]=\ueot=[-1]\<1\>_{\C\backslash \source}$.
As a result, $\ueot$ preserves the triangle structure.
Moreover, it induces the isomorphism of simplicial complexes between $\CCx(\C\backslash \sink)$ and $\CCx(\C\backslash \source)$.
\end{proof}

For any $k$-cell $\pcts{T}$ in $\CCx(\C)$, we define
\begin{equation}\label{eq:proj}
    \begin{array}{rccl}
  \proj{\pcts{T}}\colon&\Star(\pcts{T})\backslash\pcts{T}&\longrightarrow& \link(\pcts{T})  \\
    & \displaystyle{ \sum_{i=0}^k c_i T_i+\sum_{i=k+1}^m c_i Y_i }
    &\longmapsto&\dfrac{\sum_{i=k+1}^m c_i Y_i}{\sum_{i=k+1}^m c_i}.
\end{array}
\end{equation}

to be the projection from $\Star(\pcts{T})$ onto $\link(\pcts{T})$.
Now we introduce \emph{the downward/upward $\sink$-traces} of $\sink$-leaves.

\begin{definition}
The \emph{downward/upward $\sink$-traces}
\begin{gather}\label{eq:dtr}
    \des\colon \CCx(\C)\setminus\{\sink,\source\}\to\CCx(\C\backslash \sink),\\
    \label{eq:utr}
    \ori\colon \CCx(\C)\setminus\{\sink,\source\}\to\CCx(\C\backslash \source)
\end{gather}
are defined to be
\begin{gather}\label{def:des}
\des(\ccpt{\vv})=
\begin{cases}
\ueot(\ccpt{\vv}),    &\mbox{if $\ccpt{\vv}\not\in\Star(\sink)$},\\
\proj{\sink}(\ccpt{\vv}), &\mbox{if $\ccpt{\vv}\in\Star(\sink)$, }
\end{cases}
\end{gather}
and, respectively,
\begin{gather}\label{def:ori}
\ori(\ccpt{\vv})=
\begin{cases}
\ueby(\ccpt{\vv}),        &\mbox{if $\ccpt{\vv}\not\in\Star(\source)$},\\
\proj{\source}(\ccpt{\vv}),  &\mbox{if $\ccpt{\vv}\in\Star(\source)$ }.
\end{cases}
\end{gather}
\end{definition}

\begin{theorem}\label{thm:const}
The downward/upward $\sink$-traces are invariant on any $\sink$-leaf $L$.
\end{theorem}

\begin{proof}
We only proof the downward $\sink$-trace and similar for the upward $\sink$-trace.

Since any $\ccpt{\vv}\ne\ccpt{\vv'}\in\Traj$ is connected by finite nontrivial local $\sink$-leaves $\{l_i\}_{i=0}^N$ along the downward flow.
It is sufficient to show that for any $\ccpt{\vv}\ne\ccpt{\vv'}$ in the same nontrivial local $\sink$-leaf $l$, we have $\des(\vv)=\des(\vv')$.

If $l$ is the nontrivial local $\sink$-leaf in $\Star(\sink)$.
By \eqref{eq:proj}, we have
\[
\des(\vv)=\proj{\sink}(\vv)=\proj{\sink}(\vv')=\des(\vv').
\]
If $l$ is the nontrivial local $\sink$-leaf for $C_l\not\subset\Star(\sink)$.
Let \eqref{eq:evo-tri} be the (nontrivial) downward $\sink$-evolving triangle for the partial cluster $C_{l}$.
By $5^\circ$ of \Cref{prop:ie}, we have $\ueot(\ww)=\ueot(\uu)$.
The vector $\dflow{\sink}(\ccpt{Y})=\real{\ww}-\real{\uu}$ does not affect the downward trace for any $\ccpt{Y}\in l$.
Hence $\des(\ccpt{\vv})=\des(\ccpt{\vv'})$.

In conclusion, the downward $\sink$-trace is invariant on any $\sink$-leaf $L$.
\end{proof}

\begin{corollary}
When the $\sink$-foliation is compact,
$\ori$ and $\des$ are precisely the deformation retractions in \Cref{lem:compact}.
\end{corollary}

One can also extend the definition of traces to an arbitrary point $\xevo$. We leave it to the further works.

\begin{figure}\centering
\vskip -.5cm
\makebox[\textwidth][c]{
  \includegraphics[width=9cm]{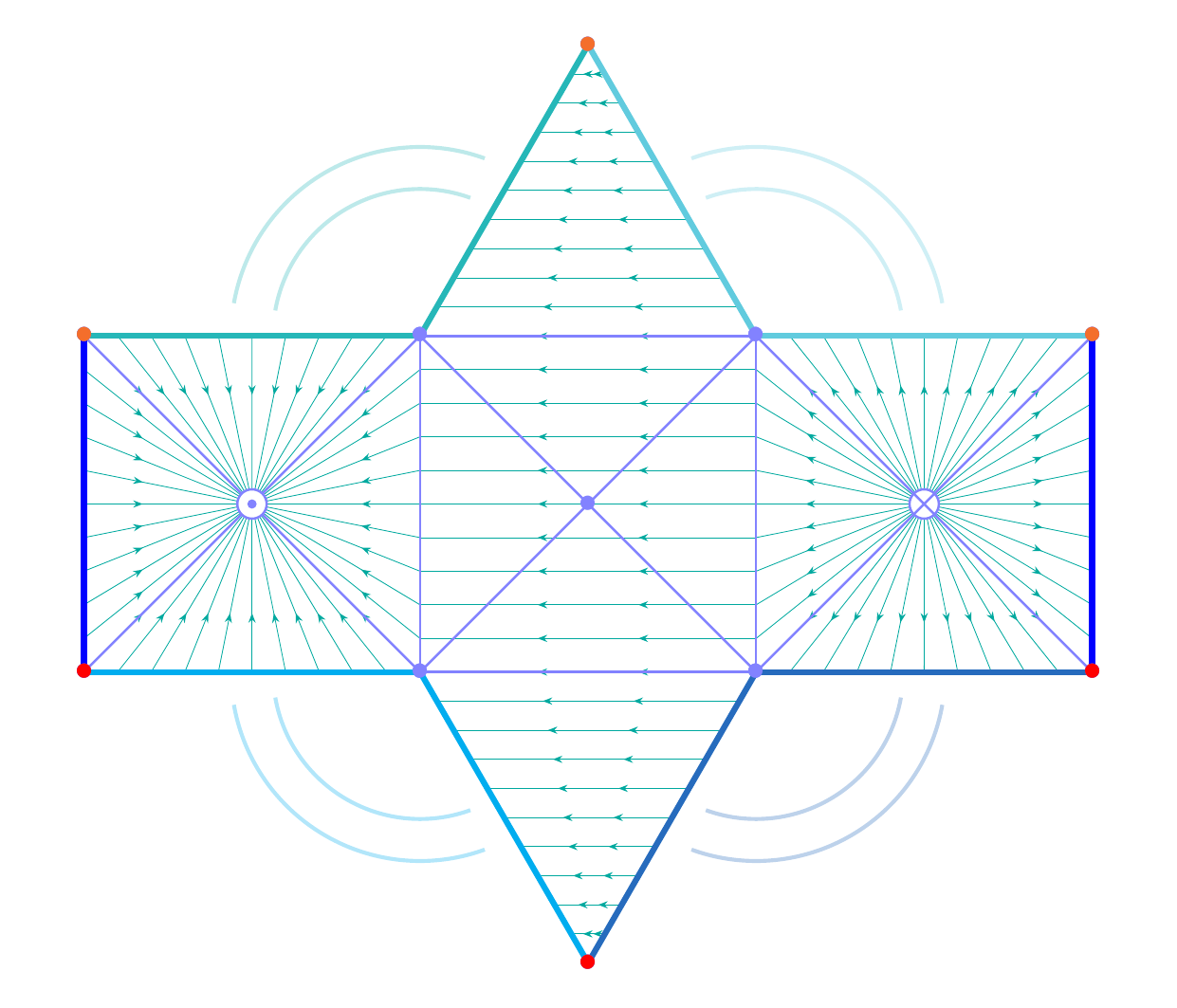}
\hskip -.7cm
  \includegraphics[width=9cm]{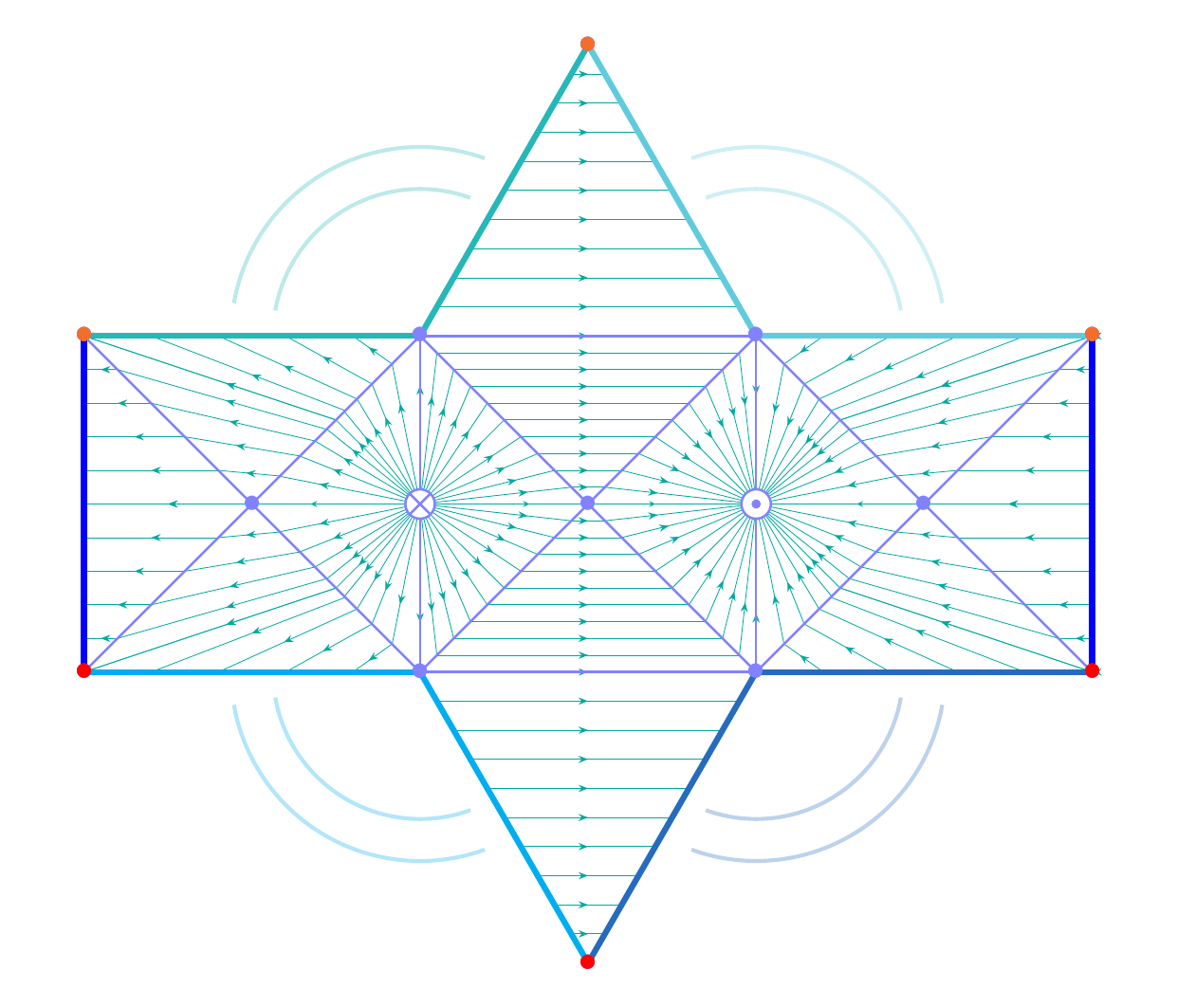}}
\vskip -.6cm
\makebox[\textwidth][c]{
  \includegraphics[width=9cm]{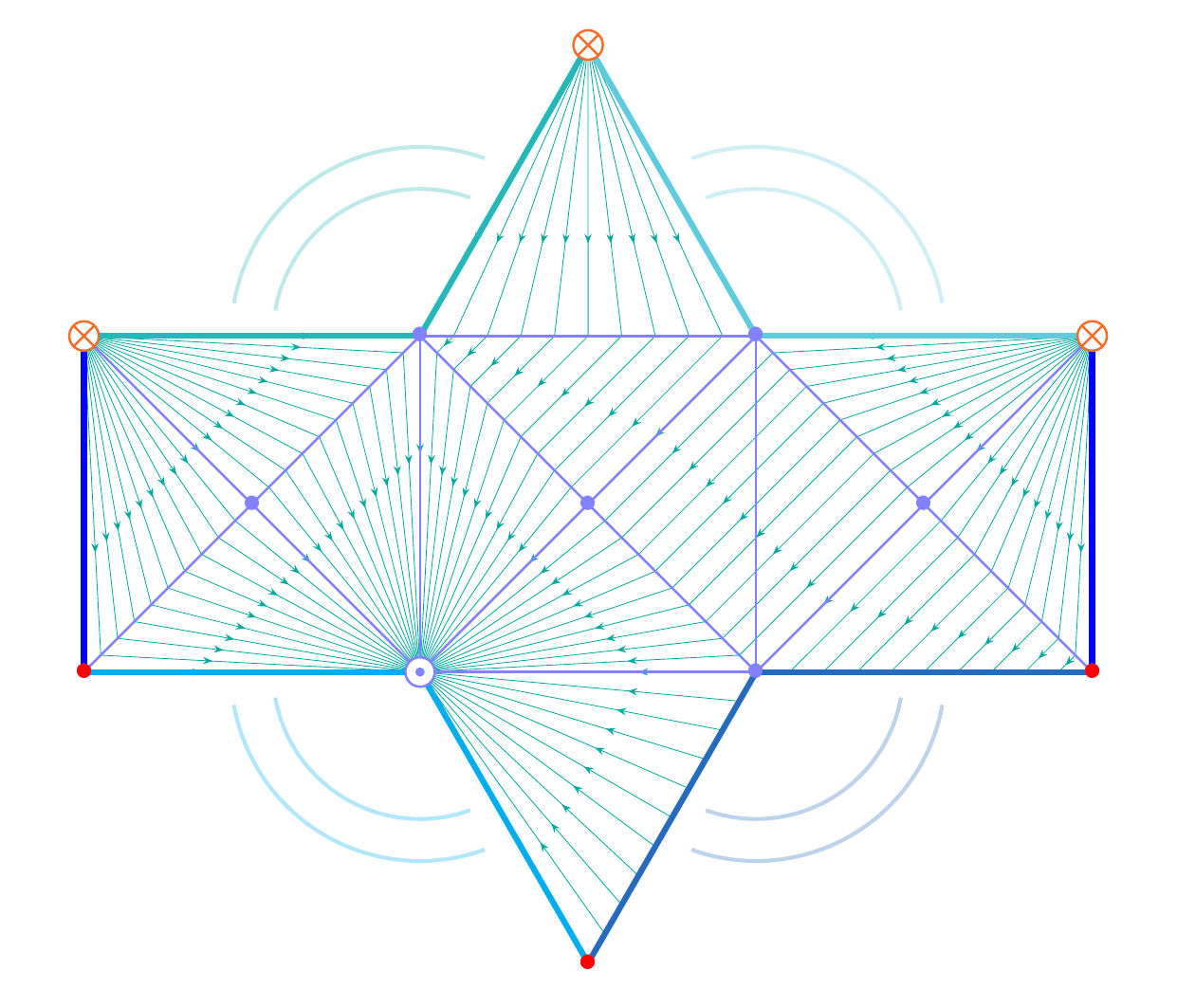}
\hskip -.7cm
  \includegraphics[width=9cm]{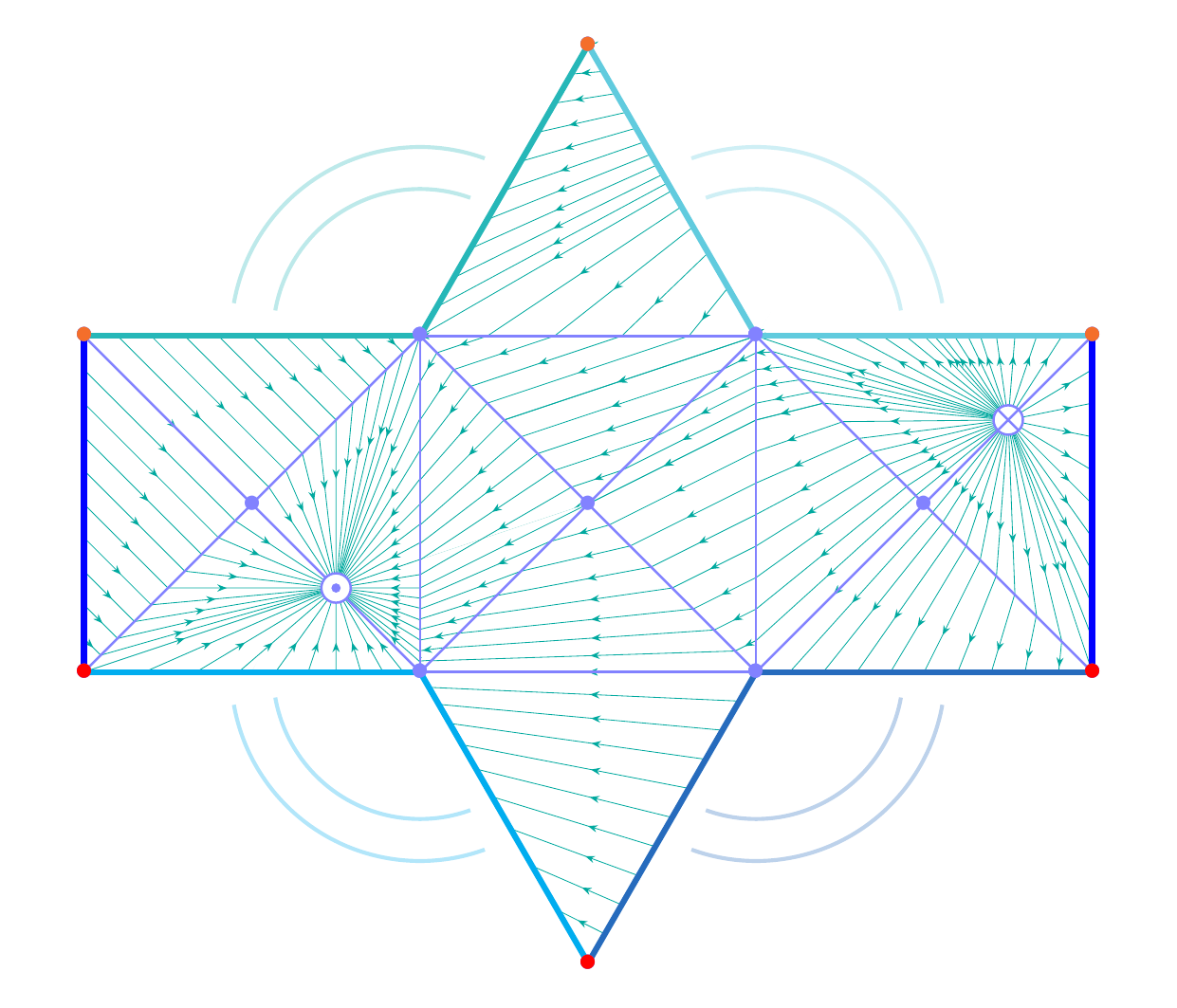}\qquad\qquad}
\vskip -.6cm
\makebox[\textwidth][c]{
  \includegraphics[width=9cm]{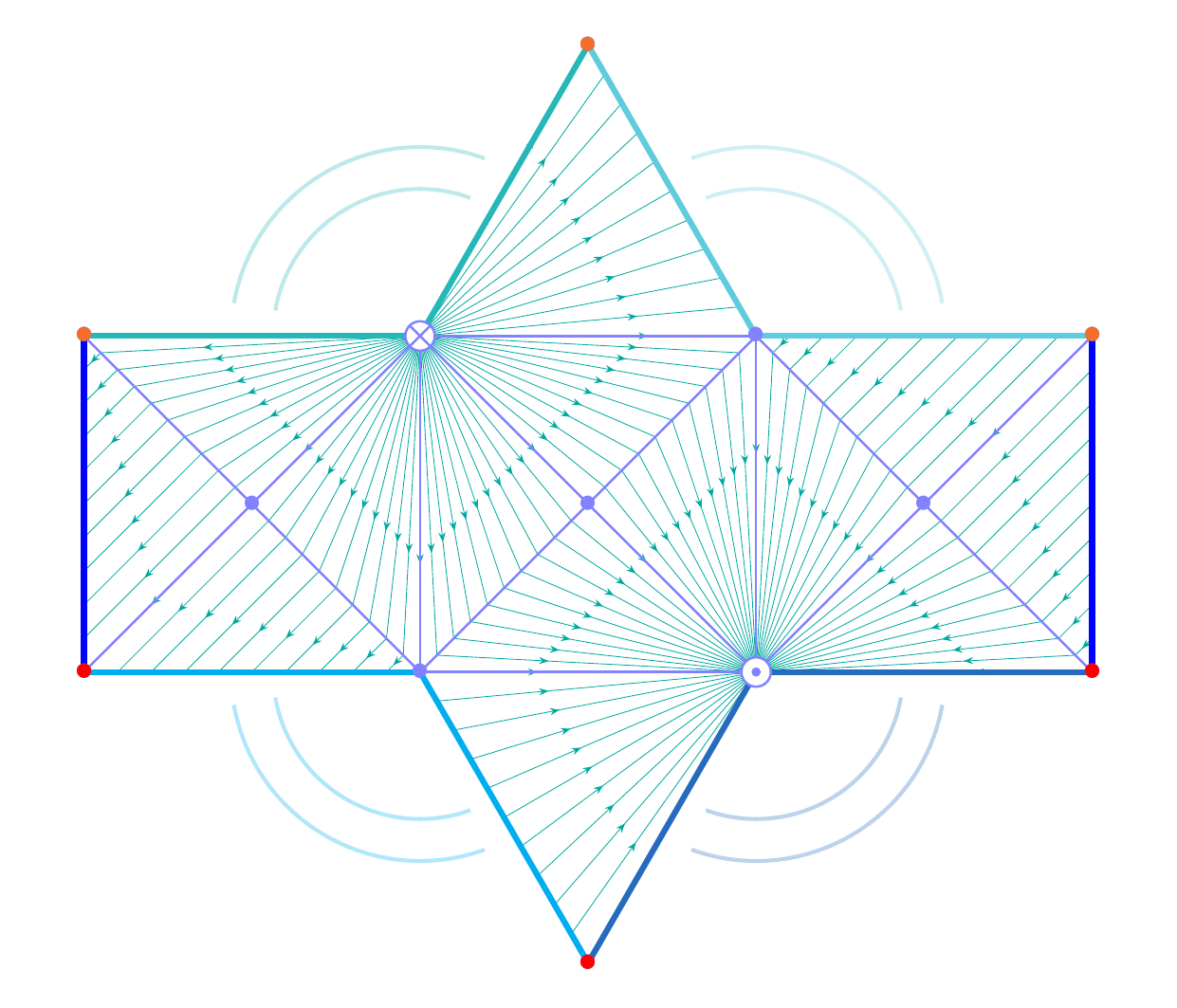}
\hskip -.7cm
  \includegraphics[width=9cm]{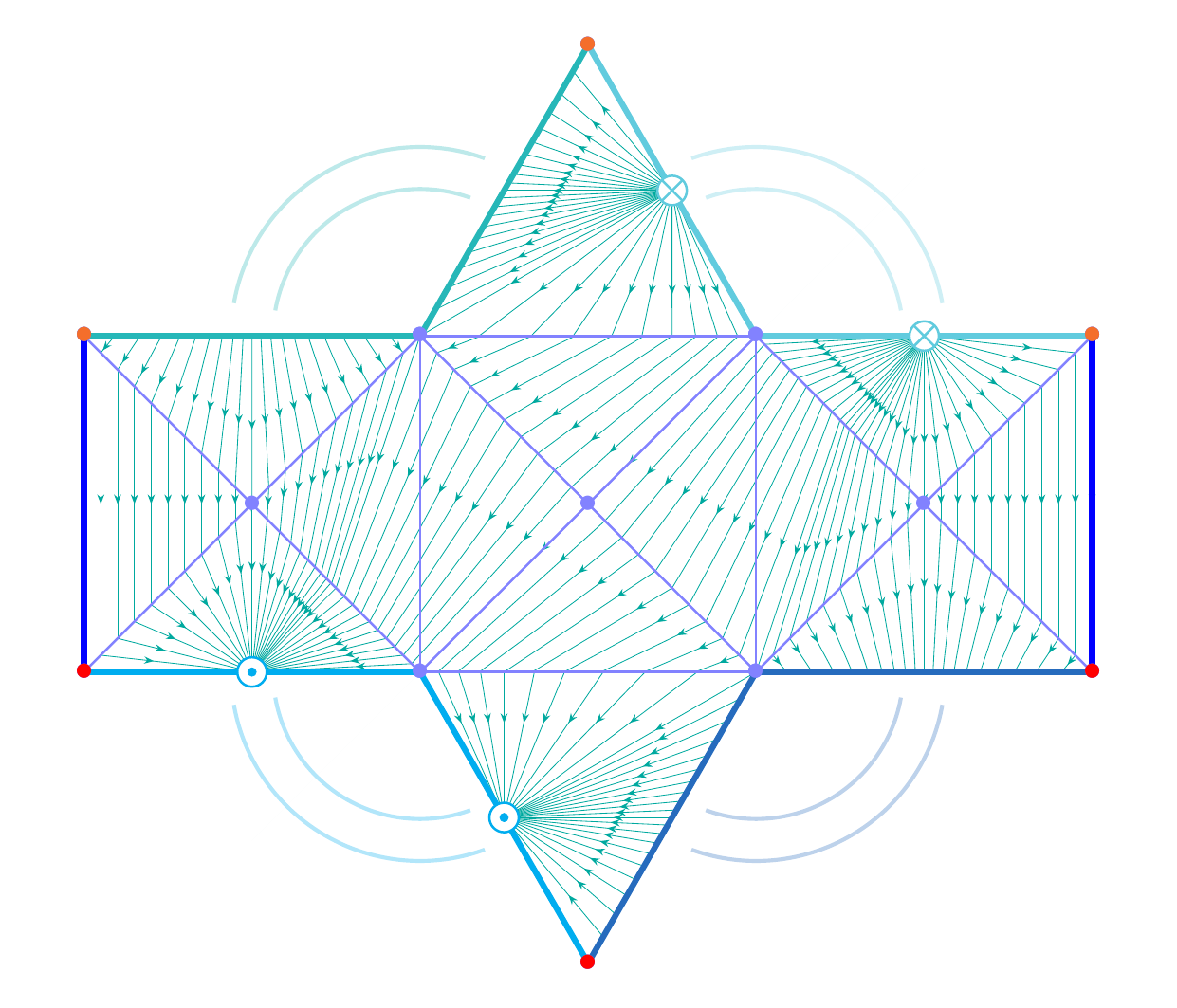}}
\caption{$\xevo$-foliations: vertex cases (left) and edge-midpoint cases (right)} \label{fig:A3-VE}
\end{figure}

\begin{example}\label{EX:A3}
Let $\C=\C(A_3)$ be the cluster category of type $A_3$.
Keep the notation in \Cref{ex:AR}, cf. \Cref{fig:A3-AR-CCx}.

The three pictures on left column of \Cref{fig:A3-VE} illustrate $\sink$-foliations for:
\begin{itemize}
    \item $\sink=P_0$-foliation on the top
    \item $\sink=P_2$-foliation in the middle.
    \item $\sink=I_2$-foliation at the bottom.
\end{itemize}
The three pictures on right column of \Cref{fig:A3-VE} illustrate $\xevo$-foliations for:
\begin{itemize}
    \item $\xevo=\real{ I_1\oplus I_2 }$-foliation on the top.
    \item $\xevo=\real{ P_0\oplus P_2 }$-foliation in the middle.
    \item $\xevo=\real{ J_2\oplus P_2 }$-foliation at the bottom.
\end{itemize}
\end{example}

%=========================================================
%=========================================================
\section{\texorpdfstring{Green mutations as special $\xevo$-evolution flows}
{Green mutations as special X-evolution flows}}\label{sec:green}
%=========================================================
In this section, we show how special $\xevo$-evolution flows induce green mutations.

%=========================================================
\subsection{Green mutation}\label{sec:Setup}
%=========================================================
We will restrict to the case when $\C$ is the cluster category of a Jacobi-finite non-degenerated quiver with potential as our paper mainly works in the categorical setting. However, the main result here only depends on sign-coherence property and tropical duality. Thus, it is straightforward to generalize it in the setting of skew-symmetrizable cluster algebra case. We will explore this generalization (among other things) in a future work.
%=========================================================
\paragraph{\textbf{The categorical settings}}\

Let $\pcts{Y}$ be a chosen cluster in $\C$, $(Q^\pcts{Y},W^\pcts{Y})$ the associated quiver with potential
and $\Gamma_\pcts{Y}=\Gamma(Q^\pcts{Y},W^\pcts{Y})$ the corresponding \emph{Ginzberg dg algebra} cf. \cite[\S~7]{Ke3} for more details.
The cluster category $\C$ can be constructed from $\Gamma_\pcts{Y}$ as its cosingularity category.
More precisely, consider the following categories:
\begin{itemize}
    \item the perfect derived category $\per(\Gamma_\pcts{Y})$ and
    \item the perfectly valued derived category $\pvd(\Gamma_\pcts{Y})$, which is 3-Calabi-Yau
\end{itemize}
and then $\C=\C(\Gamma_\pcts{Y})$ is defined by the short exact sequence (of triangulated categories):
\begin{gather}\label{eq:cosgY}
    0\to\pvd(\Gamma_\pcts{Y})\to\per(\Gamma_\pcts{Y})
        \xrightarrow{\pi_\pcts{Y}}   \C(\Gamma_\pcts{Y})\to0.
\end{gather}

Finally, recall that any simple $\Gamma_\pcts{Y}$-module $S$ is (3-)spherical, with associated spherical twist (functor) $\phi_S$ in the auto-equivalence group $\Aut(\pvd(\Gamma_\pcts{Y}))$.
Denote by $\ST(\Gamma_\pcts{Y})$ the spherical twist group of $\pvd$ generated by $\phi_S$ for the twists of all the simples.

%=========================================================
\paragraph{\textbf{Orientations of exchange graphs}}\

The cluster structure on $\C(\Gamma_\pcts{Y})$, i.e. the cluster exchange graph $\CEG(\C)$, corresponds to (variation of) tilting structures on $\pvd(\Gamma_\pcts{Y})$ and $\per(\Gamma_\pcts{Y})$.
We recall the following notions, cf. \cite[\S~2]{Q5}.
\begin{itemize}
    \item The exchange graph $\EG(\pvd(\Gamma_\pcts{Y}))$ is an oriented graph whose vertices are the hearts (=bounded t-structures) and edges are backward simple tiltings.
    \item The silting exchange graph $\SEG(\per(\Gamma_\pcts{Y}))$ is an oriented graph whose vertices are silting sets and edges are backward mutations.
    \item The oriented cluster exchange graph $\CEG(\C)$, which is obtained from replacing the unoriented edges by 2-cycles in $\uCEG(\C)$.
\end{itemize}
Note that we are using the \textbf{backward} convention here to match with Keller's in \cite{Ke2} (and in \cite{Q12}),
which is the opposite to our usual convention (e.g. in \cite{KQ1,QQ, Q5,KQ2}).
We will use the notation $\EG^\circ/\SEG^\circ$ to denote the principal component of the graph containing the canonical heart $\h_\pcts{Y}$/silting set $\Gamma_\pcts{Y}$.
The initial simple-projective duality extends to a natural isomorphism between oriented graphs (cf. \cite[\S~7]{Ke3}).

\begin{lemma}\label{lem:EG}
There is a natural identification $\iota_\pcts{Y}\colon\EG^\circ(\pvd(\Gamma_\pcts{Y}))\cong\SEG^\circ(\per(\Gamma_\pcts{Y}))$.
\end{lemma}

The relation between these graphs and cluster exchange graph is the following,
which is essentially original due to a unrealised manuscript of Keller-Nicol\'{a}s in the general setting, cf. the $N$-Calabi-Yau acyclic version in \cite[Thm.~8.6]{KQ1}.

\begin{theorem}\label{thm:EG}\cite[Thm.~2.10]{KQ2}
The exchange graph $\EG(\pvd(\Gamma_\pcts{Y}))$ is a $\ST(\Gamma_\pcts{Y})$-covering
\begin{equation}\label{eq:EG}
    \EG(\pvd(\Gamma_\pcts{Y}))/\ST(\Gamma_\pcts{Y})
    \xrightarrow[\pi_\pcts{Y} \circ \iota_\pcts{Y} ]{\cong} \CEG(\C).
\end{equation}
\end{theorem}

The key to prove the theorem above is to find a fundamental domain for the $\ST(\Gamma_\pcts{Y})$-action. More precisely, we have the following.

\begin{theorem}\label{thm:f.d.}\cite[\S~3.1]{Q5}
Let $\pcts{\vv}$ be any cluster in $\uCEG(\C)$ and
$\widetilde{\pcts{\vv}}$ any lifting silting in $\pi_\pcts{Y}^{-1}( \pcts{\vv} )$
with dual heart $\h_{\pcts{\vv}}$.
Then the interval
\begin{gather}\label{eq:f.d.}
    [\h_{\pcts{\vv}}[-1],\h_\pcts{\vv}]
\end{gather}
is a fundamental domain for $\EG(\pvd(\Gamma_\pcts{Y}))/\ST(\Gamma_\pcts{Y})$.
Equivalently, it gives an orientation of $\uCEG(\C)$ with a unique sink $\pcts{\vv}[-1]$ and a unique source $\pcts{\vv}$, called \emph{$\pcts{\vv}$-orientation (of backward simple tilting)}.
\end{theorem}

The green mutation is introduced by Keller,
while we will first give a categorical interpretation based on \cite[Thm.~1.1]{Q12}.

\begin{definition}\label{def:cat-green}\cite{Ke2}
A directed mutation in $\CEG(\C)$ is \emph{green/red} with respect to $\pcts{\vv}$
if it is compatible with the $\pcts{\vv}$-orientation/reverse $\pcts{\vv}$-orientation in \eqref{eq:f.d.}, respectively.
\end{definition}

%=========================================================
\subsection{$G$-matrix, $C$-matrix and tropical duality}\

%=========================================================
\paragraph{\textbf{$g$-Vectors and $c$-vectors}}\

We proceed to recall the notions of $g$-vectors and $c$-vectors, which is more classical from cluster algebra point of view, cf. \cite[\S~5]{Ke3} (and the references therein) for more details on the cluster theory.

Let $\pcts{\vv}$ be a cluster (treating as the initial one) with associated quiver $Q^\pcts{\vv}$.
%whose vertices are labeled by $\{0,\ldots,n-1\}$.
Choose a fixed lifting silting set
\begin{gather}\label{eq:lifting-ss}
    \pi_\pcts{Y}^{-1}( \pcts{\vv} ) \ni \widetilde{\pcts{\vv}}
        =\{ P_\alpha^\pcts{\vv} \}_{\alpha\in Q^\pcts{\vv}_0}
\end{gather}
and let $\h_{\pcts{\vv}}$ be its dual heart with
\begin{gather}\label{eq:lifting-h}
    \Sim \h_\pcts{\vv}=\{ S_\beta^\pcts{\vv} \}_{\beta\in Q^\pcts{\vv}_0}.
\end{gather}
Then $\{ [P_\alpha^\pcts{\vv}] \}$ and $\{ [S_\beta^\pcts{\vv}] \}$ form basis of
the Grothedieck group $\Grot(\per(\Gamma_\pcts{Y}))$ and $\Grot(\pvd(\Gamma_\pcts{Y}))$, respectively.
%Denote by $S_\pcts{Y}^i$ the $i$-th simple in the canonical heart $\h_{\pcts{Y}}$ of $\pvd(\Gamma_\pcts{Y})$ and $P_\pcts{Y}^i$ the $i$-th projective in the silting object/set $\Gamma_\pcts{Y}$ of $\per(\Gamma_\pcts{Y})$,
%for $i\in (Q_\pcts{Y})_0$.
There is a non-degenerate pairing (known as the Euler form)
\begin{gather}\label{eq:pairing}
    \<-,-\>:\Grot(\per(\Gamma_\pcts{Y}))\times\Grot(\pvd(\Gamma_\pcts{Y}))\to\ZZ
\end{gather}
given by
\[
    \<W,U\>=\dim\on{Hom}^\bullet(W,U)=\sum_{d}(-1)^d\dim\on{Hom}(W,U[d]).
\]
So $\Grot(\pvd(\Gamma_\pcts{Y}))$ can be identified with the dual space of $\Grot(\per(\Gamma_\pcts{Y}))$ with respect to the Euler form.

By the simple-projective duality, we have
\begin{gather}\label{SP-Hom-dual}
    \on{Hom}^\bullet(P_\alpha^\pcts{\vv},S_\beta^\pcts{\vv})=\delta_{\alpha\beta}\k,
\end{gather}
which implies that
\begin{gather}\label{eq:SP-dual}
    \<[P_\alpha^\pcts{\vv}],[S_\beta^\pcts{\vv}]\>=\delta_{\alpha\beta}.
\end{gather}

Now take any cluster $\pcts{X}$ with lifting heart $\h_\pcts{X}$ in the fundamental domain \eqref{eq:f.d.}
%[$\h_\pcts{\vv}[-1]$,$\h_\pcts{\vv}$]
%so that there is \Qy{no no no, may not exist!!! even connected, not the point} a backward simple tilting sequence from $\h_\pcts{\vv}$ to $\h_\pcts{\sink}$. Hence, we have
with simples
$
    \Sim \h_\pcts{\sink}=\{ S^\pcts{\sink}_j\}_{j\in Q_0^{\pcts{\sink}}},
$
%where $S^\pcts{\sink}_j$ is the simple obtained by applying the simple tilting sequence to $S^\pcts{\vv}_j$.
and dual silting set
$
    \widetilde{\pcts{X}}=\{ P_i^\pcts{X} \}_{i\in Q_0^{\pcts{\sink}}}.
$

\begin{definition}\label{def:cg-vec}
The $g$-vector $g_{X_i,\pcts{\sink}}^\pcts{\vv}$ of $X_i$ in $\pcts{\sink}$ with respect to $\pcts{\vv}$ is the coordinate (written as a column vector) of $[P_i^\pcts{X}]$ in $\Grot(\per(\Gamma_\pcts{Y}))$ with respect to the basis $\{ [P_\alpha^\pcts{\vv}] \}_{\alpha\in Q^\pcts{\vv}_0}$.
Putting them together, we obtain $G$-matrix $G_{\pcts{X}}^\pcts{\vv}=(g_{X_i,\pcts{\sink}}^\pcts{\vv})_{i\in Q_0^\pcts{\sink}}$ satisfying
\begin{gather}\label{eq:g-matrix}
    \Big ( [P_i^\pcts{\sink}] \Big )_{i\in Q_0^\pcts{\sink}}=
    \Big ( [P_\alpha^\pcts{\vv}] \Big)_{\alpha\in Q_0^\pcts{\vv}} \cdot G_{\pcts{X}}^\pcts{\vv}.
\end{gather}
%Note that the $g$-vector $g_{X_\alpha}^\pcts{\vv}$ is independent of $\pcts{X}$.
%  \item We do not specify the labeling $\alpha \in Q^\pcts{X}_0$
%  since it is independent of the labeling $Q^\pcts{\vv}_0=\{1,\ldots,n\}$.

Dually,
the $c$-vector $c_{X_j,\pcts{X}}^\pcts{\vv}$ of $X_j$ in $\pcts{\sink}$ with respect to $\pcts{\vv}$ is the coordinate (written as a column vector) of $[S_j^\pcts{X}]$ in $\Grot(\pvd(\Gamma_\pcts{Y}))$ with respect to the basis $\{ [S_\beta^\pcts{\vv}] \}_{\beta\in Q^\pcts{\vv}_0}$.
Putting them together, we obtain $C$-matrix $C_{\pcts{X}}^\pcts{\vv}=(c_{X_j,\pcts{X}}^\pcts{\vv})_{j\in Q_0^\pcts{\sink}}$ satisfying
\begin{gather}\label{eq:c-matrix}
    \Big ( [S_j^\pcts{\sink}] \Big)_{j\in Q_0^\pcts{\sink}}=
    \Big ( [S_\beta^\pcts{\sink}] \Big )_{\beta\in Q_0^\pcts{\vv}} \cdot C_{\pcts{X}}^\pcts{\vv}.
\end{gather}
\end{definition}

Note that the $g$-vector $g_{X_i,\pcts{\sink}}^\pcts{\vv}$ is independent of $\pcts{X}$
and hence we will denote it by $g_{X_i}^\pcts{\vv}$ from now on.

%=========================================================
\paragraph{\textbf{Sign-coherence}}\

A major conjecture, known as the Sign-coherence Conjecture about $c$-vector states that
\begin{itemize}
  \item the entries in the $c$-vector $c_{X_j,\pcts{X}}^\pcts{\vv}$ are either all non-positive or all non-negative.
\end{itemize}
In the categorical setting, this is fairly easy.
As $\h_\pcts{X}$ is in \eqref{eq:f.d.}, or equivalently, $\h_\pcts{\vv}[-1]\le \h_\pcts{X} \le \h_\pcts{\vv}$,
we know that any simple $S_j^\pcts{X}$ is either in $\h_\pcts{\vv}[-1]$ or in $\h_\pcts{\vv}$ by \cite[Lem.~3.8]{KQ1}.
Thus we have the following criterion:
\begin{equation}\label{eq:sing-c}
  \begin{cases}
    c_{X_j,\pcts{X}}^\pcts{\vv} \in \ZZ_{\ge0}^n\backslash\{\underline{0}\}
        \quad\Longleftrightarrow\quad S_j^\pcts{X}\in\h_\pcts{\vv},  \\
    c_{X_j,\pcts{X}}^\pcts{\vv} \in \ZZ_{\le0}^n\backslash\{\underline{0}\}
      \quad\Longleftrightarrow\quad  S_j^\pcts{X}\in\h_\pcts{\vv}[-1].
  \end{cases}
\end{equation}

As a result, we can define the green mutation alternatively as follows, cf. \cite[\S5]{Ke2}.
\begin{lemdef}\label{def:cls-green}
We say the mutation of $\pcts{X}$ at $X_j$ (or $X_j$ in $\pcts{X}$, or the vertex $j\in Q^{\pcts{X}}_0$)
with respect to $\pcts{\vv}$, is
\begin{itemize}
  \item \emph{green} if $c_{X_j,\pcts{X}}^\pcts{\vv} \in \ZZ_{\ge0}^n\backslash\{\underline{0}\}$;
  \item \emph{red} if $c_{X_j,\pcts{X}}^\pcts{\vv} \in \ZZ_{\le0}^n\backslash\{\underline{0}\}$.
\end{itemize}
By \eqref{eq:sing-c}, $X_j$ in $\pcts{X}$ is either green or red.
\end{lemdef}

%Note that the identification in \Cref{lem:CEG=int} gives a categorification of the green mutation.

%=========================================================
\paragraph{\textbf{Tropical duality}}\

The tropical duality, originally due to \cite{NZ}, holds in general (for cluster algebras).
In the categorical setting, they are much more straightforward again.

\begin{lemma}\label{eq:T-dual}
The $G$-matrices and $C$-matrices satisfy the following relations (where $(-)^t$ is the transpose of matrices):
\begin{itemize}
  \item $G_{\pcts{\sink}}^\pcts{\vv} \cdot \left( C_{\pcts{\sink}}^\pcts{\vv} \right)^t=1$.
  \item $C_{\pcts{\sink}}^\pcts{\vv} \cdot C^{\pcts{\sink}[1]}_\pcts{\vv} =-1$.
\end{itemize}
\end{lemma}
\begin{proof}
The first equality is a consequence of the simple-projective duality \eqref{eq:SP-dual}
and the one for $\pcts{\sink}$ (i.e. $\<[P_i^\pcts{\sink}],[S_j^\pcts{\sink}]\>=\delta_{ ij }$).

Next, we note that
\[
    \h_\pcts{\vv}[-1]\le \h_\pcts{\sink}\le \h_\pcts{\vv} \Longleftrightarrow
        \h_\pcts{\sink}\le \h_\pcts{\vv}\le \h_\pcts{\sink}[1].
\]
Then $C^{\pcts{\sink}[1]}_\pcts{\vv}$ is defined via
\[
    \Big([S_i^\pcts{\vv}]\Big)_{i\in Q_0^\pcts{\vv}}=
    \Big([S_\beta^\pcts{\sink[1]}]\Big)_{\beta\in Q_0^\pcts{\sink[1]}} \cdot C_{\pcts{\vv}}^\pcts{\sink[1]}.
\]
Since $S_\beta^\pcts{\sink[1]}=S_\beta^\pcts{\sink}[1]$, we have
$([S_\beta^\pcts{\sink[1]}])_{\beta\in Q_0^\pcts{\sink[1]}}=-([S_\beta^\pcts{\sink}])_{\beta\in Q_0^\pcts{\sink}}$.
Comparing with \eqref{eq:c-matrix}, we see that the second equality holds.
\end{proof}

%=========================================================
\subsection{$\xevo$-evolution flows induces green mutation}\label{sec:X-green}\
%=========================================================
%\paragraph{\textbf{Evolution flow induces green mutation}}\

Recall that $\uCEG(\C)$ is the $1$-skeleton of the dual complex of $\CCx(\C)$.
%i.e. each cluster is a top-cell $\pcts{\vv}=\{V_\gamma\}$ and each mutation is a codimension-1 cell
%$ \mut{\vv}{\alpha}=\pcts{\vv}\backslash\{\vv_\alpha\}$.
Consider the downward $\xevo$-evolution flow and $\xevo$-foliation on $\CCx(\C)$
for an arbitrary point $\xevo=\sum_{i} c_i X_i$ in $\CCx(\C)$.
%by \Cref{thm:foliation}.

Given a top-cell $\pcts{\vv}=\{\vv_\beta\}_{\beta\in Q^\pcts{\vv}_0}$,
there is a downward partition \eqref{eq:d-part} of the index set of $\xevo$.
For any point $\ccpt{V}$ in $\pcts{\vv}$, the $\xevo$-downward flow can be expressed as \eqref{eq:quasi-ratio} with all terms in $\pcts{\vv}$.
So we have
\begin{align}\label{eq:b}
    \sum_{\beta\in Q^\pcts{\vv}_0} b_\beta \vv_\beta\colon
    &=\left[
    (\sum_{i_0\in I_0}c_{i_0} \sink_{i_0}+\sum_{j\in J}c_j\real{\ww_j} ) -
    ( \sum_{j\in J}c_j\real{\uu_j}+\sum_{i_1\in I_1}c_{i_1} \sink_{i_1}[1] )
    \right]\\
    &=\dflow{\xevo}(\ccpt{\vv}) + ( \sum_{i_0\in I_0}c_{i_0} - \sum_{i_1\in I_1}c_{i_1} )\ccpt{\vv},
\end{align}
where $b_\beta=b_\beta(\pcts{\vv})$ only depends on the top-cell $\pcts{\vv}$
(instead of the point $\ccpt{\vv}$).
\begin{lemma}\label{lem:cell-crossing-b}
Consider a codimension-1 cell $\mut{\vv}{\alpha}=\pcts{\vv}\backslash\{\vv_\alpha\}$ of $\pcts{\vv}$.
We have the following for the downward $\xevo$-evolution flow:
\begin{itemize}
  \item If $b_\alpha<0$, then there is a cell-crossing from $\pcts{\vv}$ to $\mu_\alpha(\pcts{\vv})$ through $\mut{\vv}{\alpha}$.
  \item If $b_\alpha=0$, then there is no cell-crossing on $\mut{\vv}{\alpha}$.
  \item If $b_\alpha>0$, then there is a cell-crossing from $\mu_\alpha(\pcts{\vv})$ to $\pcts{\vv}$ through $\mut{\vv}{\alpha}$.
\end{itemize}
%\begin{equation}\label{eq:b-direction}
%  \begin{cases}
%    b_\alpha(\pcts{\vv})=0\iff
%    \forall\text{ $\xevo$-leaf $L$ intersects $\mut{\vv}{\alpha}$ nontrivially.}\\
%    b_\alpha(\pcts{\vv})<0\iff
%    \exists\text{ a cell-crossing on the wall $\mut{\vv}{\alpha}$ from $\pcts{\vv}$ to $\mu_\alpha(\pcts{\vv})$}.\\
%    b_\alpha(\pcts{\vv})>0\iff
%    \exists\text{ a cell-crossing on the wall $\mut{\vv}{\alpha}$ from $\mu_\alpha(\pcts{\vv})$ to $\pcts{\vv}$}.\\
%  \end{cases}
%\end{equation}
\end{lemma}

\begin{proof}
To determine the direction of the cell-crossing, we will take $\ccpt{\vv}\in\mut{\vv}{\alpha}$ for $\dflow{\xevo}$.
As $\vv_\alpha$ does not occur in $\ccpt{\vv}\in\mut{\vv}{\alpha}$, the coefficient of $\vv_\alpha$ in $\dflow{\xevo}$ equals $b_\alpha$.
Hence the lemma follows, cf. \Cref{fig:b-direction}.
\end{proof}

\begin{figure}
\begin{tikzpicture}[scale=.7]

%draw the leaves
\foreach \i in {-1.6,-2,-2.4,-2.8,-3.2,-3.6}
{
\begin{scope}[shift={(\i-1.5,0)}]
\draw[thick, >=stealth,->-=.5, Emerald](60:3) to (0,0);
\draw[thick, >=stealth,->-=.5, Emerald](0,0) to (-110:3);
\end{scope}
\begin{scope}[shift={(-\i+1.5,0)},xscale=-1]
\draw[thick, >=stealth,-<-=.5, Emerald](60:3) to (0,0);
\draw[thick, >=stealth,-<-=.5, Emerald](0,0) to (-110:3);
\end{scope}
}
%balpha
\draw (-2,-3) node {$b_\alpha<0$}
      (2,-3) node {$b_\alpha>0$}
;
%four direction
\draw[ultra thick,-stealth]
    (-0.5,0) to (-7,0) node[above, font=\large]{$\mut{\vv}{\alpha}$};
\draw[ultra thick,-stealth]
    (0.5,0) to (7,0) node[above, font=\large]{$\mut{\vv}{\alpha}$};
\draw[thick, -stealth](0,-4) to (0,4);
\node at (0,4.5) {${\vv_\alpha}$};
%\node[font=\large] at (0,-4.5) {$\vv_\alpha^\sharp$};
\end{tikzpicture}
\caption{The directions of cell-crossings on $\mut{\vv}{\alpha}$}
\label{fig:b-direction}
\end{figure}

\begin{remark}
In \Cref{EX:A3}, we have $b_\alpha=0$ for
\begin{itemize}
    \item $\mut{\vv}{\alpha}=P_1\oplus I_1$ in the $P_0$-foliation,
    \item $\mut{\vv}{\alpha}=P_1\oplus I_1$ in the $\real{I_1\oplus I_2}$-foliation,
    \item $\mut{\vv}{\alpha}=P_2\oplus I_2$ in the $\real{P_0\oplus P_2}$-foliation,
\end{itemize}
In those cases, the $\xevo$-evolution flows in the neighbor top-cells are `parallel' to the codim-1 cell $\mut{\vv}{\alpha}$ (which is in fact a local $\xevo$-leaf itself).
\end{remark}

%\begin{definition}
%the (downward) $\xevo$-foliation induces an (partial) orientation on $\uCEG(\C)$
%by endowing each codimension-1 cell-wall $\mut{\vv}{\alpha}$ with the direction of the cell-crossing.
%\end{definition}

\begin{theorem}\label{thm:green}
Fix $\xevo=\sum_{i=0}^{n-1} c_iX_i$ be a generic point/interior point in a cluster $\pcts{\sink}$ (i.e. $c_i\ne0$).
Then the downward $\xevo$-foliation induces an orientation on $\uCEG(\C)$ (by cell-crossing),
which coincides with the $\pcts{X}[1]$-orientation (i.e. the green mutation with respect to $\pcts{\sink}[1]$, cf. \Cref{def:cat-green}).
\end{theorem}
\begin{proof}
By \Cref{lem:cell-crossing-b}, we only need to prove that for any top-cell $\pcts{\vv}$:
\begin{itemize}
    \item $b_\alpha(\pcts{\vv})<0 \iff \vv_\alpha$ in $\pcts{\vv}$ is green with respect to $\pcts{\sink}[1]$.
    \item $b_\alpha(\pcts{\vv})>0 \iff \vv_\alpha$ in $\pcts{\vv}$ is red with respect to $\pcts{\sink}[1]$.
\end{itemize}
Note that these imply $b_\alpha(\pcts{\vv})\ne0$ in particular.

Consider all the $X_i$-evolving triangles for $\pcts{\vv}$
\begin{gather}\label{eq:evo-i}
    \sink_i \xrightarrow{g} \ww_i \to \uu_i \xrightarrow{f} \sink_i[1],
\end{gather}
where
\[
    \ww_i=\bigoplus_{\beta\in Q^\pcts{\vv}_0} V_\beta^{\oplus w_i^\beta} \quad\text{and}\quad
    \uu_i=\bigoplus_{\beta\in Q^\pcts{\vv}_0} V_\beta^{\oplus u_i^\beta}
\]
for $i\in Q^\pcts{\sink}_0=\{1,\ldots,n\}$, where $w_i^\beta, u_i^\beta\in \ZZ_{\ge0}$.
Note that we do not divide $Q^\pcts{\sink}_0$ into \eqref{eq:d-part} here.
Then \eqref{eq:b} becomes
\begin{equation}\label{eq:===}
\begin{array}{rcl}
  \displaystyle\sum_{\beta\in Q^\pcts{\vv}_0} b_\beta \vv_\beta
    &=&\displaystyle \sum_{i\in Q^\pcts{\sink}_0}
    \Big ( \sum_{\beta\in Q^\pcts{\vv}_0}
        \frac{w^\beta_i \vv_\beta}{\sum_{\gamma\in Q^\pcts{\vv}_0}w^\gamma_i}
    -\sum_{\beta\in Q^\pcts{\vv}_0}
        \frac{u^\beta_i \vv_\beta}{\sum_{\gamma\in Q^\pcts{\vv}_0}u^\gamma_i}
    \Big )
    \\
    &=&\displaystyle
    \sum_{\beta\in Q^\pcts{\vv}_0} \vv_\beta \sum_{i\in Q^\pcts{\sink}_0}
    \Big (
    \frac{w^\beta_i}{\sum_{\gamma\in Q^\pcts{\vv}_0}w^\gamma_i} -  \frac{u^\beta_i}{\sum_{\gamma\in Q^\pcts{\vv}_0}u^\gamma_i}
    \Big ).
\end{array}
\end{equation}
Comparing the coefficients of $\beta\in Q^\pcts{\vv}_0$, we have
\begin{gather}\label{eq:any-b}
    b_\beta=\sum_{i\in Q^\pcts{\sink}_0}
    \Big (
    \frac{w^\beta_i}{\sum_{\gamma\in Q^\pcts{\vv}_0}w^\gamma_i} -  \frac{u^\beta_i}{\sum_{\gamma\in Q^\pcts{\vv}_0}u^\gamma_i}
    \Big ).
\end{gather}

Now take a lifting heart $\h_\pcts{\vv}$ of $\pcts{\vv}$
in $[\h_\pcts{\sink},\h_\pcts{\sink}[1]]$, or equivalently,
$\h_\pcts{\sink}\in[\h_\pcts{\vv}[-1],\h_\pcts{\vv}]$.
By \Cref{def:cg-vec}, the $g$-vectors for $X_i$ with respect to $\pcts{\vv}$ are derived from \eqref{eq:evo-i} as follows.
\[
g^\pcts{\vv}_{X_i}=
\left(
\begin{array}{cc}
     w^{\beta_1}_i-u^{\beta_1}_i  \\
     w^{\beta_2}_i-u^{\beta_2}_i  \\
     \ldots \\
     w^{\beta_{n}}_{i}-u^{\beta_{n}}_{i}
\end{array}
\right),
\]
By the dualities in \Cref{eq:T-dual}, we have
\[
    (G^\pcts{\vv}_\pcts{\sink})^t=(C^\pcts{\vv}_\pcts{\sink})^{-1}=-C^{\pcts{\sink}[1]}_\pcts{\vv}.
\]
Hence the $c$-vectors of $\vv_\alpha$ in $\pcts{\vv}$ with respect to $\pcts{\sink}[1]$ are
\[
    c^{\pcts{\sink}[1]}_{V_\alpha,\pcts{V}}=
\left(
\begin{array}{cc}
     u^\alpha_1-w^\alpha_1  \\
     u^\alpha_2-w^\alpha_2  \\
     \ldots \\
     u^\alpha_n-w^\alpha_n
\end{array}
\right).
\]

Note that $\Add(\ww_i)\cap\Add(\uu_i)=0$ implies that one of $u^\alpha_i$ and $w^\alpha_i$ is zero. Thus
\[\begin{cases}
  u^\alpha_i-w^\alpha_i>0 \Longleftrightarrow (u^\alpha_i>0 \;\&\; w^\alpha_i=0) ,\\
  u^\alpha_i-w^\alpha_i=0 \Longleftrightarrow (u^\alpha_i=0 \;\&\; w^\alpha_i=0) ,\\
  u^\alpha_i-w^\alpha_i<0 \Longleftrightarrow (u^\alpha_i=0 \;\&\; w^\alpha_i>0) .
\end{cases}\]
Therefore
\begin{equation}
\begin{cases}
  c^{\pcts{\sink}[1]}_{\vv_\alpha,\pcts{\vv}}\in\ZZ_{\ge0}^n\backslash\{\underline{0}\}
    \;\xLongleftrightarrow{}\;
  \Big( \forall~i\in Q^\pcts{\sink}_0~\text{s.t.}~ u^\alpha_i\ge0=w^\alpha_i \Big)
    \;\&\; \sum_{i\in Q^\pcts{\sink}_0}  u^\alpha_{i}>0,\\
  c^{\pcts{\sink}[1]}_{\vv_\alpha,\pcts{\vv}}\in\ZZ_{\le0}^n\backslash\{\underline{0}\}
    \;\xLongleftrightarrow{}\;
  \Big( \forall~i\in Q^\pcts{\sink}_0~\text{s.t.}~ u^\alpha_i=0\le w^\alpha_i \Big)
    \;\&\; \sum_{i\in Q^\pcts{\sink}_0}  w^\alpha_{i}>0,
\end{cases}
\end{equation}
Comparing with \eqref{eq:any-b}, we have
\[
\begin{cases}
  c^{\pcts{\sink}[1]}_{\vv_\alpha,\pcts{\vv}}\in\ZZ_{\ge0}^n\backslash\{\underline{0}\}
    \;\xLongleftrightarrow{}\; b_\alpha<0,\\
  c^{\pcts{\sink}[1]}_{\vv_\alpha,\pcts{\vv}}\in\ZZ_{\le0}^n\backslash\{\underline{0}\}
    \;\xLongleftrightarrow{}\; b_\alpha>0.
\end{cases}
\]
Finally, the sign-coherence property of $c$-vectors implies the claim at the beginning of the proof and we are done.
%we have the following equivalences as required.
%\begin{align*}
%  b_\alpha<0 \;(\text{resp. }b_\alpha>0)
%  &\xLongleftrightarrow{}
%    \exists~i\in Q^\pcts{\sink}_0~\text{s.t.}~u^\alpha_i>0
%    \;(\text{resp. }u^\alpha_i<0)\\
%  &\xLongleftrightarrow{(*)}
%    \exists~i\in Q^\pcts{\sink}_0~\text{s.t.}~u^\alpha_i-w^\alpha_i>0
%    \;(\text{resp. }u^\alpha_i-w^\alpha_i<0)\\
%  &\xLongleftrightarrow{(\star)}
%    c^{\pcts{\sink}[1]}_{\vv_\alpha,\pcts{\vv}}\in\ZZ_{\ge0}^n\backslash\{\underline{0}\}
%    \;(\text{resp. } c^{\pcts{\sink}[1]}_{\vv_\alpha,\pcts{\vv}}\in\ZZ_{\le0}^n\backslash\{\underline{0}\})\\
%  &\xLongleftrightarrow{ \textcolor{white}{*} }
%    {\vv_\alpha}~\text{is green (resp. red) in $\pcts{\vv}$ with respect to}~\pcts{\sink}[1],
%%   b_\alpha>0
%%   &\xLongleftrightarrow{} \exists~i\in Q^\pcts{\sink}_0~\text{s.t.}~w^\alpha_i>0\\
%%   &\xLongleftrightarrow{(*)}
%%   c^{\pcts{\sink}[1]}_{\vv_\alpha,\pcts{\vv}}\in\ZZ_{\le0}^n\backslash\{\underline{0}\}\\
%%   &\xLongleftrightarrow{} {\vv_\alpha}~\text{is red in $\pcts{\vv}$ with respect to}~\pcts{\sink}[1],
%\end{align*}
%where $(*)$ follows from the fact that $\Add(\ww_i)\cap\Add(\uu_i)=0$, so that one of $u^\alpha_i, w^\alpha_i$ is zero; $(\star)$ follows from the sign-coherence property of $c$-vectors.
%This completes the proof.
\end{proof}

%\iffalse
%\begin{itemize}
%    \item $ b_\alpha<0
%    \iff\exists~i\in Q^\pcts{\sink}_0~\text{s.t.}~u^\alpha_i>0
%    \xLongleftrightarrow{(*)} c^{\pcts{\sink}[1]}_{\vv_\alpha,\pcts{\vv}}\in\ZZ_{\ge0}^n\backslash\{\underline{0}\}$
%
%    $\qquad\iff {\vv_\alpha}~\text{is green in $\pcts{\vv}$ with respect to}~\pcts{\sink}[1],$
%    \item $ b_\alpha>0
%    \iff\exists~i\in Q^\pcts{\sink}_0~\text{s.t.}~w^\alpha_i>0
%    \xLongleftrightarrow{(*)} c^{\pcts{\sink}[1]}_{\vv_\alpha,\pcts{\vv}}\in\ZZ_{\le0}^n\backslash\{\underline{0}\}$
%
%    $\qquad\iff {\vv_\alpha}~\text{is red in $\pcts{\vv}$ with respect to}~\pcts{\sink}[1],$
%\end{itemize}
%\fi

%=========================================================
\subsection{Showcase}
%=========================================================

\begin{example}\label{EX:A3+}[Inducing green mutation]
We continue with \Cref{EX:A3}.
\begin{figure}\centering
\vskip -.5cm
\makebox[\textwidth][c]{
  \includegraphics[width=9cm]{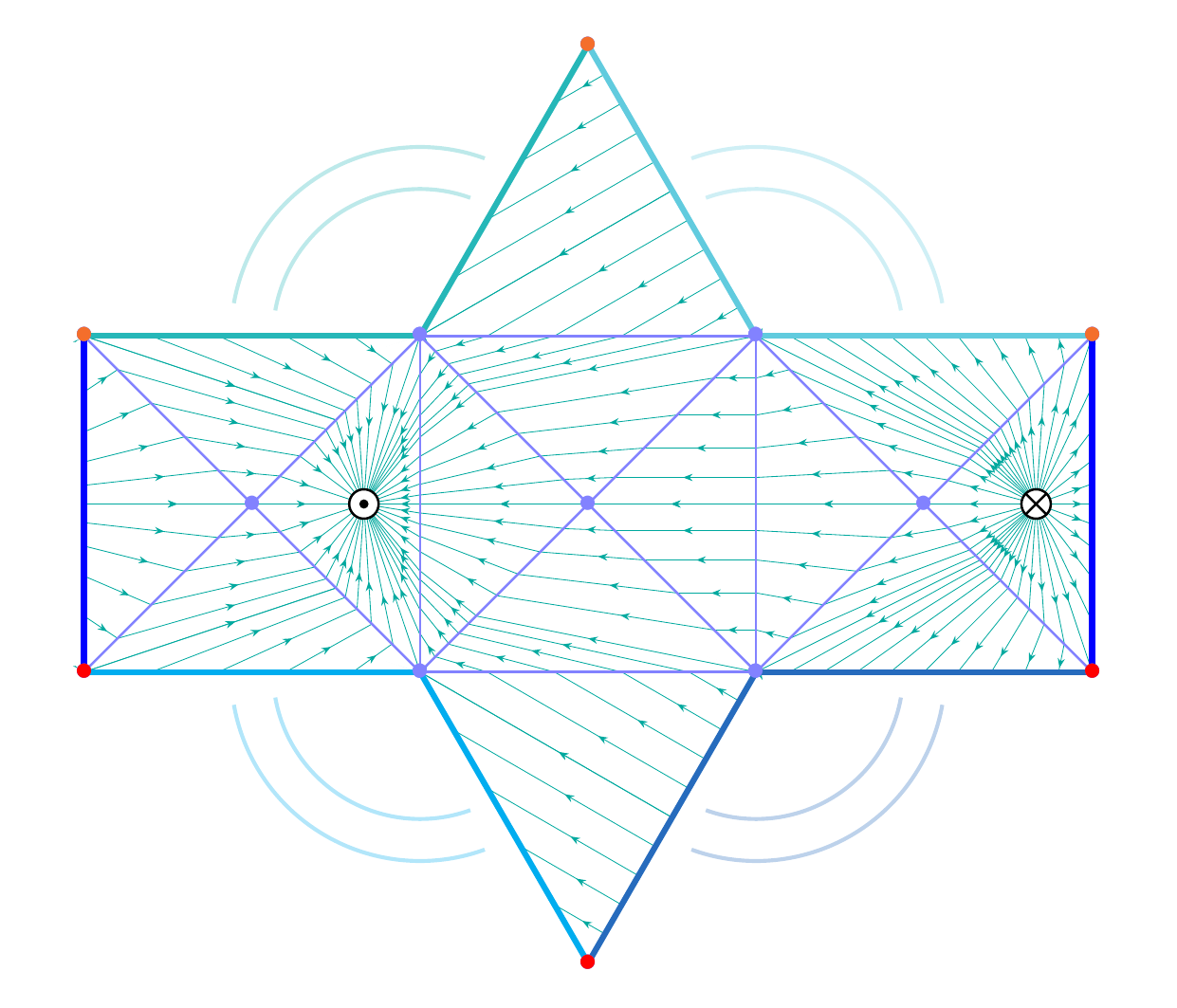}
\hskip -.7cm
  \includegraphics[width=9cm]{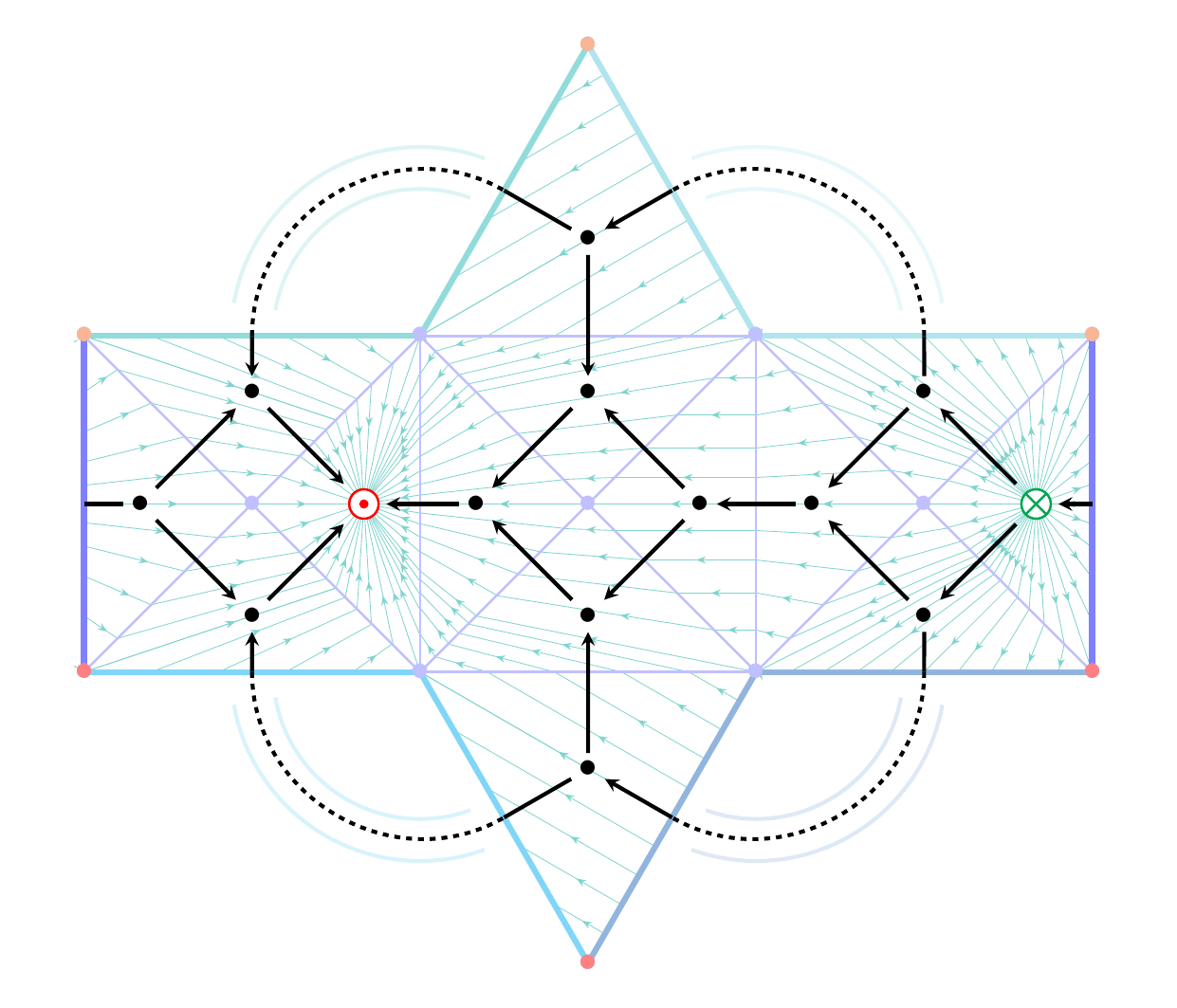}}
\vskip -.6cm
\makebox[\textwidth][c]{\qquad\qquad
  \includegraphics[width=9cm]{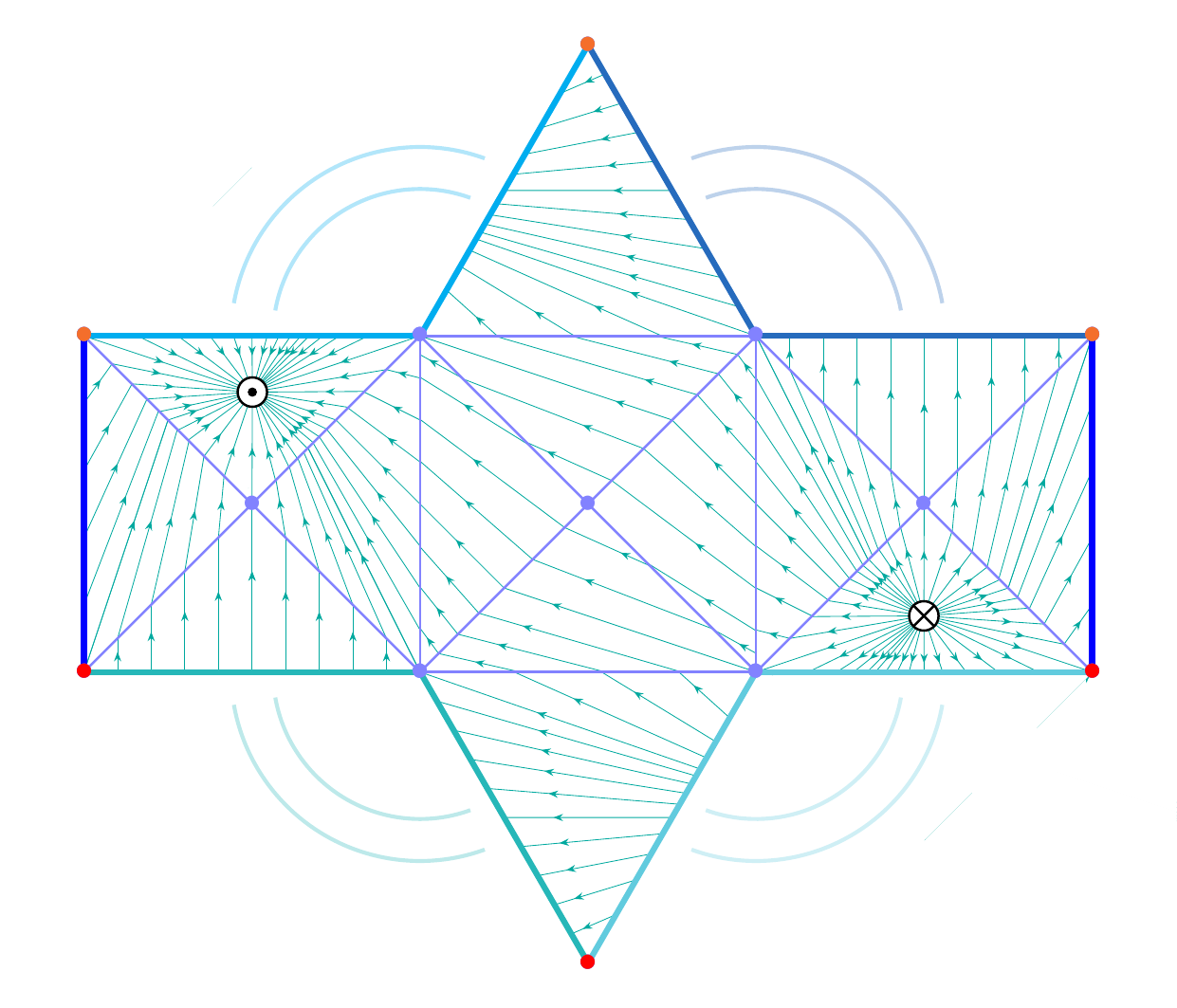}
\hskip -.7cm
  \includegraphics[width=9cm]{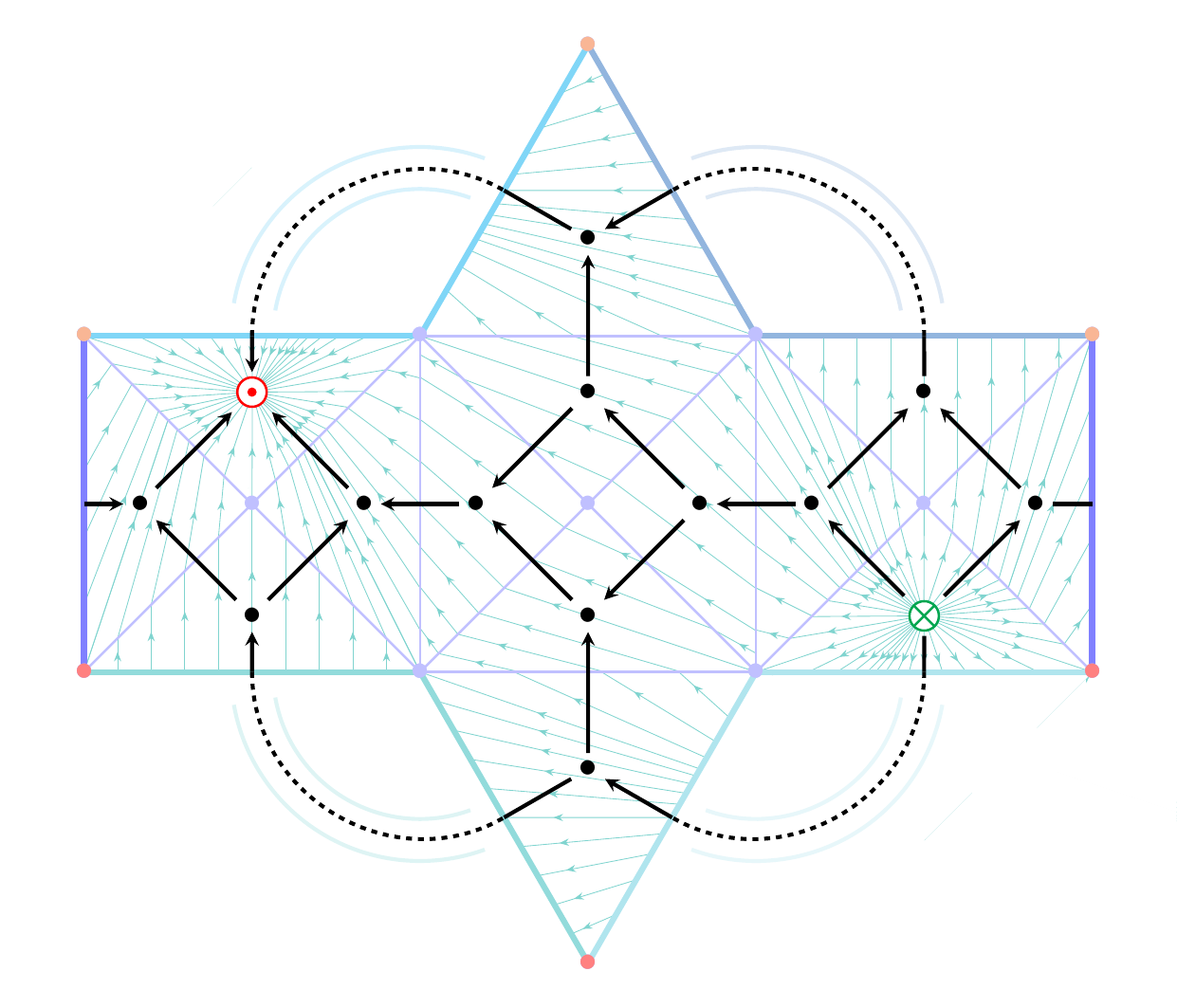}}
\vskip -.6cm
\makebox[\textwidth][c]{
  \includegraphics[width=9cm]{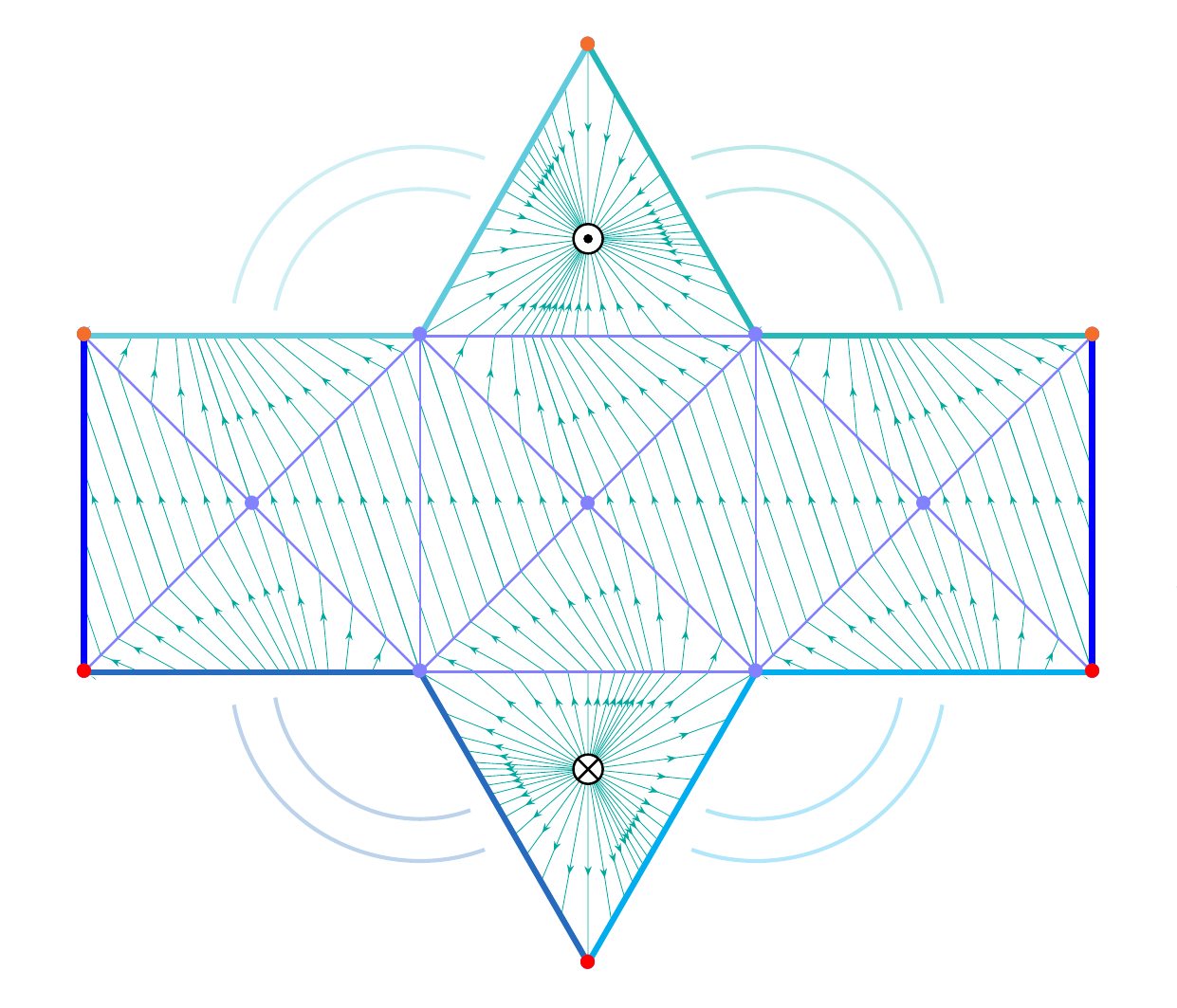}
\hskip -.7cm
  \includegraphics[width=9cm]{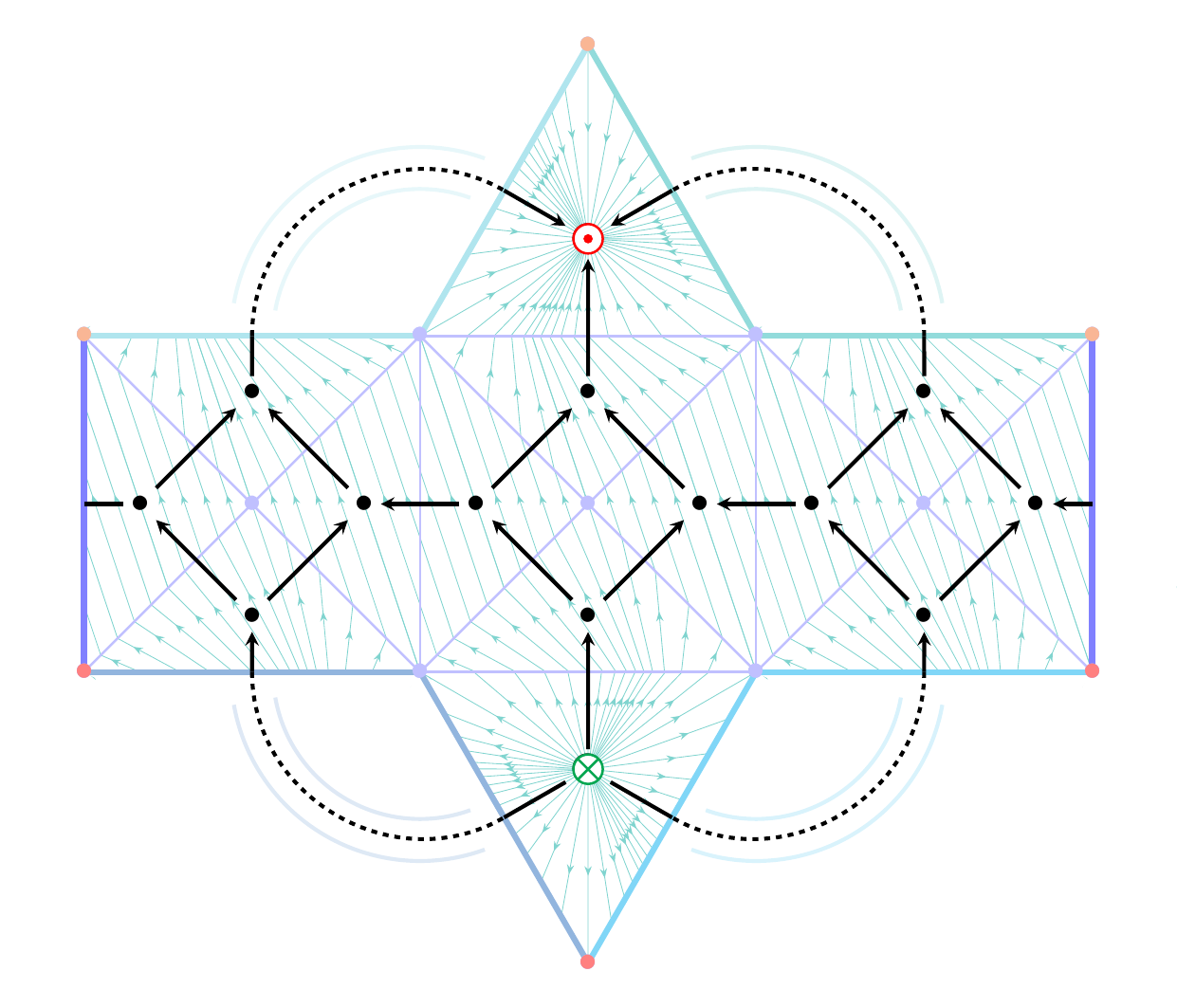}}
\caption{$\xevo$-foliations: face-center cases and induced orientations on $\uCEG(A_3)$} \label{fig:A3-F-F*}
\end{figure}

The three pictures on left column of \Cref{fig:A3-F-F*} illustrate $\xevo$-foliations for:
\begin{itemize}
    \item $\xevo=\real{ \pcto{P} }$-foliation on the top, for the initial cluster of $Q$:
      \[
        \pcto{P}=\k Q=P_0\oplus P_1\oplus P_2.
      \]
    \item $\xevo=\real{ \pcto{\uu} }$-foliation in the middle, for the cluster
      \[
        \pcto{\uu}=\mu_{P_2}(\pcto{P})=P_0\oplus P_1\oplus J_1.
      \]
    \item $\xevo=\real{ \pcto{\ww} }$-foliation at the bottom, for the cluster
      \[
        \pcto{\ww}=\mu_{P_0}(\pcto{\uu})=I_1\oplus P_1\oplus J_1.
      \]
\end{itemize}
%\begin{remark}[Orientations of cluster exchange graphs]\
The three pictures on right column of \Cref{fig:A3-F-F*} are
the induced orientations of the cluster exchange graph $\uCEG(A_3)$,
which indeed coincide with the orientation induced from intervals of hearts in the 3-Calabi-Yau category.

Note that the middle and bottom pictures on the right column of \Cref{fig:A3-F-F*} are precisely the ones in \Cref{fig:KQ1-Fig5},
which is taken from \cite[Fig.~5]{KQ1} (but with opposite orientation convention; see there for more details).
This figure shows how the orientations change when the source/sink hearts are tilted.
%The orientations there are defined by the forward simple tilting of the corresponding hearts
%in the 3-Calabi-Yau category associated to the quiver $A_3$.
%So we see that the downward/upward evolution flows match with
%the backward/forward simple tilting of hearts in the $A_3$ case,
%and we expect this is true in general.
%The mutation from the orientation $\CEG(A_3,\pcto{W})$ to the orientation $\CEG(A_3,\pcto{U})$ is described in \cite[\S~9]{KQ1},
%by reversing a set of `parallel' arrows (in red) and identifying two remaining parts
%(in various colorway, cf. explanation there).
%\end{remark}
\end{example}

\begin{figure}[hbt]\centering\makebox[\textwidth][c]{
\begin{tikzpicture}[scale=0.8,
  arrow/.style={<-,>=stealth,thick},
  equalto/.style={double,double distance=2pt},
  mapto/.style={|->},
  c-before/.style={cyan},
  c-after/.style={red!15!blue!40!green}, % prev. \dexc
  c-cross/.style={red},
  c-div/.style={orange},
  c-fixed/.style={black}]
% coordinates
\foreach \n/\a\b in {1/1/6, 2/1/4, 3/2/7, 4/2/5, 5/3.25/5, 6/4/7.75,
  7/4/6.5, 8/4/3.5, 9/4/2, 10/4.75/5, 11/6/7, 12/6/5, 13/7/6, 14/7/3.75}
 \coordinate (X\n) at (\a,\b);
% x vertices
\draw[c-before] (X3) node (x3) {$\vsink$};
\draw[c-fixed] (X14) node (x14) {$\vsource$};
\foreach \n in {1,2,4,5,8,9}
  \draw[c-before] (X\n) node (x\n) {$\vertx$};
\foreach \n/\x/\y in {6,7,10,11,12,13}
  \draw[c-fixed] (X\n) node (x\n) {$\vertx$};
% x arrows
\draw[c-cross] (x1) edge[arrow,dashed] (x13);
\foreach \t/\h in {3/6, 5/7, 8/10, 9/14}
  \draw[c-cross] (x\t) edge[arrow] (x\h);
\foreach \t/\h in {1/2, 2/9, 3/1, 3/4, 4/5, 4/2, 5/8, 8/9}
  \draw[c-before] (x\t) edge[arrow] (x\h);
\foreach \t/\h in {6/7, 6/11, 7/10, 10/12, 11/12, 11/13, 12/14, 13/14}
  \draw[c-fixed] (x\t) edge[arrow] (x\h);
% y vertices
\draw[c-fixed] (X6)++(3,-6) node (y6) {$\vsink$};
\draw[c-after] (X9)++(3,-6) node (y9) {$\vsource$};
\foreach \n in {1,2,3,4,5,8}
  \draw[c-after] (X\n)++(3,-6) node (y\n) {$\vertx$};
\foreach \n in {7,10,11,12,13,14}
  \draw[c-fixed] (X\n)++(3,-6) node (y\n) {$\vertx$};
% y arrows
\draw[c-cross] (y13) edge[arrow,dashed] (y1); % reversed
\foreach \t/\h in {3/6, 5/7, 8/10, 9/14}
  \draw[c-cross] (y\h) edge[arrow] (y\t); % reversed
\foreach \t/\h in {1/2, 2/9, 3/1, 3/4, 4/5, 4/2, 5/8, 8/9}
  \draw[c-after] (y\t) edge[arrow] (y\h);
\foreach \t/\h in {6/7, 6/11, 7/10, 10/12, 11/12, 11/13, 12/14, 13/14}
  \draw[c-fixed] (y\t) edge[arrow] (y\h);
% divider
\draw[c-div] (2.5,8) edge[dotted,thick] (8.5,-4);
% side bits
\draw (9.2, 4.8) node (R1) {$\EG(A_3,\pcto{\uu})^\vsource_{P_0}$};
\draw (11, 2) node (R2) {$\EG(A_3,\pcto{\ww})^\vsink_{I_1}$};
\draw[c-before] (0.4, 1.8) node (L1) {$\EG(A_3,\pcto{\uu})^\vsink_{P_0}$};
\draw[c-after] (2, -1) node (L2) {$\EG(A_3,\pcto{\ww})^\vsource_{I_1}\quad $};
\draw (R1) edge[equalto] (R2);
\draw (L1) edge[equalto] (L2);
\end{tikzpicture}
}
\caption{Mutation of orientations of $\uCEG(A_3)$}\label{fig:KQ1-Fig5}
\end{figure}

\begin{remark}[Realizing cluster complexes in g-fans]
Consider the dual lattices
\[\begin{cases}
    M=\Grot(\pvd(\Gamma_\pcts{Y})), \\
    N=\Grot(\per(\Gamma_\pcts{Y})).
  \end{cases}
\]
with pairing \eqref{eq:pairing}.
Let $M_\RR=M\otimes_{\ZZ}\RR$ be the $\RR$-extension of $M$, %and $N_\RR=\Hom_\RR(M,\RR)$ its dual.
which admits a basis
$
\{[P^\pcts{\vv}_\alpha]\}_{\alpha\in Q^\pcts{\vv}_0}
$
with respect to the initial cluster $\pcts{\vv}$.
There exists an piecewise linear embedding
\begin{equation}\label{eq:real-g}
\begin{array}{rccc}
  g^\pcts{\vv}\colon
    & \CCx(\C) & \longrightarrow & M_\RR  \\
    & \ccpt{Y}=\sum_{i}c_i Y_i & \longmapsto & \sum_{i} c_i g^\pcts{\vv}_{Y_i}.
\end{array}
\end{equation}
%Its dual basis in $N_\RR$ is $\{[S^\pcts{\vv}_\beta]\}_{\beta\in Q^\pcts{\vv}_0}$by \eqref{eq:SP-dual}.
%Note that the $g$-vectors live in $M_\RR$.
%In particular,
Under such an embedding, a partial cluster $\pcts{Y}=\{ Y_i \}_{i=1}^k$ becomes a $g$-cone in $M_\RR$:
\begin{gather}
    \con( \pcts{Y} ) \colon=
    \sum_{i=1}^k \RR_{\ge0} \cdot g^\pcts{\vv}_{Y_i}.
\end{gather}
In particular, clusters become \emph{$G$-cones} (top dimensional cones) and
the cluster complex $\CCx(\C)$ is realized in the \emph{$g$-fan}, which is the union of $g$-cones:
\begin{gather}
    \fan^\pcts{\vv}(\C)\colon=
        \{ \con(\pcts{Y}) \mid \pcts{Y}~\text{is a partial cluster in}~\CCx(\C)\}.
\end{gather}
See \cite{Nak} for more details about $g$-cones and $g$-fans.
\end{remark}

\begin{figure}[thb]\centering
\makebox[\textwidth][c]{
\includegraphics[scale=.4]{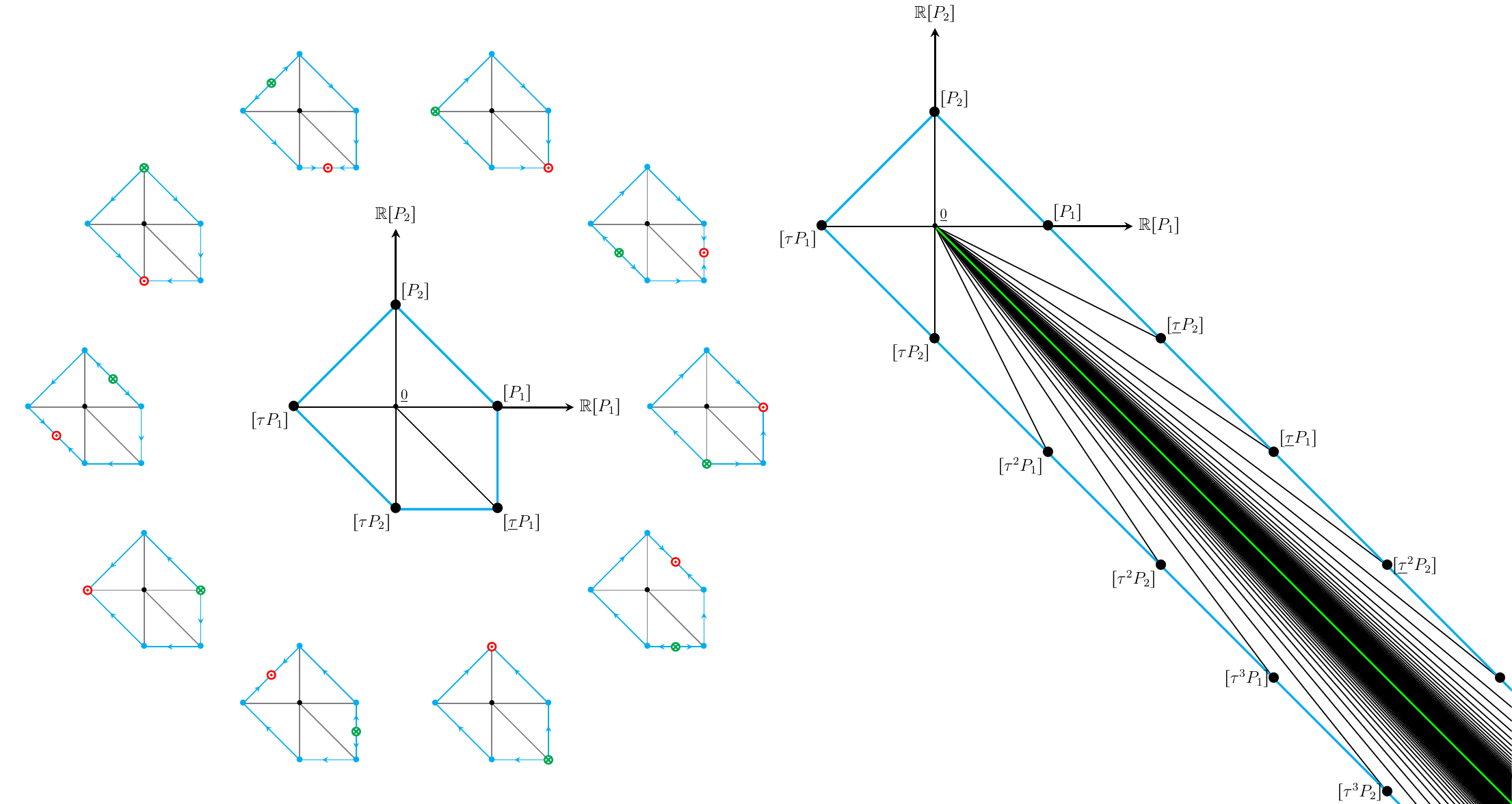}\qquad}
\caption{The flows on $\CCx(A_2)$ and the embedding of $\CCx(\EucA{1}{1})$ into $\fan^\pcts{P}(\EucA{1}{1})$}
\label{fig:A2-cycle}
\end{figure}

\begin{example}[g-fans]
Let $Q$ be an acyclic quiver and $P_i$ its $i$-th projective.
We have
\begin{itemize}
    \item The middle picture on the left of \Cref{fig:A2-cycle} show 
    the embedding \eqref{eq:real-g} for the $A_2$ case, where $\CCx(A_2)$ is in cyan.
    \item The 10 surrounding pictures on the left of \Cref{fig:A2-cycle} demonstrate 10 $\xevo$-evolutions, where $\xevo$ (the green sinks)and $\xevo[1]$ (the red sources) are moving clockwise around.
    \item The picture on the right of \Cref{fig:A2-cycle}, and \Cref{fig:A3-gfan}
    show the embedding \eqref{eq:real-g} for the $\widetilde{A_{1,1}}$ and the $A_3$ cases, respectively.
 \end{itemize}
 \end{example}

\vskip -.5cm
\begin{figure}[bht]\centering
\makebox[\textwidth][c]{
\includegraphics[scale=.45]{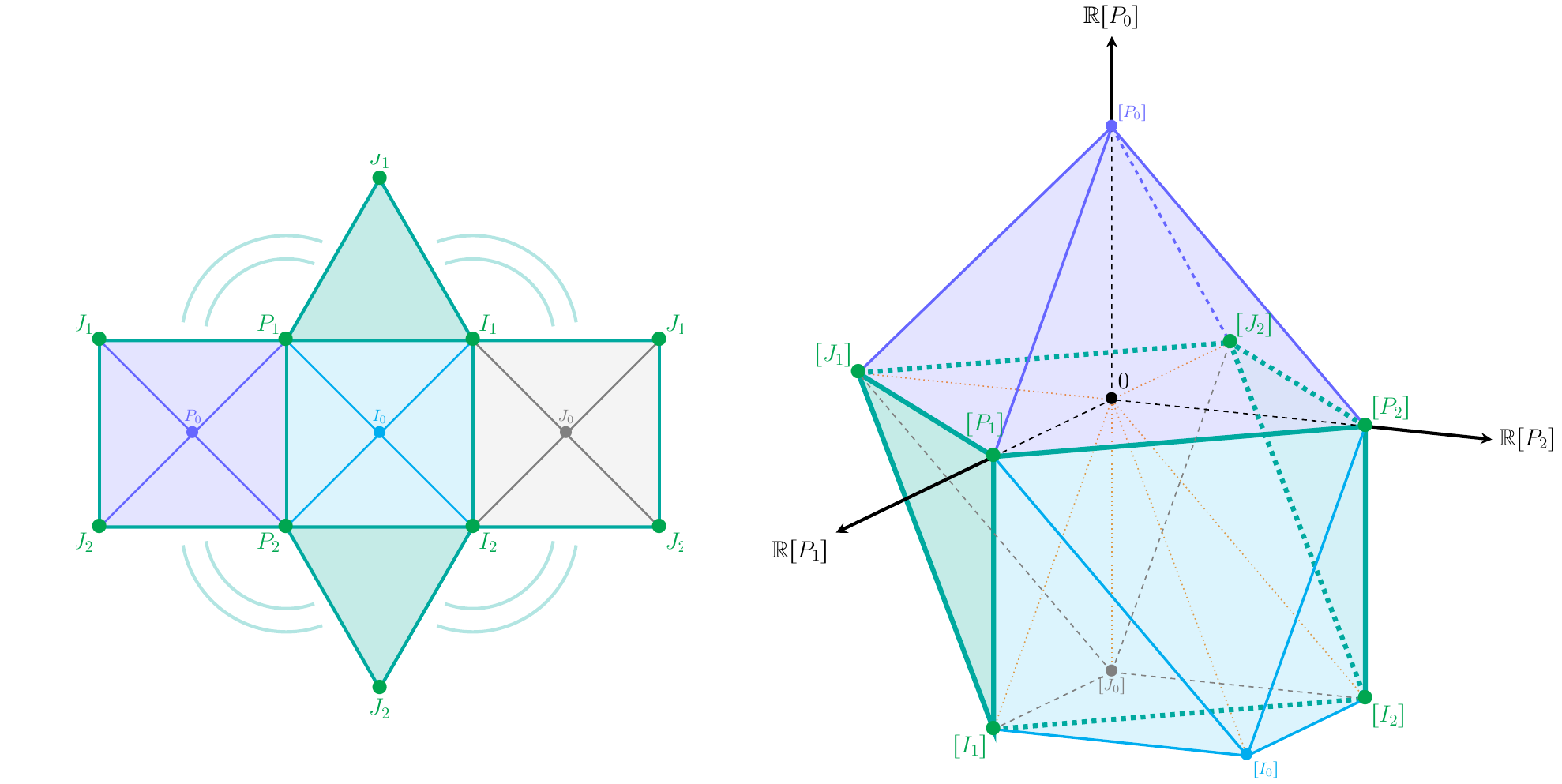}}
\vskip -.2cm
\caption{The cluster complex $\CCx(A_3)$(left) embeds into the $\fan^\pcts{P}(A_3)$(right)}
\label{fig:A3-gfan}
\end{figure}

%=========================================================
%=========================================================
\section{Case study and examples}\label{sec:case}
%=========================================================
%=========================================================

%=========================================================
\subsection{Reduction in quiver categories}\label{sec:redQ}\
%=========================================================

Let $Q$ be an acyclic quiver and $\CQ$ its cluster category as in \Cref{sec:qc}.
Take an rigid indecomposable $\sink$ in the $\tau$-orbit of a projective object $P_0$ corresponding to the vertex $0$.
Recall the CY-reduction in \Cref{sec:Red}.
Since the reduction functor $\Red_{\sink}$ (resp. $\Red_{\source}$) preserves (partial) cluster,
$\CQ\backslash \sink$ (resp. $\CQ\backslash \sink[1]$) admits a cluster tilting set $\Ind(\Proj(Q)\backslash P_0)$ (resp. $\Ind(\Proj(Q)\backslash P_0)[1]$).
Hence, we have
\begin{gather}\label{eq:Qred}
\CQ\backslash \sink\cong\C(\Qminus)\cong\CQ\backslash \sink[1],
\end{gather}
where $\Qminus$ is the quiver obtained by deleting $0$ from $Q$.
By the realization \eqref{eq:red},
the formula above can be translated to the following homotopy equivalence in $\CCx(Q)$.
\begin{gather}\label{eq:ccred}
    \CCx(Q)\backslash\sink\simeq\CCx(\Qminus)\simeq\CCx(Q)\backslash\source.
\end{gather}
For application, we will consider $Q$ to be a Dynkin or Euclidean quiver from now on.
The discussion above covers the following cases of reduction:
\begin{itemize}
    \item $Q$ is a Dynkin quiver, $\sink$ is any rigid indecomposable.
    \item $Q$ is an Euclidean quiver, $\sink$ is any rigid indecomposable in $\vect(Q)$.
\end{itemize}
Recall that each Euclidean quiver is derived equivalent to $\plpqr$ as in \Cref{sec:qc}.
The following lemma extend the formula \eqref{eq:Qred} and \eqref{eq:ccred} to the case
when $Q$ is Euclidean and $\sink$ is any rigid indecomposable in $\tube(Q)$.

\begin{lemma}\label{lem:EG-red}
Let $Q$ be an Euclidean quiver.
For any rigid indecomposable $\sink$ in the height-$l$ layer of the fat tube $\tube_{\infty}(Q)$,
there are triangle equivalences
\begin{gather}\label{eq:EG-red}
\C(\plpqr)\backslash\sink\cong\C(\pl_{p,q,r-l})\sqcup\C(\DynA{l-1})\cong\C(\plpqr)\backslash \source.
\end{gather}
In particular, we have the following homotopy equivalence in $\CCx(\plpqr)$.
\begin{gather}\label{eq:plredcc}
    \CCx(Q)\backslash\sink\simeq\CCx(\pl_{p,q,r-l})*\CCx(\DynA{l-1})\simeq\CCx(Q)\backslash\source.
\end{gather}
\end{lemma}

\begin{proof}
By \cite[\S6.9]{CK}, each $\D(\plpqr)$ admits a tilting object whose endomorphism algebra is isomorphic to the squid algebra $\sqd{p}{q}{r}$ presented as follows (cf. \cite{BB}).
\[
\begin{tikzpicture}[scale=.8,rotate=0]
%The Kronecker part
\draw (-2,0) node {$\bullet$} (0,0) node {$\bullet$};
\draw[-stealth] (-1.8,0.1) to (-0.2,0.1);
\draw[-stealth] (-1.8,-0.1) to (-0.2,-0.1);
\draw (-1,0.5) node[font=\small]{$b_0$}
(-1,-0.5) node[font=\small]{$b_1$};
%The ternary tree part
\foreach \i in {-2,0,2}{
%draw the vertices
\draw (1.5,\i) node {$\bullet$} (6,\i) node {$\bullet$}
(3.75,\i) node {$\ldots$};
%draw the arrows
\draw[-stealth] (0.2,0) to (1.3,0);
\draw[-stealth] (0.12,0.16) to (1.38,1.84);
\draw[-stealth] (0.12,-0.16) to (1.38,-1.84);

\begin{scope}[shift={(1.5,\i)}]
    \draw[-stealth] (0.2,0) to (1.3,0);
\end{scope}
\begin{scope}[shift={(4.5,\i)}]
    \draw[-stealth] (0.2,0) to (1.3,0);
\end{scope}
%draw the braces
\draw[decorate, decoration={brace,raise=6pt}](1.5,\i) -- (6,\i);
}
\draw (3.75,2.7) node[font=\small]{$p-1$}
(3.75,0.7) node[font=\small]{$q-1$}
(3.75,-1.3) node[font=\small]{$r-1$}
;
%label the arrow
\draw (0.55,1.2) node[font=\small]{$c_0$}
(0.75,0.2) node[font=\small]{$c_1$}
(0.55,-1.2) node[font=\small]{$c_\infty$}
;
%relations
\draw (9,1.8) node{Relations:};
\draw (10,0) node[]{$\begin{cases}

    c_0 b_1 &=0\\
    c_1 b_1-c_1 b_0&=0\\
    \qquad-c_\infty b_0&=0,

\end{cases}$};
\end{tikzpicture}
\]
where $\sink$ is the projective $P_0$
corresponding to the $l$-th vertex away from the end of the length-$r$ branch, denoted by $0$.
Hence $\C(\sqd{p}{q}{r})\backslash \sink$ admits a cluster tilting object whose endomorphism algebra is isomorphc to $\sqd{p}{q}{r}\backslash\{\underline{0}\}=\sqd{p}{q}{r-l}\sqcup\DynA{l-1}$.
We have
\[
\C(\sqd{p}{q}{r})\backslash \sink\cong\C(\sqd{p}{q}{r-l}\sqcup\DynA{l-1})\cong\C(\sqd{p}{q}{r})\backslash \source,
\]
hence \eqref{eq:EG-red} holds.
Combining \eqref{eq:red} and \eqref{eq:union}, we have \eqref{eq:plredcc}.
\end{proof}

\begin{corollary}\label{cor:red}
Let $Q$ be a Dynkin or Euclidean quiver with $n$ vertices and $\sink$ be any rigid indecomposable in $\CQ$.
Then $\CQ\backslash\sink=\C(\hat{Q})$,
where $\hat{Q}$ is a quiver with $n-1$ vertices whose connected components are Dynkin or Euclidean.
\end{corollary}

%=========================================================
\subsection{Compact examples in Dynkin case}\label{sec:homotopy}
%=========================================================

Recall that an $\xevo$-foliation on $\CCx(\C)$ is compact if
each $\xevo$-leaf is compact.
In such a case, the $\xevo$-flow provides many nice homotopy properties.

We study the cluster complexes for Dynkin quivers with specific choices of $\xevo$
and show that the corresponding $\xevo$-foliations are compact.
As applications, we obtain inductive (homotopic) relations and thus their homotopy types.

%=========================================================
Let $Q$ be a Dynkin quiver with $n$ vertices.
Fix any indecomposable $\sink$ in $\CQ$.
Let $0$ be a vertex of $Q$ such that $\sink$ is in the $\tau$-orbit of the projective $P_0$.
Then we have \eqref{eq:Qred}.
The main result in this section is the following, where the proof is in \Cref{sec:P1}.

\begin{theorem}\label{thm:Dynkin}
For any vertex $\sink$ in $\CCx(Q)$,
the $\sink$-foliation is compact. As a result, we deduce that $\CCx(Q)$ is homotopy to the $(n-1)$-sphere $\SS^{n-1}$.
\end{theorem}

\begin{remark}
We do not assume the connectedness of $Q$.
Even if $Q$ is connected, $\Qminus$ is likely to be disconnected.
Nevertheless, \Cref{thm:Dynkin} still holds.
Indeed, we have \eqref{eq:union} and $\SS^{m-1}*\SS^{n-1}\simeq\SS^{m+n-1}$.
\end{remark}

\begin{example}
Let $Q=\DynE{8}$ and take $0$ to be the trivalent vertex, we have $\Qminus=A_1\sqcup A_2\sqcup A_4$
(see \Cref{fig:E8-A124}) and hence
\[
    \SS^7\simeq\CCx(\DynE{8})
        \simeq\susp\CCx( \DynA{1}
        \sqcup \DynA{2} \sqcup \DynA{4} )
        \simeq\susp(\SS^0*\SS^1*\SS^3)  \simeq\SS^7.
\]
\begin{figure}
\begin{tikzpicture}
    \node[label=below:0] (v0) at (0,0) {$\bullet$};
    \node (v1) at (0,1) {$\bullet$};
    \node (v2) at (-1,0) {$\bullet$};
    \node (v3) at (-2,0) {$\bullet$};
    \node (v4) at (1,0) {$\bullet$};
    \node (v5) at (2,0) {$\bullet$};
    \node (v6) at (3,0) {$\bullet$};
    \node (v7) at (4,0) {$\bullet$};
\draw[thick]
    (v0) edge (v1)  edge (v2) edge (v4)
    (v2) edge (v3) (v4) edge (v5) (v5) edge (v6) (v6) edge (v7)
;
\begin{scope}[shift={(0,-3)}]
\node[label=below:, gray!10] (v0) at (0,0) {$\circ$};
    \node (v1) at (0,1) {$\bullet$};
    \node (v2) at (-1,0) {$\bullet$};
    \node (v3) at (-2,0) {$\bullet$};
    \node (v4) at (1,0) {$\bullet$};
    \node (v5) at (2,0) {$\bullet$};
    \node (v6) at (3,0) {$\bullet$};
    \node (v7) at (4,0) {$\bullet$};
\draw[thick,dotted,gray!20]
    (v0) edge (v1)  edge (v2) edge (v4);
\draw[thick]
    (v2) edge (v3) (v4) edge (v5) (v5) edge (v6) (v6) edge (v7);
\end{scope}

\draw (-3,0.5) node {$Q$};
\draw (-3,-2.5) node {$\Qminus$};

\draw (1,-1.5) node {$\Downarrow$};

\end{tikzpicture}
\caption{From $E_8$ to $A_1\sqcup A_2\sqcup A_4$}\label{fig:E8-A124}
\end{figure}
\end{example}
%=========================================================
%\subsection{Compact foliations of type $A_3$}\label{sec:A3}
%=========================================================

%=========================================================
\subsection{Compact and semi-compact examples in Euclidean case}\label{sec:homotopy2}
%=========================================================
We continue the compact examples in the Euclidean case.
Let $Q$ be a connected Euclidean quiver with $n$ vertices.
The compact examples are studied the following theorem, where the proof is in \Cref{sec:P2}.

\begin{theorem}\label{thm:Ess}
Let $\sink$ be any rigid regular simple.
Then the $\sink$-foliation on $\CCx(Q)$ is compact.
As a result, we deduce that $\CCx(Q)$ is contractible.
\end{theorem}

\begin{example}[]\label{ex:A12-cfoli}
Consider the cluster category $\C(\EucA{1}{2})$ in \Cref{ex:A12}.
We have that:
\begin{itemize}
    \item \Cref{fig:A12-cc} illustrates the cluster complex $\CCx( \EucA{1}{2} )$.
    \item \Cref{fig:A12-V1} illustrates a compact $\xevo$-foliation for $\xevo=S^{(\infty)}_{1,1}$.
\end{itemize}

\begin{figure}
\centering\makebox[\textwidth][c]{
    \includegraphics[width=1\linewidth]{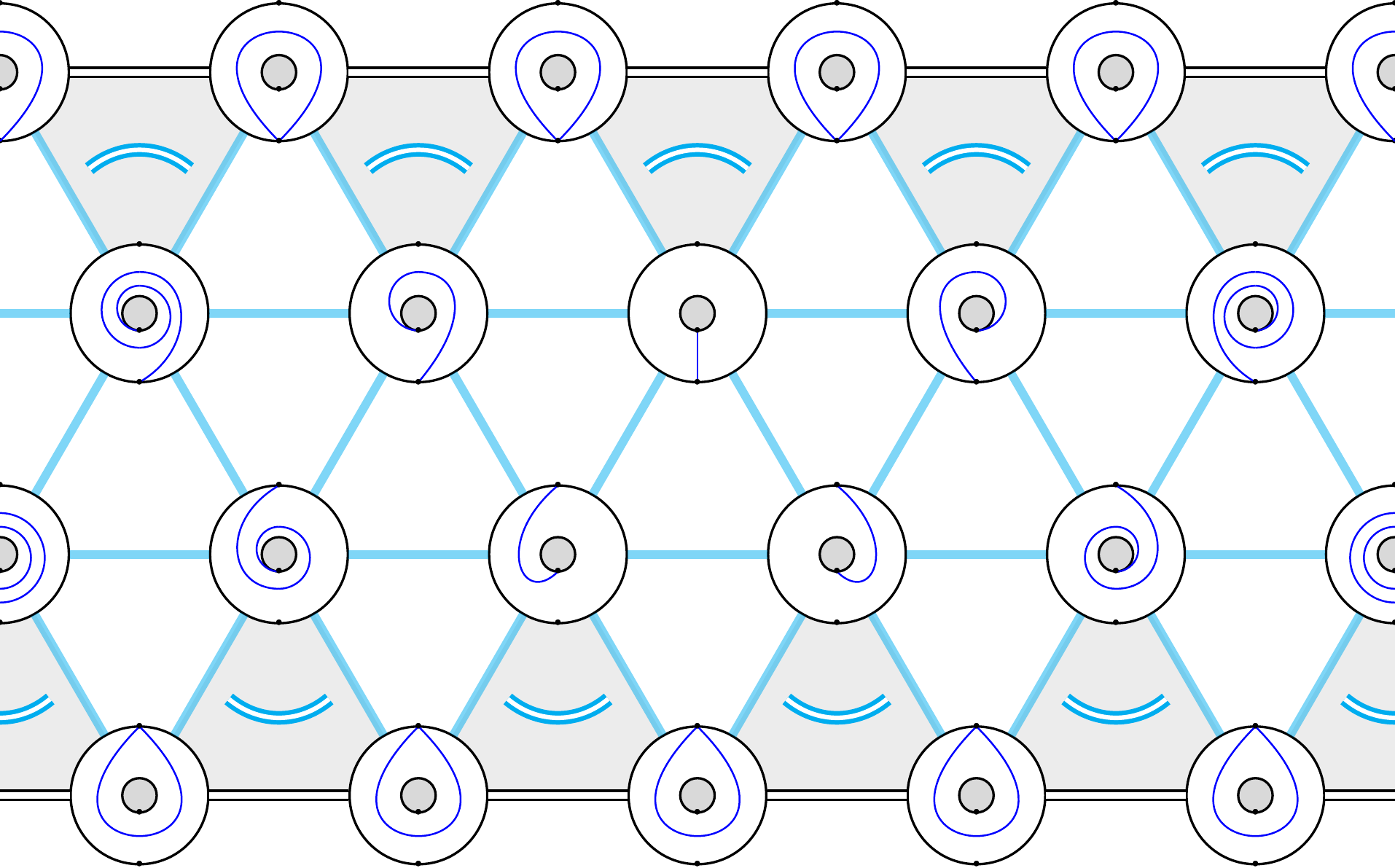}
}\caption{The cluster complex of $C( \EucA{1}{2} )$}
\label{fig:A12-cc}
\end{figure}

\begin{figure}\centering\makebox[\textwidth][c]{
    \includegraphics[width=1\linewidth]{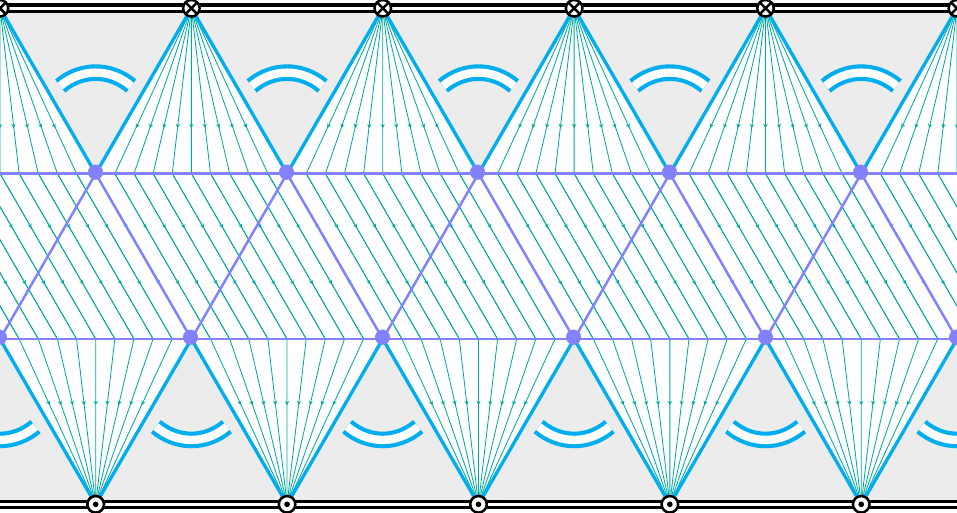}
}\caption{A compact $\sink$-foliation for $\CCx(\widetilde{A_{1,2}})$}
\label{fig:A12-V1}
\end{figure}
\end{example}

Next, we study the non-compact cases.

\begin{lemma}\label{lem:non-compact}
For any Euclidean quiver $Q$ and $\sink\in\Ind\vect(Q)$, the $\sink$-foliation is non-compact.
\end{lemma}
\begin{proof}
There exist a vertex $0$ of $Q$ such that $\sink$ is in the $\tau$-orbit of the projective $P_0$ in $\CQ$.
Then we have \eqref{eq:Qred}.
Suppose the $\sink$-foliation is compact,
then there is a homotopy equivalence \eqref{eq:susp} by \Cref{lem:compact}.
Combining with \eqref{eq:Qred}, we have
\[
\CCx(Q)\simeq\Sigma\CCx(\Qminus).
\]
Since $\Qminus$ is Dynkin, the homotopy type of $\CCx(\Qminus)$ is spherical as in \Cref{thm:Dynkin} and so does $\CCx(Q)$,
which contradicts to \Cref{thm:Ess}.
\end{proof}

We will give some semi-compact examples from the Euclidean quiver of type $\EucA{p}{q}$ in the following theorem, where the proof is in \Cref{sec:P3}.

\begin{theorem}\label{thm:E-charge}
For an $\EucA{p}{q}$ quiver $Q$ and any $\sink\in\Ind\vect(Q)$,
the $\sink$-foliation in $\CCx(Q)$ is semi-compact.
\end{theorem}

\begin{figure}[hbt]
    \centering

    \makebox[\textwidth][c]{
    \includegraphics[width=1.2\linewidth]{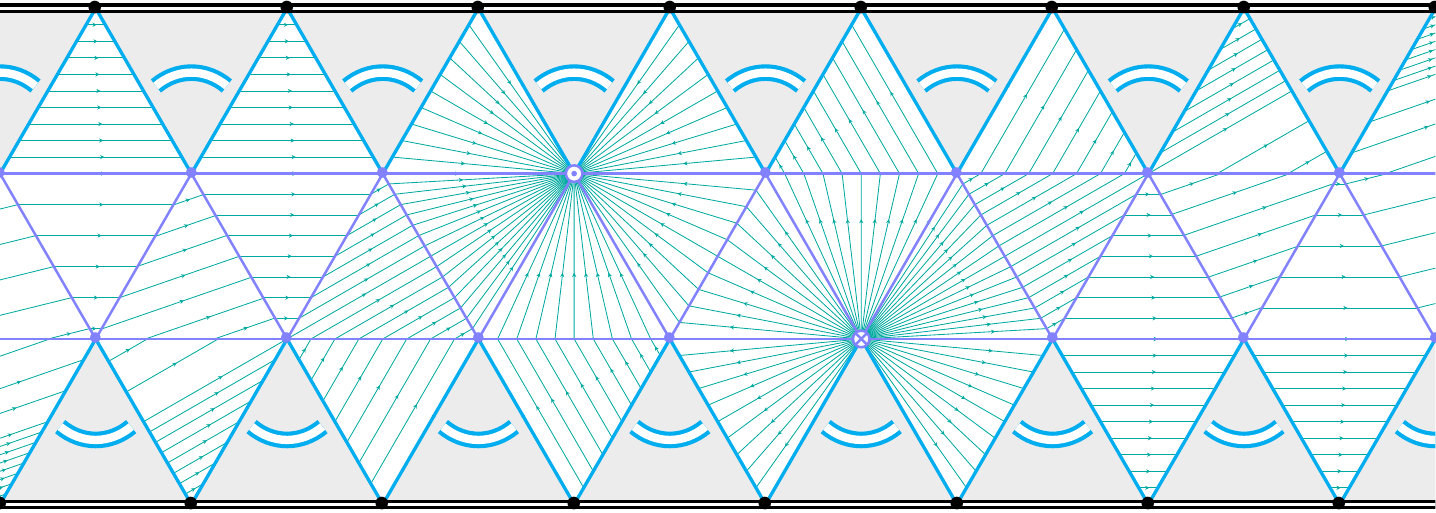}}

    \caption{A semi-compact $\sink$-foliation for $\CCx(\widetilde{A_{1,2}})$}
    \label{fig:A12-sfoli}
\end{figure}

\begin{example}\label{EX:A12}
Continue the $\EucA{1}{2}$ example in \Cref{ex:A12-cfoli}.
We have various semi-compact $\xevo$-foliations as follows.
\begin{itemize}
    \item \Cref{fig:A12-sfoli} illustrates the $P_1$-foliation.
    \item The first picture in \Cref{fig:A12-F-F*} illustrates the $\real{\pcto{P}}$-foliation,
    for the initial cluster
    \[
    \pcto{P}=\k Q=P_1\oplus P_2\oplus P_3.
    \]
    The second picture shows its induced green mutation.
    \item The third picture in \Cref{fig:A12-F-F*} illustrates the $\real{\pcto{U}}$-foliation,
    for the cluster
    \[
    \pcto{U}=\mu_{P_2}(\pcto{P})=P_1\oplus S^{(0)}_{1,1}\oplus P_3.
    \]
    The last picture shows its induced green mutation.
\end{itemize}
\end{example}

\begin{figure}
    \centering\makebox[\textwidth][c]{
    \includegraphics[width=1\linewidth]{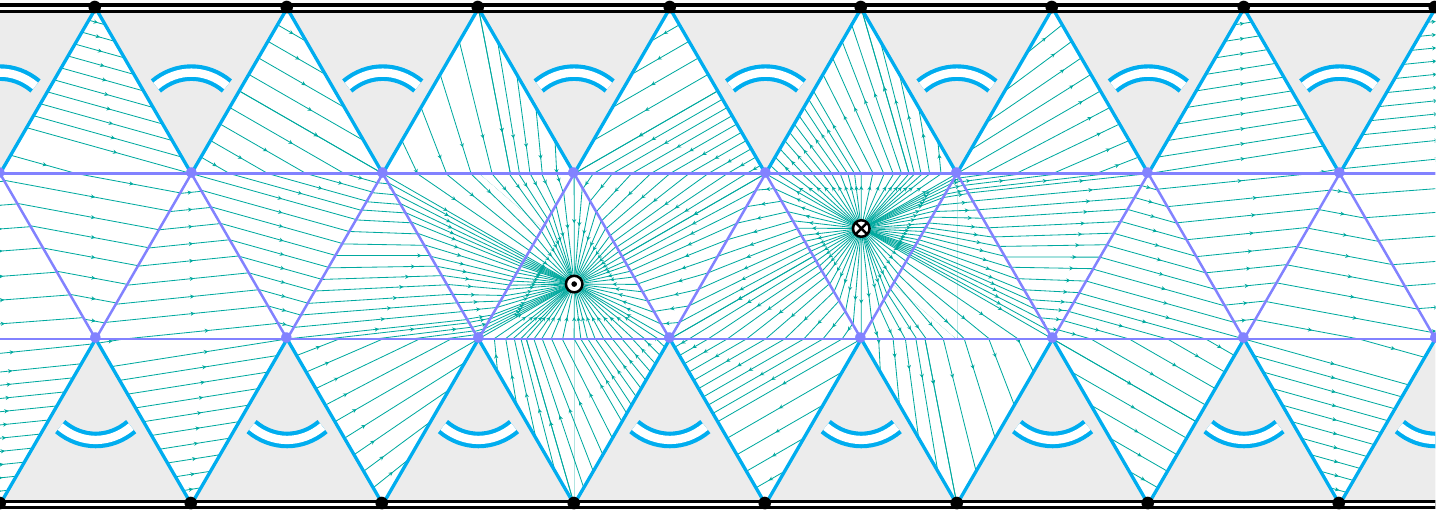}}
    \makebox[\textwidth][c]{
    \includegraphics[width=1\linewidth]{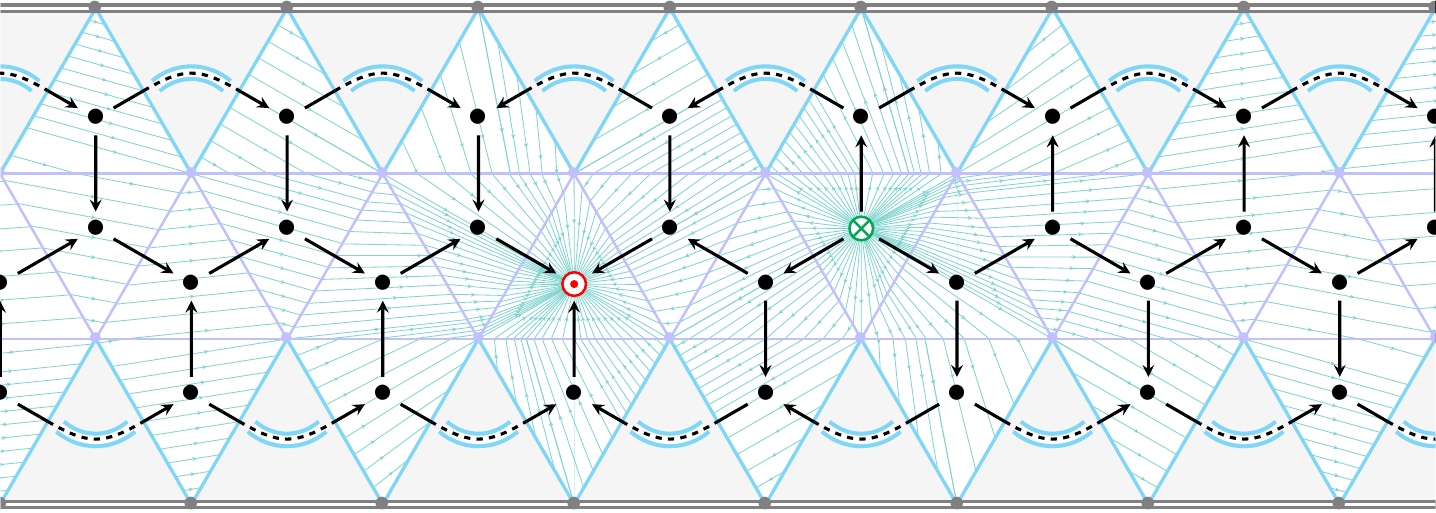}}
    \makebox[\textwidth][c]{
    \includegraphics[width=1\linewidth]{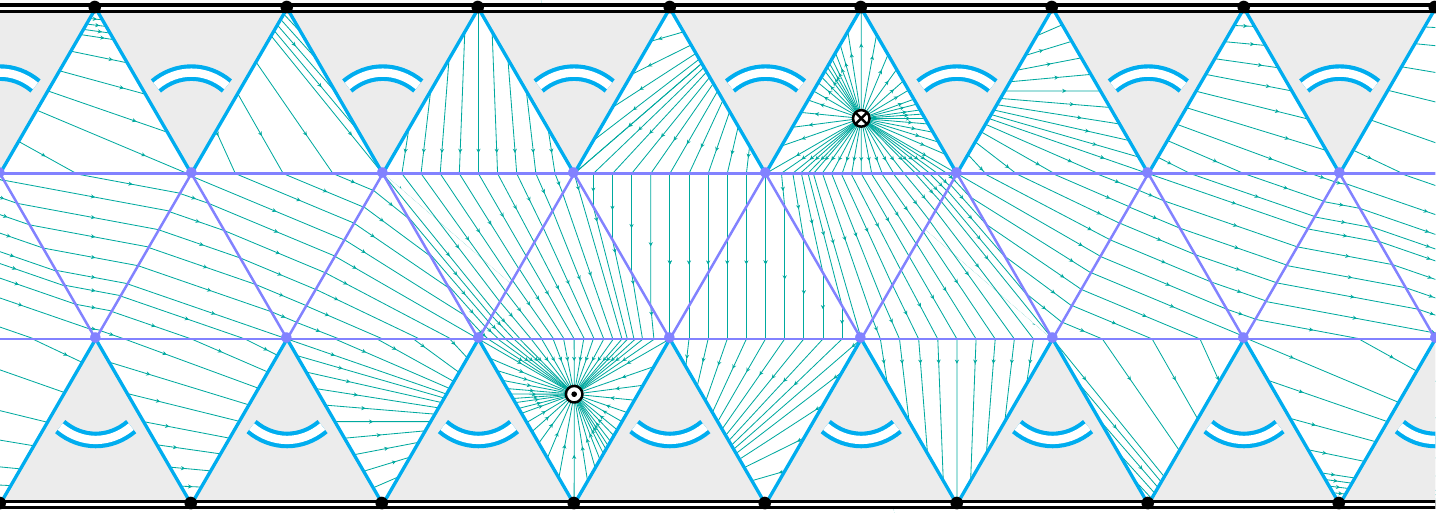}}
    \makebox[\textwidth][c]{
    \includegraphics[width=1\linewidth]{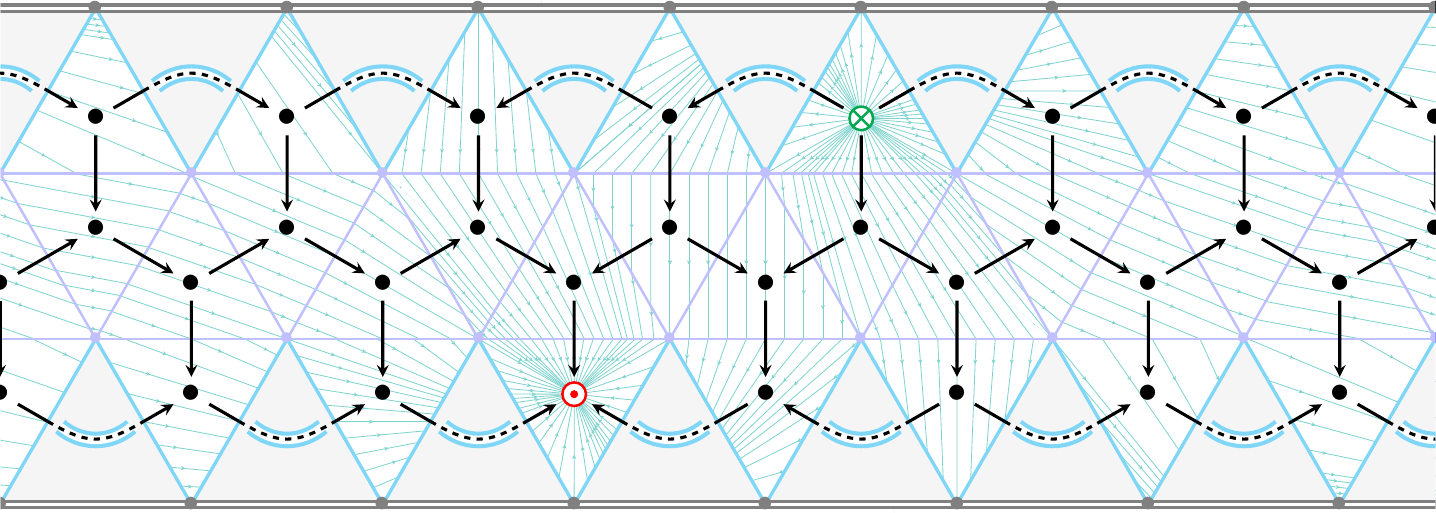}}
    \caption{$\xevo$-foliations: face center cases and induced orientations on $\uCEG(\EucA{1}{2})$}.
    \label{fig:A12-F-F*}
\end{figure}

%=========================================================
\section{Application to cluster exchange graphs}\label{sec:ceg}
%=========================================================
In this section, let $\C=\C(\class)$ be the cluster category associated to a mutation-equivalent classes of non-degenerate quivers with potential, cf. \cite[\S~2]{KQ2} for more details.
For each cluster $\pcts{\vv}$ in $\C(\class)$,
there is an associated quiver $Q_{\pcts{\vv}}$, which is the Gabriel quiver of $\pcts{\vv}$.
More precisely, the vertices of $Q_{\pcts{\vv}}$ are indecomposables in $\pcts{\vv}$
and the arrows correspond to (basis of) the irreducible morphisms between the vertices.

Let $\dim\CCx(\C)=n-1$.
Then any loop in $\CCx(\C)$ is homotopy to a loop in $\CCx^{\ge n-2}(\C)$.
%the interior of $\CCx^{\ge n-2}(\C)$.
Recall from \Cref{rem:dual}, we know that,
\begin{gather}\label{eq:Cpx-CEG}
    \CCx^{\ge n-2}(\C)\simeq\uCEG(\C).
\end{gather}

There are length four and five loops in $\uCEG(\C)$, known as the squares and pentagon.
One has the following well-known lemma, cf. \Cref{fig:45-gon} and \cite[\S2]{KQ2}.

\begin{lemma}
Let $\pcts{\vv}=\{\vv_i\}_{i=0}^n$ be any cluster in $\uCEG(\C)$
and take $\vv_i\ne \vv_j$ in $\pcts{\vv}$.
\begin{itemize}
  \item If $\dim\Irr_{\pcto{\vv}}(\vv_i,\vv_j)=0$, then there is a square in $\uCEG(\C)$.
  \item If $\dim\Irr_{\pcto{\vv}}(\vv_i,\vv_j)=1$, then there is a pentagon in $\uCEG(\C)$.
\end{itemize}
\end{lemma}

\begin{figure}
    \centering
\begin{tikzpicture}[scale=1.5]
%4-gon
\begin{scope}[shift={(0,0)}]
\node (y1) at (0,0) {$\pcts{\vv}$};
\node (y2) at (1,1) {$\pcts{\vv}_i$};
\node (y3) at (1,-1) {$\pcts{\vv}_j$};
\node (y4) at (2,0) {$\pcts{\vv}_{ij}$};
\draw[thick] (y1) to (y2) (y1) to (y3) (y2) to (y4) (y3) to (y4);
\draw
(.3,.7) node[font=\small]{$\vv_i$}
(1.7,.7) node[font=\small]{$\vv_j$}
(.3,-.7) node[font=\small]{$\vv_j$}
(1.7,-.7) node[font=\small]{$\vv_i$}
;
\end{scope}
%5-gon
\begin{scope}[shift={(3,0)}]
\node (y1) at (0,0) {$\pcts{\vv}$};
\node (y2) at (.75,1) {$\pcts{\vv}_i$};
\node (y3) at (2.25,1) {$\pcts{\vv}_*$};
\node (y4) at (1.5,-1) {$\pcts{\vv}_j$};
\node (y5) at (3,0) {$\pcts{\vv}_{ij}$};
\draw[thick] (y1) to (y2) (y2) to (y3) (y3) to (y5) (y1) to (y4) (y4) to (y5);
\draw
(.2,.7) node[font=\small]{$\vv_i$}
(2.8,.7) node[font=\small]{$\vv_j$}
(1.5,1.2) node[font=\small]{$Y_j$}
(.6,-.7) node[font=\small]{$\vv_j$}
(2.4,-.7) node[font=\small]{$\vv_i$}
;
\end{scope}
\end{tikzpicture}
    \caption{The squares and pentagons}
    \label{fig:45-gon}
\end{figure}

A key assumption about $\uCEG(\C)$ in \cite{KQ2} is the following.
\begin{itemize}
  \item the fundamental group $\pi_1\CEG$ is generated by squares and pentagons.
\end{itemize}
%Such an assumption appeared in \cite{}.
It indeed holds for various cases:
\begin{itemize}
  \item If $\C=\CQ$ for a Dynkin quiver $Q$, it was studied in \cite{Q1}, cf. \cite{QW}.
  \item If $\C=\C(\surf)$ for a marked surface $\surf$, then it holds by \cite{FST}.
\end{itemize}

In particular, we can give a uniform prove for all Dynkin and Euclidean quivers.

\begin{lemma}\label{lem:starsum}
Suppose that $\CCx(\C)$ is simply connected.
For any $[\alpha]\in\pi_1\CCx^{\ge n-2}(\C)$, there are finite $0$-cells $\{\sink_i\}_{i\in I}$ and $[\alpha_i]\in\pi_1\Star^{\ge n-2}(\sink_i)$ such that
\[
[\alpha]=\sum_{i\in I}[\alpha_i].
\]
\end{lemma}

\begin{proof}
Let $\alpha:[0,1]\to\CCx^{\ge n-2}(\C)$ be a representative of $[\alpha]$.
Since $\CCx(\C)$ is simply connected, $\alpha$ can be realized as the boundary of a contractible disk $D_\alpha$.

Note that there is an open cover for $\CCx(\C)$ as follows.
\[
\CCx(\C)=\bigcup_{\sink\in\C} \Star^\circ(\sink),
\]
where $\Star^\circ(\sink)=\Star(\sink)\setminus\link(\sink)$.
By the compactness of $D_\alpha$, there is a finite subcover as follows.
\[
D_\alpha\subset\bigcup_{i\in I} \Star^\circ(\sink_i).
\]
Hence one can decompose $\D_\alpha$ into finitely many small disks
\[
D_\alpha=\bigcup_{i\in I}D_{\alpha_i}
\]
such that each $D_{\alpha_i}$ is in $\Star^\circ(\sink_i)$ with boundary $\alpha_i$.
Hence, we have $[\alpha_i]\in\pi_1\Star^\circ(\sink_i)$ and $[\alpha]=\sum_{i\in I }[\alpha_i]$.
By a small perturbation up to homotopy, we can make $[\alpha_i]\in\pi_1\Star^{\ge n-2}(\sink_i)$ as required.
\end{proof}

\begin{theorem}
For any Dynkin or Euclidean quiver $Q$,
$\pi_1\uCEG(Q)$ is generated by squares and pentagons.
\end{theorem}

\begin{proof}
We use induction on the number of vertices $n$ of $Q$.
The $n=1$ case is trivial and for $n=2$, there are three cases: $Q=\DynA{1}\times\DynA{1},\DynA{2}$ or $\EucA{1}{1}$.
The theorem holds as
\begin{itemize}
    \item $\uCEG(\DynA{1}\times\DynA{1})$ is a square;
    \item $\uCEG(\DynA{2})$ is a pentagon;
    \item $\uCEG(\EucA{1}{1})$ is contractible.
\end{itemize}
Now assume the theorem holds for Dynkin/Euclidean quivers with vertices less than $n$ and consider the case for $n>2$.
By \eqref{eq:Cpx-CEG}, it is equivalent to show $\pi_1\CCx^{n\ge2}(Q)$ is generated by squares and pentagons.

Since $\CCx(Q)$ is either spherical or contractible by \Cref{thm:Dynkin} and \Cref{thm:Ess}.
By \Cref{lem:starsum}, for each $[\alpha]\in\CCx^{\ge n-2}(Q)$, there are $0$-cells $\{\sink_i\}_{i\in I}$ and $[\alpha_i]\in\pi_1\Star^{\ge n-2}(\sink_i)$ such that
\[
[\alpha]=\sum_{i\in I}[\alpha_i].
\]
Note that for each $0$-cell $\sink$, the projection $\proj{\sink}$ (cf. \eqref{eq:proj}) induces a homotopy equivalence as follows.
\[
\Star^{\ge n-2}(\sink)\simeq\CCx^{\ge n-3}(\C(Q)\backslash\sink).
\]
By \Cref{cor:red}, $\CQ\backslash\sink_i\cong\C(\hat{Q})$ for some subquiver $\hat{Q}$ with $n-1$ vertices whose connected components are Dynkin or Euclidean.
As a result, we have
\[
\pi_1\Star^{\ge n-2}(\sink)\cong\pi_1\CCx^{\ge n-3}(\hat{Q}).
\]
%and the image of each $[\alpha_i]$ under $\proj{\sink_i}$ are in $\pi_1\CCx^{\ge n-3}(\CQ\backslash\sink_i)$.
Using the inductive assumption, we deduce that $[\alpha_i]\in\pi_1\Star^{\ge n-2}(\sink_i)$ is generated by squares and pentagons for any $i\in I$ and so does $[\alpha]$.
\end{proof}

\appendix

%=========================================================
\section{Elements on quivers and Auslander-Reiten theory}
%=========================================================

%=========================================================
\subsection{Quiver categories}\label{sec:qc}\
%=========================================================

Let $Q$ be an acyclic quiver, i.e. an oriented graph without cycles.

Denote by $\k Q$ the path algebra, $\h(Q)=\mathrm{mod}~\k Q$ the module category and
and $\Proj(\k Q)$ the full subcategory of $\h(Q)$ consisting of projective modules.

Let $\DQ=\D^b(\k Q)$ be the bounded derived category of $\k Q$ with the canonical heart $\h(Q)$.
As $\k Q$ is hereditary, $\Ind\DQ=\Ind\h(Q)[\ZZ]$,
where $\Ind\C$ denotes a complete set of indecomposables in an additive category $\C$.
Denote by $\tau\in\Aut(\DQ)$ its Auslander-Reiten (AR) functor and $\Serre=\tau\circ[1]$ its Serre functor.
Set $\Serre_d=\Serre\circ[-d]$ for $d\in\ZZ$.

The {(2-)cluster category} $\CQ$ of $Q$ is the orbit category $\DQ/\Serre_2$ (cf. \cite{BMRRT,Ke1}),
where
\[
    \Hom_{\CQ}(M,N)=\bigoplus_{i\in\ZZ}\Hom_{\DQ}(M,\Serre_2^i N).
\]
Moreover, $\CQ$ is 2-Calabi-Yau with an initial/canonical cluster tilting set $\Ind\Proj(\k Q)$ and
\[
    \Ind\CQ=\Ind\h_Q\cup\Ind\Proj(\k Q)[1].
\]

%=========================================================
\paragraph{\textbf{Dynkin and Euclidean quivers}}\

A Dynkin quiver is a quiver whose underlying diagram is a ternary tree, as shown in \Cref{fig:ter-tre},
with branches of length $(p,q,r)$ satisfy \eqref{eq:invsum>1}.

\begin{figure}
\begin{minipage}[b]{.3\linewidth}
    \centering
\begin{tikzpicture}[scale=.5,rotate=0]
\draw (0,0) node {$\bullet$};
\foreach \i in {0,120,240}{
%draw the vetices/ldots
\begin{scope}[shift={(\i:1.5)}]
\draw (0,0) node {$\bullet$};
\end{scope}
\begin{scope}[shift={(\i:3)}]
\draw (0,0) node[rotate=\i]{$\ldots$};
\end{scope}
\begin{scope}[shift={(\i:4.5)}]
\draw (0,0) node {$\bullet$};
\end{scope}
%draw the edges
\draw[thick] (\i:0.3)to(\i:1.2);
\draw[thick] (\i:1.8)to(\i:2.3);
\draw[thick] (\i:3.7)to(\i:4.2);
%draw the braces
\draw[decorate, decoration={brace,raise=6pt}](0,0) -- (\i:4.5);
}
\draw (2.25,1.4) node[font=\small]{$p$}
    (-2.3,1.35)node[font=\small]{$q$}
    (0.25,-2.5)node[font=\small]{$r$};
%the descriptions
\end{tikzpicture}
\end{minipage} %\par\medskip
\begin{minipage}[b]{.5\linewidth}
\centering
\begin{gather}\label{eq:invsum>1}
    \text{Dynkin type:\quad}\invsum>1.
\end{gather}\vskip .5cm
\begin{gather}\label{eq:invsum=1}
     \text{Euclidean type:\quad}\invsum=1.
\end{gather}
\[\]\[\]
\end{minipage}
\caption{The ternary tree for Dynkin/Euclidean quivers}
    \label{fig:ter-tre}
\end{figure}

An Euclidean quiver is a quiver whose underlying diagram is one of the following (extended Dynkin) diagrams:
\begin{description}
  \item[$\EucA{p}{q}$] as shown in the left pictures of \Cref{fig:affineAD} with $p$ clockwise arrows.
  \item[$\EucD{n}$] as shown in the right pictures of \Cref{fig:affineAD}.
  \item[$\EucE{n}$] a ternary tree, as shown in \Cref{fig:ter-tre}, with branches of length $(p,q,r)$ satisfy \eqref{eq:invsum=1}.
\end{description}

\begin{figure}[htb]\centering
\begin{tikzpicture}[scale=.8,rotate=0]
%The left part
\foreach \i in {-60,60}{
%draw the vertices
\draw (0,0) node {$\bullet$};
\draw (\i:1.5)node {$\bullet$};
%draw the edges
\draw[thick] (\i:0.3) to (\i:1.2);
\begin{scope}[shift={(\i:1.5)}]
\draw[thick] (0:0.3) to (0:1.5);
\end{scope}
}

%The right part, xscale=-1
\begin{scope}[shift={(6,0)},xscale=-1]
\foreach \i in {-60,60}{

\draw (0,0) node {$\bullet$};
\draw (\i:1.5)node {$\bullet$};

\draw[thick] (\i:0.3) to (\i:1.2);
\begin{scope}[shift={(\i:1.5)}]
\draw[thick] (0:0.3) to (0:1.5);
\end{scope}
}
\end{scope}

%draw ldots
\draw (3,1.3) node {$\ldots$} (3,-1.3) node {$\ldots$};
%draw the braces
\begin{scope}[shift={(0,1.5)}]
\draw[decorate, decoration={brace,raise=4pt}](0:0.7) -- (0:5.3);
\end{scope}
\begin{scope}[shift={(0,-1.5)}]
\draw[decorate, decoration={brace,raise=4pt}](5.3,0) -- (180:-0.7);
\end{scope}

\draw (3,2.1) node[font=\small]{$p-1$};
\draw (3,-2.1) node[font=\small]{$q-1$};
%=================================================
%Type Affine D
%=================================================
\begin{scope}[shift={(8,0)}]
%the left part
\foreach \i in {120,240}{
%draw the vertices
\draw (0,0) node {$\bullet$};
\begin{scope}[shift={(\i:1.5)}]
\draw (0,0) node {$\bullet$};
\end{scope}
%draw the edges
\draw[thick] (\i:0.3)to(\i:1.2);
}
\draw[thick] (0:0.3)to(0:1.2);

%the right part
\begin{scope}[shift={(4,0)},xscale=-1]
\draw (0,0) node {$\bullet$};
\foreach \i in {120,240}{

\begin{scope}[shift={(\i:1.5)}]
\draw (0,0) node {$\bullet$};
\end{scope}

\draw[thick] (\i:0.3)to(\i:1.2);
}
\draw[thick] (0:0.3)to(0:1.2);
\end{scope}

%draw ldots
\begin{scope}[shift={(0:2)}]
\draw (0,0) node{$\ldots$};
\end{scope}
%draw the brace
\draw[decorate, decoration={brace,raise=6pt}](0:0) -- (0:4);
\draw (2,0.7) node[font=\small]{$n-3$};
\end{scope}

\end{tikzpicture}
\caption{The underlying graph of type $\EucA{p}{q}$(left) and type $\EucD{n}$(right)}
    \label{fig:affineAD}
\end{figure}

%=========================================================
\paragraph{\textbf{Weighted projective lines of tame type}}\

In the Euclidean case, besides the `algebraic' canonical heart $\h(Q)$,
there is a `geometric' canonical heart corresponding to coherent sheaves on a weighted projective line.

The weighted projective line $\pl_\wt$ is a projective line with orbifold points $(p_1,\ldots,p_t)\subset\pl$ and weights $\wt=(w_1,\ldots,w_t)$, that is, equipping an $\ZZ_{w_i}$-action for the orbifold point $p_i$, $1\le i\le t$.

The category $\coh(\pl_\wt)$ of coherent sheaves on $\pl_\wt$ is an abelian category and naturally splits into two parts:
\[
    \Ind\coh(\pl_\wt)=\Ind\vect(\pl_\wt) \coprod \Ind\tube(\pl_\wt),
\]
where $\vect(\pl_\wt)$ is the category
of vector bundles and $\tube(\pl_\wt)=\coprod_{\lambda\in\pl_{\wt}}\tube_{\lambda}(\pl_\wt)$ is the category consisting of $\pl_{\wt}$-family of stable tubes.
The following decomposition for an object $\vv$ will be used later.
\begin{gather}\label{eq:Euc-dcp}
\vv=\vv|_{\vect}\oplus\bigoplus_{\lambda\in\pl_{\wt}}\vv|_{\tube}.
\end{gather}
We only consider the tame case,
Let $t=3$, $(p_1,p_2,p_3)=(0,1,\infty)$ and $(w_1,w_2,w_3)=(p,q,r)$ satisfying \eqref{eq:invsum>1}.
We will write $\plpqr$ for $\pl_\wt$ and $\D(\plpqr)$ for $\D^b(\coh(\plpqr))$.

We will fix a triangle equivalence (cf. \cite[\S3]{GL} and \cite[\S6]{LR})
\begin{equation}\label{eq:derived-equiv}
    \DQ\cong\D(\plpqr)
\end{equation}
such that $\tube(\pl_\wt)$ becomes the regular modules.
This makes $\coh(\plpqr)$ the other canonical heart of $\DQ$.
More precisely, we have
\[
    \D(\widetilde{A_{p,q}}) \cong \D( \pl_{p,q,1} ), \quad
    \D(\widetilde{D_{n}}) \cong \D( \pl_{2,2,n-1} ) \quad\text{and}\quad
    \D(\widetilde{E_{n}}) \cong\D( \pl_{2,3,n-3} ).
\]
For convenience, we will not distinguish $\DQ\cong\D(\plpqr)$ and similarly for $?(\plpqr)$ and $?(Q)$,
where $?=\C, \vect, \tube, \CCx$.

%==========================================================
\subsection{Shape of AR quivers}\label{sec:shape}\
%=========================================================

Let $\ZZ Q$ be the translation quiver of $Q$.
We will write $\uat=\tau^{-1}$ for simplicity.

The Auslander-Reiten (AR-)quiver $\ar(\C)$ of a $k$-linear additive category $\C$ is a directed graph, whose vertices are (parameterized by) indecomposable objects $\Ind(\C)$ of $\C$
and, the number of arrows from $X$ to $Y$ equals the dimension of
    the space $\Irr_{\C}(X,Y)$ of irreducible maps.

The following facts are well-known. (\cite[Chap.~I]{Hap})
\begin{itemize}
    \item If $Q$ is a Dynkin quiver, then $\ar(\DQ)\cong\ZZ Q$;
    \item If $Q$ is an Euclidean quiver, then $\ar(\DQ)$ consists of $\ZZ$-copies of the components $\ar(\vect(Q))\cong\ZZ Q$
    and the components of thin tubes (i.e. rank one), except for the three below (possibly fat).
    \begin{gather}\label{eq:fatt}
    \begin{cases}
    \ar(\tube_{0}(Q))\cong\ZZ_p\DynA{\infty}\\
    \ar(\tube_{1}(Q))\cong\ZZ_q\DynA{\infty}\\
    \ar(\tube_{\infty}(Q))\cong\ZZ_r\DynA{\infty}.
    \end{cases}
    \end{gather}
\end{itemize}
In particular, we concern the rigid objects in $\DQ$.
Denote by $(-)^\rgd$ the full subcategory of $-$ consisting of rigid objects,
then $\vect^\rgd(Q)=\vect(Q)$ and
$\tube_{\lambda}^\rgd(Q)$ consists of the objects of length$<k$,
where $k$ is the rank of $\tube_{\lambda}^\rgd(Q)$
(cf. \cite[\S1]{CK}).
In particular,
\[
    \ar(\tube^\rgd_{\lambda}(Q))\cong\ZZ_k\DynA{k-1}.
\]
For application,
let $\widetilde{\ar}(\tube^\rgd_{\lambda}(Q))\cong\ZZ\DynA{m-1}$ be the $\ZZ$-covering space of $\ar(\tube^\rgd_{\lambda}(Q))$ induced by the natural quotient $\ZZ\DynA{k-1}\to\ZZ_k\DynA{k-1}$.

%=========================================================
\subsection{Intervals in AR quivers}\label{sec:int}\
%=========================================================

In $\ZZ Q$,
we call $M_0$ the predecessor of $M_t$ and $M_t$ the successor of $M_0$
if there is a path $ M_0\xrightarrow{f_1}\ldots\xrightarrow{f_t}M_t $.
The intervals associated with vertices $M,N\in\ZZ Q$ are:
\begin{itemize}
    \item $[M,\infty)$ is the set of all the successors of $M$,
    \item $(-\infty,N]$ is the set of all the predecessors of $N$,
    \item $[M,N]=[M,\infty)\cap(-\infty,N]$.
\end{itemize}
We also define the interval between two set of vertices $\pcts{M}=\{M_i\}_i$ and $\pcts{N}=\{N_j\}_j$ as:
\begin{gather}\label{def:D-interval}
[\pcts{M},\pcts{N}]=\bigcup_{i,j}[M_i,N_j].
\end{gather}
This definition can be extend to disjoint union of $\ZZ Q_s$
by setting $[M_i,N_j]=\emptyset$ when $M_i$ and $N_j$ are not in the same component of $\sqcup_{s}\ZZ Q_s$.
The following notation of restriction will be used later.
\begin{gather}\label{def:res}
    [\pcts{M},\pcts{N}]|_{\ZZ Q_s}=[\pcts{M}|_{\ZZ Q_s},\pcts{N}|_{\ZZ Q_s}]=\bigcup_{M_i,N_j\in\ZZ Q_s}[M_i,N_j].
\end{gather}
We will establish a technical criterion about interval-decreasing for a series of intervals from set of vertices.

\begin{lemma}\label{lem:dd}
Fix $\pcts{Y}$ and let $\{\pcts{Z}_k\}_{k\in\NN}$ be a series of sets of vertices satisfying
\begin{gather}\label{eq:dd}
    [\pcts{Y},\pcts{Z}_{k+1}\backslash\pcts{Z}_{k}]\subset[\pcts{Y},\pcts{Z}_{k}\backslash\pcts{Z}_{k+1}]
\quad\text{with equality only if}\quad
[\pcts{Y},\pcts{Z}_k]=\emptyset.
\end{gather}
for any $k\in\NN$.
If $[\pcts{Y},\pcts{Z}_0]\ne\emptyset$, then there exists $k_0>0$ such that $[\pcts{Y},\pcts{Z}_{k_0}]\subsetneqq[\pcts{Y},\pcts{Z}_{0}]$.
In particular, $[\pcts{Y},\pcts{Z}_k]=\emptyset$ when $k\gg0$.
\end{lemma}

\begin{proof}
We use induction on $|[\pcts{Y},\pcts{Z}_0]|=N$,
which is trivial for the case $N=1$.
When $N>1$, we assume the lemma holds for any $\{\pcts{Z}'_k\}_{k\in\NN}$ satisfying \eqref{eq:dd} with $|[\pcts{Y},\pcts{Z}'_0]|<N$.
Suppose that $|[\pcts{Y},\pcts{Z}_0]|=N$ and $[\pcts{Y},\pcts{Z}_{k+1}]=[\pcts{Y},\pcts{Z}_{k}]$ for any $k\in\NN$.
Then for any maximal interval $[\pcts{Y},Z_0]$ in $\{[\pcts{Y},Z]~|~Z\in\pcts{Z}_0\}$,
we have $Z_0\in\bigcap_{k\in\NN}\pcts{Z}_k$.

Now let $\pcts{Z}'_k=\pcts{Z}_k\backslash\{Z_0\}$ then $\{\pcts{Z}'_k\}_{k\in\NN}$
satisfies $|[\pcts{Y},\pcts{Z}'_0]|<N$ and \eqref{eq:dd}.
By the assumption, $[\pcts{Y},\pcts{Z}'_k]=\emptyset$ when $k\gg0$ thus the original interval have to decrease for some $k_0\ge0$.
\end{proof}

The section in $\ZZ Q$ is a full subquiver whose underlying graph is a connected graph that meets each $\tau$-orbit exactly once.
For any vertex $M,N$ in $\ZZ Q$, denote by
\begin{itemize}
    \item $\Ps_{+}(M)$ the section containing $M$ in $[M,\infty)$.
    \item $\Ps_{-}(N)$ the section containing $N$ in $(-\infty,N]$.
\end{itemize}

In the end, we will state some Hom-properties of Dynkin/Euclidean quiver without proof.
Suppose $Q$ is a Dynkin quiver.
Given any $\sink, Y\in\Ind\DQ$, we have
\begin{gather}
\begin{cases}
\Hom_{\DQ}(\sink,Y)\ne0\implies Y\in[\sink,\tau \sink[1]],\\
\Hom_{\DQ}(Y,\sink)\ne0\implies Y\in[\uat \sink[-1],\sink].
\end{cases}
\end{gather}
Suppose $Q$ is an Euclidean quiver and
fix any $\sink\in\Ind\vect(Q)=\Ind\vect^\rgd(Q)$.
Then for any $Y\in\Ind\DQ$ we have similarly
\begin{gather}\label{eq:Homvect}
\begin{cases}
\Hom_{\DQ}(\sink,Y)\ne0\implies Y\in[\sink,\infty)\cup\tube(Q)\cup(-\infty,\tau \sink[1]],\\
\Hom_{\DQ}(Y,\sink)\ne0\implies
Y\in[\uat \sink[-1],\infty)\cup\tube(Q)[-1]\cup(-\infty,\sink].
\end{cases}
\end{gather}
When we take $\sink$ to be a rigid regular simple in a fat tube $\tube_{\lambda}(Q)$.
It is also known that for any $Y\in\Ind\tube_{\lambda}(Q)$,
\begin{gather}\label{eq:Homtube}
\begin{cases}
\Hom_{\DQ}(\sink,Y)\ne0\implies Y\in\Ps_{+}(\sink)\\
\Hom_{\DQ}(Y,\sink)\ne0\implies Y\in\Ps_{-}(\sink).
\end{cases}
\end{gather}

%=========================================================
\subsection{Geometric model for type $\EucA{p}{q}$}\label{sec:GeoApq}\
%========================================================

In case study, we will use a geometric model for the cluster category $\C(\EucA{p}{q})\cong\C(\surf)$.
Here $\surf$ is an annulus with $p$ and $q$ marked points on its boundary components, respectively.
Recall that the rotation $\rho(\gamma)$ of an arc $\gamma$ is the arc obtained from $\gamma$ by moving each endpoint of which along the boundary anti-clockwise to the next marked point.
\begin{lemma}\cite[\S3]{QZ1}\label{lem:gm}
There is a bijection
\[\begin{array}{rccl}
  &\{\text{simple closed curves in $\surf$}\}
  &\longleftrightarrow
  &\{\text{rigid objects in $\C(\surf)$}\}%\\
\end{array}\]
sending $\gamma$ to $M_{\gamma}$, such that
\begin{itemize}
    \item the rotation $\rho$ becomes the AR-functor $\tau$, i.e.
    $M_{\rho(\gamma)}=\tau M_{\gamma}$.
    \item $\Int(\gamma,\delta)=\dim_\k\Ext_{\C(\surf)}^1(M_\gamma,M_\delta)$.
\end{itemize}
\end{lemma}

In particular, indecomposable objects in $\tube_0(Q)$ (resp. $\tube_1(Q)$) correspond to the arcs with both endpoints in the boundary component with $p$ (resp. $q$) marked points.

\begin{example}\label{ex:A12}
Consider an $\EucA{1}{2}$ quiver $Q$, whose geometric model for $\CQ$ is shown in the left side of \Cref{fig:A12-tube},
where the triangulation corresponds to the initial cluster $\k Q$.
The AR-quiver of $\CQ$ has $2$ connected components containing rigid objects:
\begin{itemize}
    \item The vector bundle part $\vect(Q)$ shown in \Cref{fig:A12-vect} and
    \item the fat tube $\tube_0(Q)$ shown in the right side of \Cref{fig:A12-tube}.
\end{itemize}

\begin{figure}[htbp]
    \centering\makebox[\textwidth][c]{
    \includegraphics[width=.9\linewidth]
    {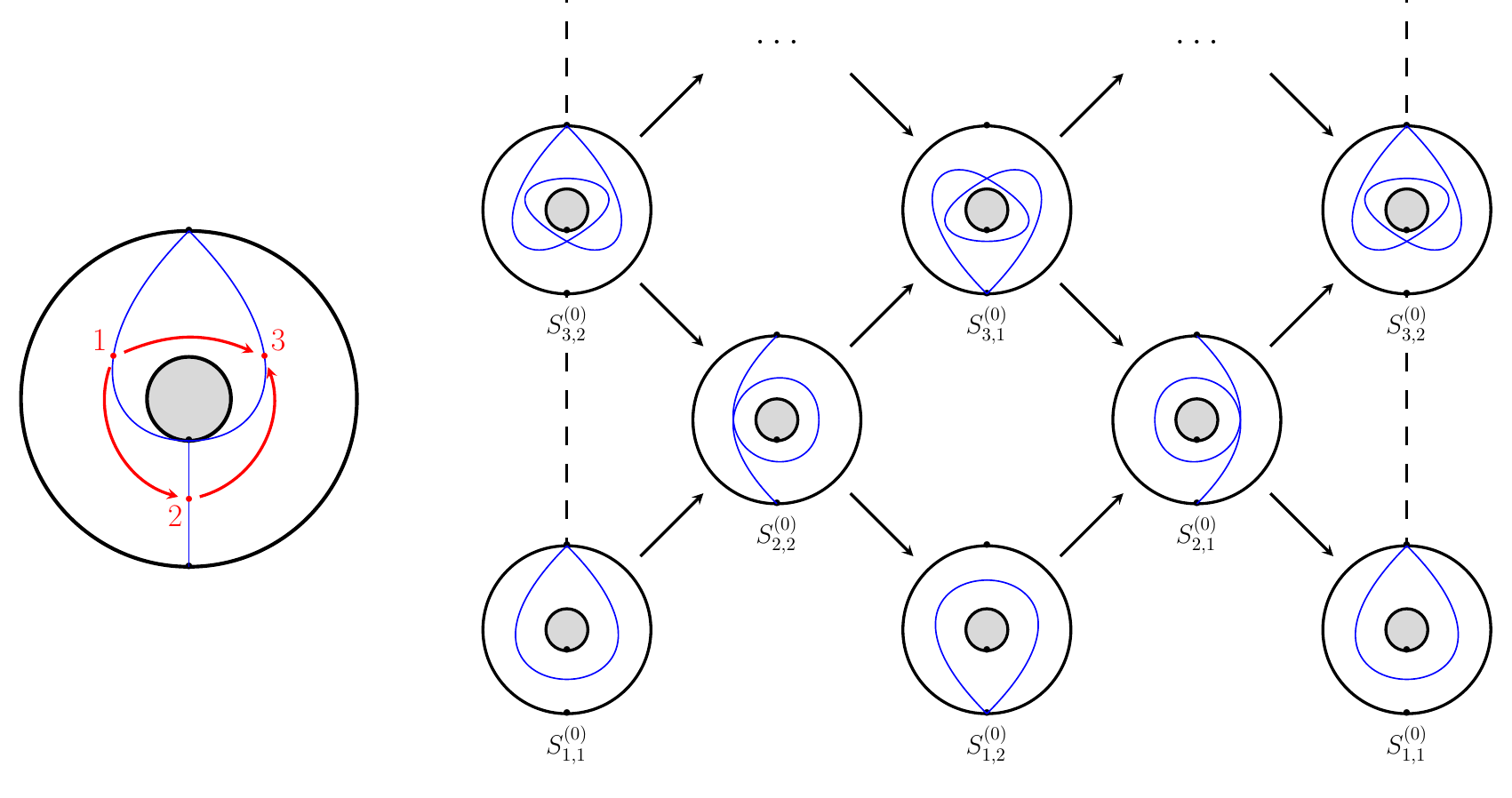}\qquad}
    \caption{The triangulation for $\k\EucA{1}{2}$ on the left;
    the rank $2$ fat tube $\tube_0({\EucA{1}{2}})$ on the right}
    \label{fig:A12-tube}
\end{figure}

\begin{figure}[h]
    \centering\makebox[\textwidth][c]{\quad
    \includegraphics[width=1.1\linewidth]{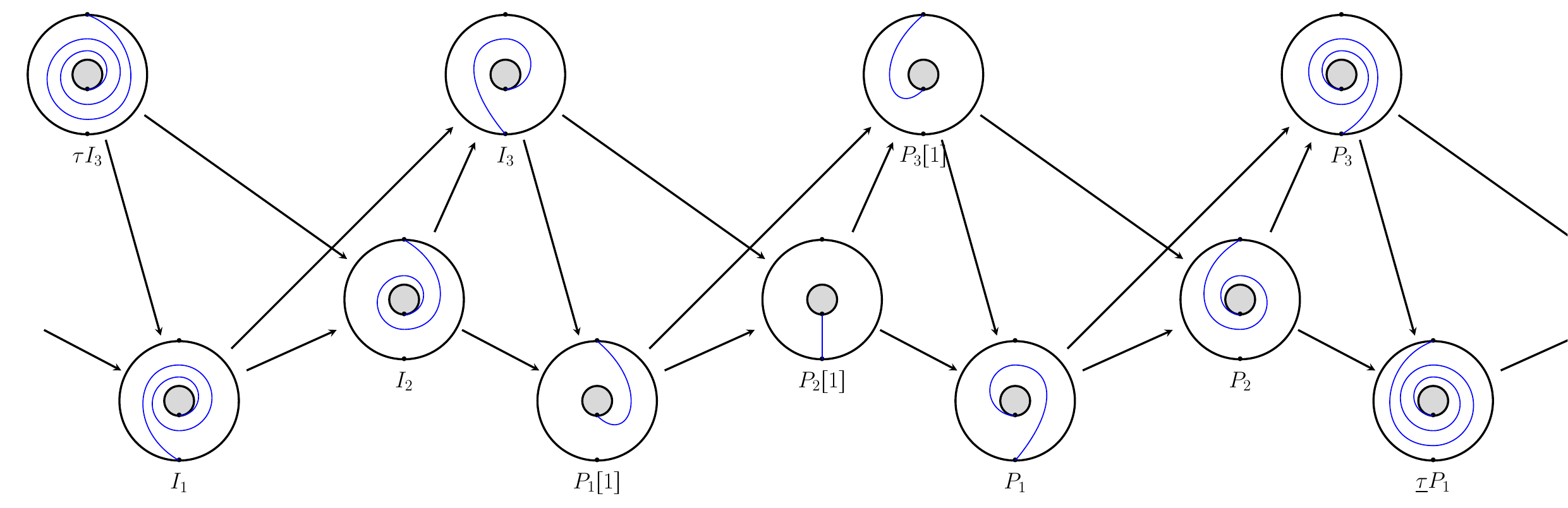}}
    \caption{The AR-quiver $\ar(\vect(\EucA{1}{2})) \cong\ZZ\EucA{1}{2}$}
    \label{fig:A12-vect}
\end{figure}
\end{example}

%=========================================================
\section{On the proofs in Dynkin/Euclidean cases}
%=========================================================
%=========================================================
\paragraph{\textbf{Convention}} We will write $\dapp{\vv}=\cc(\App{R}{\pctc{\vv}}{\source})$ for short.
%=========================================================
\subsection{Proof of \Cref{thm:Dynkin}}\label{sec:P1}
%=========================================================
We may assume that $\sink=\sink_0$ is a simple-projective of $\h(Q)$ corresponding to vertex $0$ (using AR-translation) and $\Proj \k Q=\Ps_{+}(\sink)$ (using APR-tilting).

Take a fundamental domain
\[
    \on{FdD}_{\sink}(\CQ)=\h(Q) \cup \Proj \k Q[1] \subset \DQ
\]
with the canonical identification/bijection
\begin{gather}\label{id1}
\Ind\CQ \longleftrightarrow [ \Ps_{+}(\sink), \Ps_{+}(\sink)[1]].
\end{gather}
Let $L$ be an $\sink$-leaf and $\pcts{\vv}$ be any cell intersecting $L$ nontrivially.
We define the intervals associated to $\pcts{\vv}$:
\begin{gather}\label{eq:int}
    \II(\pcts{\vv})=[\sink,\dapp{\vv}].
\end{gather}
It is straightforward to see that
\begin{gather}\label{criterion1}
\pcts{\vv}\subset\Star(\sink)\iff\Ext^1(\pcto{\vv},\sink)=0
\iff\II(\pcts{\vv})=\emptyset.
\end{gather}
In such a case, $L$ intersects $\pcts{\vv}\subset\Star(\sink)$ nontrivially and ends at $\sink$ along $\dflow{\sink}$.

\begin{lemma}\label{lem:ccross1}
When there is a cell-crossing from $\pcts{\uu}$ to $\pcts{\ww}$ along $\dflow{\sink}$ (cf. \Cref{def:ccross}), we have
\begin{gather}\label{eq:throw}
    [\sink,\dapp{\ww}\backslash\dapp{\uu}]
    \subsetneqq
    [\sink,\dapp{\uu}\backslash\dapp{\ww}].
\end{gather}
\end{lemma}

\begin{proof}
Since $\pcts{\uu}\not\subset\Star(\sink)$, $\II(\pcts{\uu})\ne\emptyset$ by \eqref{criterion1}.
We may assume that $\dapp{\ww}\backslash\dapp{\uu}\ne\emptyset$.

For any $\ww\in\dapp{\ww}\backslash\dapp{\uu}$, we have $\Irr_{\pctc{\ww}}(\ww,\source)\ne0$.
There exists $\uu\in\dapp{\uu}$ such that any $w\in\Irr_{\pctc{\ww}}(\ww,\source)$ factors through some $u\in\Irr_{\pctc{\uu}}(\uu,\source)$ since $\ww\not\in\dapp{\uu}$.
Moreover, the irreducibility of $w$ in $\pctc{\ww}$ implies that $\uu\not\in\dapp{\ww}$.
Hence
\begin{gather}\label{eq:ind1}
[\sink,\ww]\subsetneqq[\sink,\uu],
\end{gather}
which implies that
\[
[\sink,\dapp{\ww}\backslash\dapp{\uu}]\subset
[\sink,\dapp{\uu}\backslash\dapp{\ww}].
\]
The inclusion have to be proper since the maximal interval in
\[
\{[\sink,\uu]~|~\uu\in
\dapp{\uu}\backslash\dapp{\ww}\}
\]
can not be in any $[\sink,\ww]$ for $\ww\in\dapp{\ww}\backslash\dapp{\uu}$ due to \eqref{eq:ind1}.
\end{proof}

Suppose $\pcts{\vv}_0\not\subset\Star(\sink)$ intersects $L$ nontrivially.
Let $\pcts{\vv}_1,\pcts{\vv}_2,...$ be the consecutive cells intersecting $L$ nontrivially from $\pcts{\vv}_0$ along $\dflow{\sink}$.
By \Cref{lem:ccross1}, we can apply \Cref{lem:dd} for $\{\pcts{Y}=\sink,\pcts{Z}_k=\dapp{\vv}_k\}$
and deduce that $\II(\pcts{\vv}_k)=\emptyset$ when $k\gg0$.

As a consequence, $L$ enters $\Star(\sink)$ by \eqref{criterion1} and then eventually hits $\sink$ along $\dflow{\sink}$.
Dually, $L$ ends up at $\source$ along $\uflow{\sink}$
thus any $\sink$-leaf $L$ on $\CCx(Q)$ is compact, i.e. the $\sink$-foliation is compact.
By \Cref{lem:compact}, the flow induces the homotopy equivalence \eqref{eq:susp}.
Combining with \eqref{eq:Qred},
we have
\[
\CCx(Q)\simeq\susp\CCx( \Qminus ).
\]
Note that $\CCx( \DynA{1} )\cong\SS^0$. We inductively deduce that the cluster complex is spherical.

%=========================================================
\subsection{Proof of \Cref{thm:Ess}}\label{sec:P2}
%=========================================================
Recall that $Q$ is derived equivalent to a weighted projective line $\plpqr$ as in \Cref{sec:qc}.
Fix the rigid regular simple $\sink\in\Ind\tube^\rgd_{\infty}(Q)$, i.e. the indecomposable in the bottom of the fat tube.
By \Cref{lem:EG-red}, we have
\begin{gather}\label{eq:EG-red1}
    \C(\plpqr)\backslash\sink\cong\C(\pl_{p,q,r-1})\cong\C(\plpqr)\backslash\source.
\end{gather}

Take a fundamental domain
\[
    \on{FdD}({\CQ})=
    \bigcup_{\lambda\in\plpqr}\tube_\lambda(Q)\cup\vect(Q)[1]\subset\DQ
\]
for $\CQ$ with an identification:
\begin{gather}\label{id2}
    \Ind\CQ\longleftrightarrow
    \Ind(\bigcup_{\lambda\in\plpqr}\tube_\lambda(Q)\cup
    \vect(Q)[1]).
\end{gather}
Recall the $\ZZ$-cover of (the A-R quiver of) the fat tube (cf. \Cref{sec:shape}) and
choose the following identification:
\begin{gather}\label{idtube}
    \Ind\tube^\rgd_{\infty}(Q)\longleftrightarrow
    [\Psp(\sink),\Psp(\uat^{r-1}\sink)]\subset\widetilde{\ar}(\tube^\rgd_{\infty}(Q))\cong\ZZ\DynA{r-1}.
\end{gather}
Then the following holds for any $\vv$ in \eqref{idtube} by \eqref{eq:Homvect}.
\begin{gather}\label{eq:Hom1}
\begin{cases}
    \Hom_{\CQ}(\sink,\vv)=\Hom_{\DQ}(\tau \sink[-1],\vv)\oplus\Hom_{\DQ}(\sink, \vv),\\
    \Hom_{\CQ}(\vv,\source)=\Hom_{\DQ}(\vv,\source)\oplus\Hom_{\DQ}(\vv,\tau \sink).
\end{cases}
\end{gather}
In particular, combining with \eqref{eq:Homtube} we obtain (cf. \Cref{fig:5tube1})
\begin{gather}
\begin{cases}\label{eq:Zone}
    \Hom_{\CQ}(\sink,\vv)\ne0\iff \vv\in\tau(\Psm(\tau \sink)\cup\Psp(\uat\sink))\\
    \Hom_{\CQ}(\vv,\source)\ne0\iff \vv\in\Psm(\tau \sink)\cup\Psp(\uat \sink).
\end{cases}
\end{gather}

We need one more lemma to classify the $\sink$-evolving triangles.
For $T_\pm$ in $\on{Ps}_\pm^\rgd(\tau^{\mp1} \sink)$,
there are unique maps (up to scaling)
\begin{gather}\label{eq:tpm}
t_\pm\in\Hom_{\CQ}(T_\pm,\source),
\end{gather}
where $t_-$ comes from $\Hom_{\DQ}(T_-,\tau\sink)$ and $t_+$ comes from $\Hom_{\DQ}(T_+,\source)$.

\begin{figure}[htb]\centering
\begin{tikzpicture}[rotate=0,scale=.9]
\clip(-6,-.7)rectangle(8,6.5);
%draw dashed lines
\draw[dashed, ultra thick] (-5,0)--(-5,7);
\draw[dashed, ultra thick] (5,0)--(5,7);
%even height part
\foreach \i in {-5,-3,-1,1,3,5}
{
%draw the ind objs
\draw (\i,0) node[fill=white] {$\bullet$} ;
\draw (\i,2) node[fill=white] {$\bullet$} ;
\draw (\i,4) node[fill=white] {$\circ$} ;
}
%odd height part
\foreach \i in {-4,-2,0,2,4}
{
%draw the ind objs
\draw (\i,1) node {$\bullet$} ;
\draw (\i,3) node {$\bullet$} ;
\draw (\i,5) node[fill=white] {$\circ$} ;
%draw the irr mors
\foreach \j in {0,2,4}{
\begin{scope}[shift={(0,\j)}]
\draw[-stealth, thin] (\i-0.9,0.1) to (\i-0.1,0.9);
\begin{scope}[shift={(1,1)},yscale=-1]
    \draw[-stealth, thin] (\i-0.9,0.1) to (\i-0.1,0.9);
\end{scope}
\begin{scope}[shift={(0,2)}, yscale=-1]
    \draw[-stealth, thin] (\i-0.9,0.1) to (\i-0.1,0.9);
    \begin{scope}[shift={(1,1)},yscale=-1]
        \draw[-stealth, thin] (\i-0.9,0.1) to (\i-0.1,0.9);
    \end{scope}
\end{scope}
\end{scope}
}
%draw the Ps+/-
\draw[line width=4mm, color=SkyBlue, opacity=.1, line cap=round]
(-3,0) -- (3,6);
\draw[line width=4mm, color=Orange, opacity=.1, line cap=round]
(-3,6) -- (3,0);
}
%draw non-rigid parts
\filldraw[color=white, opacity=.7] (-5.1,3.1) rectangle (5.1,7);
\filldraw[color=white](-5.2,5.5) rectangle (5.2,7.1);
\draw[thick, color=gray, dotted]
(-6,3.1)--(8,3.1);
%discriptions
\draw
%(8,3.5) node {non-rigid}
(-2.6,5.8) node[color=Orange] {$\Ps_{+}(\uat \sink)$}
(2.6,5.8) node[color=SkyBlue] {$\Ps_{-}(\tau \sink)$}
(-5,-0.5) node {$\sink$}
(-3,-0.5) node {$\uat \sink$}
(3,-0.5) node {$\tau \sink=\uat^{r-1}\sink$}
(5,-0.5) node {$\sink$}
(6.5,2.7) node {$\rgd$=rigid part}
(6.5,3.5) node {non-rigid part};
\end{tikzpicture}
\caption{The tube of length 5}
    \label{fig:5tube1}
\end{figure}

\begin{lemma}\label{lem:factor}
Let $\vv\in\Ind\vect(Q)[1]$ and $f\in\Hom_{\CQ}(\vv,\source)$ be nonzero,
then $f$ factors through $t_-$ and $t_+$ factors through $f$.
\end{lemma}

\begin{proof}
By \eqref{eq:Hom1}, \eqref{eq:tpm} come from the following triangles in $\DQ$.
\begin{gather}\label{eq:tri+}
\sink\to T'_+\to T_+\xrightarrow{\on{t}_+}\source,\\
\label{eq:tri-}
\sink\to\uat T'_-[1]\to \uat T_-[1]\xrightarrow{\on{t}_-} \source,
\end{gather}
where $T'_\pm\in\Ind\tube^\rgd_{\infty}(Q)$ and $\on{t}_\pm$ are preimages of $t_\pm$ in $\DQ$.
Also, $f$ comes from the following triangle in $\DQ$.
\begin{gather}\label{eq:trif}
\vv'\to \vv\xrightarrow{\on{f}}\source\to \vv'[1],
\end{gather}
where $\vv'$ is an object in $\vect(Q)[1]$ since $\source$ is a rigid regular simple. $\on{f}$ is a preimage of $f$ in $\DQ$.

Applying $\Hom_{\DQ}(\vv,-)$ to \eqref{eq:tri-}, we obtain
\[
\Hom_{\DQ}(\vv,\uat T_-[1])\xrightarrow{\Hom(\vv,\on{t}_-)}\Hom_{\DQ}(\vv,\source)\to \Hom_{\DQ}(\vv,\uat T'_-[2])\xlongequal{\eqref{eq:Homvect}}0,
\]
thus $\on{f}$ factors through $\on{t_-}$ and so do their images in $\CQ$.
Applying $\Hom_{\DQ}(T_+,-)$ to \eqref{eq:trif},
we obtain
\[
\Hom_{\DQ}(T_+,\vv)\xrightarrow{\Hom(T_+,\on{f})}\Hom_{\DQ}(T_+,\source)\to\Hom_{\DQ}(T_+,\vv'[1])\xlongequal{\eqref{eq:Homvect}}0,
\]
Hence $\on{t}_+$ factors through $\on{f}$ and so do their images in $\CQ$.
\end{proof}

\begin{corollary}\label{cor:sim-evo}
In any $\sink$-evolving triangle $\eqref{eq:evo-tri}$,
either $\uu\in\vect(Q)$ or
\[
\uu\in\Ind(\Psm(\tau \sink)\cup\Psp(\uat \sink)).
\]
\end{corollary}

\begin{proof}
If $\uu\not\in\vect(Q)$, then there exists $\uu_0\in\Ind\tube^\rgd_{\infty}(Q)$ in $\Add(\uu)$.
By \eqref{eq:Zone}, $\uu_0\in\Psm(\tau \sink)\cup\Psp(\uat \sink)$.
By \Cref{lem:factor}, $\uu$ have no summand in $\vect(Q)[1]$.
In particular, since the indecomposable in the higher layer must factor through the lower one in $\Psm(\tau \sink)$ and dually in $\Psp(\uat \sink)$, we have $\uu=\uu_0\in\Psm(\tau \sink)\cup\Psp(\uat \sink)$.
\end{proof}

Fix an $\sink$-leaf $L$ and let $\pcts{\vv}$ be any cell intersecting $L$ nontrivially.
Similar to \eqref{eq:int}, we consider the following data associated with $\pcts{\vv}$.
\begin{gather}
\begin{cases}
    \num_{-}(\pcts{\vv})=|\pcts{\vv}\cap\Psm(\tau\sink)|,\\
    \II|_{\vect}(\pcts{\vv})=
    [\ueot(L),\dapp{\vv}]|_{\vect},\\
    \num_{+}(\pcts{\vv})=|\pcts{\vv}\cap\Psp(\uat\sink)|,
\end{cases}
\end{gather}
Recall that $\ueot(L)=\ueot(\pcto{\vv})$ is an invariant on any $\sink$-leaf $L$ as a corollary of \Cref{thm:const}.
The interval is well-defined by the properties of the irreducible extension, cf. \Cref{prop:ie}.
By \Cref{cor:sim-evo}, the following are equivalent.
\begin{gather}\label{criterion2}
\pcts{\vv}\subset\Star(\sink)\iff\Ext^1(\pcto{\vv},\sink)=0\iff
\num_{\pm}(\pcts{\vv})=0~\&~\II|_{\vect}(\pcts{\vv})=\emptyset.
\end{gather}

\begin{lemma}\label{lem:ccross2}
When there is a cell-crossing of $L$ from $\pcts{\uu}$ to $\pcts{\ww}$ along $\dflow{\sink}$ (cf. \Cref{def:ccross}), we have the following.
\begin{enumerate}
    \item $\num_{-}(\pcts{\ww})=\max\{\num_{-}(\pcts{\uu})-1,0\}$.
    \item
    If $\num_{-}(\pcts{\uu})=0$, then
    \begin{gather}\label{eq:throw2}
        [\ueot(L),\dapp{\ww}\backslash\dapp{\uu}]|_{\vect}\subset
        [\ueot(L),\dapp{\uu}\backslash\dapp{\ww}]|_{\vect}
    \end{gather}
    with equality only if $\II|_{\vect}(\pcts{\uu})=\emptyset$.
    \item If $\num_{-}(\pcts{\uu})=0$ and $\II|_{\vect}(\pcts{\uu})=\emptyset$, then
    $\num_{+}(\pcts{\ww})=\num_{+}(\pcts{\uu})-1$.
\end{enumerate}
\end{lemma}
\begin{proof}
Let $\pcts{\vv}$ be the cell-wall of the cell-crossing and
\[
\sink\to\ww_0\to\dapp{\vv}\to\source
\]
the downward $\sink$-evolving triangle for $\pcts{\vv}$.
Then $\pcts{\ww}=\pcts{\vv}\cup\{\ww_0\}$.

\textbf{For $1^\circ$}:
Since $\Hom(\sink,\ww_0)\ne0$, $\ww_0\not\in\Psm(\tau \sink)$ by \eqref{eq:Zone}.
By \Cref{cor:sim-evo} the equation holds.

\textbf{For $2^\circ$}:
Let $\num_{-}(\pcts{\uu})=0$,
then $\num_{-}(\pcts{\ww})=0$ by $1^\circ$.
$\dapp{\uu}$ (resp. $\dapp{\ww})$ is either in $\Psp(\uat \sink)$ or in $\vect(Q)$ by \Cref{cor:sim-evo}.

If $\II|_{\vect}(\pcts{\uu})\ne\emptyset$, then $\dapp{\uu}$ is in $\vect(Q)$.
The proper inclusion is trivial for $\dapp{\ww}\subset\Psp(\uat \sink)$.
Suppose not then $\dapp{\uu},\dapp{\ww}\subset\vect(Q)$.
Similar to \Cref{lem:ccross1}, we also have the proper inclusion.

If $\II|_{\vect}(\pcts{\uu})=\emptyset$, then $\dapp{\uu}\subset\Psp(\uat \sink)$ and so does $\dapp{\ww}$ by \Cref{lem:factor}.
Two sides in \eqref{eq:throw2} are equal since they are both empty.

\textbf{For $3^\circ$}:
Let $\num_{-}(\pcts{\uu})=0$, $\II|_{\vect}(\pcts{\uu})=\emptyset$,
then $\num_{-}(\pcts{\ww})=0$, $\II|_{\vect}(\pcts{\ww})=\emptyset$ by $1^\circ$ and $2^\circ$.
Note that $\pcts{\uu}\not\subset\Star(\sink)$, $\num_{+}(\pcts{\uu})\ne0$ by \eqref{criterion2}.
By \Cref{cor:sim-evo}, $\dapp{\ww}\subset\Psp(\uat \sink)$ thus $\ww_0\not\in\Psp(\uat \sink)$.
It is similar to $1^\circ$ that $3^\circ$ holds.
\end{proof}

Take $\pcts{\vv}_0$ to be an arbitrary cell intersecting $L$ nontrivially.
Consider the continuous cells $\pcts{\vv}_1,\pcts{\vv}_2$ intersecting $L$ nontrivially from $\pcts{\vv}_0$ along $\dflow{\sink}$.
By \Cref{lem:ccross2},
    $\num_{-}(\pcts{\vv}_k)=0$ for $k\gg0$ firstly.
    Then one can use \Cref{lem:dd}
    to show $\II|_{\vect}(\pcts{\vv}_k)=\emptyset$ when $k\gg0$.
    At last, $\num_{+}(\pcts{\vv}_k)=0$ for $k\gg0$.

Combining with \eqref{criterion2}, $L$ intersects $\pcts{\vv}_k\subset\Star(\sink)$ nontrivially for $k\gg0$,
which will contract into $\sink$ along $\dflow{\sink}$.
Dually, $L$ will end up at $\source$ along $\uflow{\sink}$.
As a consequence, the $\sink$-foliation is compact.
By \Cref{lem:compact}, it induces the homotopy equivalences \eqref{eq:susp}.
Combining with \eqref{eq:EG-red1}, we have
\[
\CCx(\pl_{p,q,r})\simeq\susp\CCx(\pl_{p,q,r-1}).
\]
Note that $\CCx(\EucA{1}{1})=\CCx(\pl)$ is contractible. Inductively, we deduce that the cluster complex is contractible.

%=========================================================
\subsection{Proof of \Cref{thm:E-charge}}\label{sec:P3}
%=========================================================
Consider the following heart.
\[
\h_\sink=\on{mod}(\k \Ps_{+}(\sink)^{\on{op}}).
\]
Take the fundamental domain
\[
\on{FdD}_{\sink}(\CQ)=\h_\sink\cup\Proj(\k \Ps_{+}(\sink)^{\on{op}})[1]\subset\DQ,
\]
it induces an identification of $\CQ$ as follows
\begin{gather}\label{id3}
    \Ind\CQ\longleftrightarrow[\Ps_+(\sink),\infty)\cup\tube(Q)\cup(-\infty,\Ps_+(\sink)[1]].
\end{gather}

Up to symmetric, take $X$ to be the green arc in the left of \Cref{fig:Apq+5tube}
and $T^{(0)}$ and $T^{(1)}$ the blue and red ones.
Choose the following bijection for each tube:
\begin{gather}\label{idtube2}
\Ind(\tube_\lambda^\rgd(Q))\longleftrightarrow[\Psp(T^{(\lambda)}),\Psp(\uat^{l-1}T^{(\lambda)})]\subset\widetilde{\ar}(\tube^\rgd_\lambda(Q))\cong\ZZ\DynA{l-1}.
\end{gather}
Then for any $\vv\in\Ind\tube_\lambda(Q)$,
we have the following by \Cref{lem:gm} (cf. right of \Cref{fig:Apq+5tube}):
\begin{gather}
\begin{cases}
\Hom_{\CQ}(\sink,\vv)\ne0\iff \vv\in\tau[\uat T^{(\lambda)},\tau T^{(\lambda)}]\\
\Hom_{\CQ}(\vv,\source)\ne0\iff \vv\in[\uat T^{(\lambda)},\tau T^{(\lambda)}].
\end{cases}
\end{gather}

\begin{figure}
    \centering
    \includegraphics[width=\linewidth]{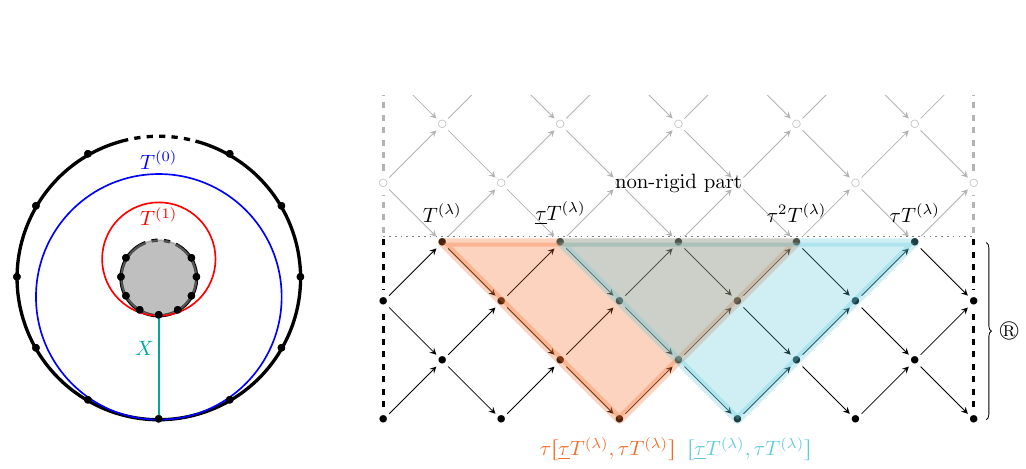}
    \caption{The geometric model for $\EucA{p}{q}$ (left) and the rank 5 tube (right)}
    \label{fig:Apq+5tube}
\end{figure}

Consider any partial cluster $\pcts{\vv}$.
By \eqref{eq:Homvect}, $\Ext^1(\vv,\sink)\ne0~\text{for any}~\vv\in\dapp{\vv}$ implies that
\[
\dapp{\vv}\subset[\uat\sink,\infty)\cup\tube(Q)\cup(-\infty,\source].
\]
Note that $\Ext^1([\uat\sink,\infty),(-\infty,\source])\ne0$ in type $\EucA{p}{q}$ case, we have moreover
\begin{gather}\label{eq:ext}
    \text{Either}~\dapp{\vv}\subset[\uat\sink,\infty)\cup\tube(Q)
    ~\text{or}~\dapp{\vv}\subset\tube(Q)\cup(-\infty,\source].
\end{gather}
Fix an $\sink$-leaf $L$.
For any cell $\pcts{\vv}$ intersecting $L$ nontrivially
with $\dapp{\vv}\in[\uat\sink,\infty)\cup\tube(Q)$,
we consider the following two intervals associated to it.
\begin{gather} \label{eq:int3}
\begin{cases}
     \II|_{\tube}(\pcts{\vv})=
    \bigsqcup_{\lambda=0,1}[T^{(\lambda)},\dapp{\vv}]|_{\tube_\lambda},\\
     \II|_{\vect}(\pcts{\vv})=
    [\sink,\dapp{\vv}]|_{\vect}.
\end{cases}
\end{gather}
Similar to the previous cases, we have the following observation.
\begin{gather}\label{criterion3}
    \pcts{\vv}\subset\Star(\sink)\iff
    \Ext^1(\pcto{\vv},\sink)=0\iff
    \II|_{\tube}(\pcts{\vv})=\emptyset=\II|_{\vect}(\pcts{\vv}).
\end{gather}

\begin{lemma}\label{lem:ccross3}
When cell-crossing from $\pcts{\uu}$ to $\pcts{\ww}$ along $\dflow{\sink}$ (cf. \Cref{def:ccross})
provided $\dapp{\uu}\subset[\uat\sink,\infty)\cup\tube(Q)$,
then the following holds:
\begin{enumerate}
    \item
    $\dapp{\ww}\subset[\uat\sink,\infty)\cup\tube(Q)$ and
    \begin{gather}\label{eq:throw3}
    \bigsqcup_{\lambda=0,1}[T^{(\lambda)},\dapp{\ww}\backslash\dapp{\uu}]|_{\tube_\lambda}\subset
     \bigsqcup_{\lambda=0,1}[T^{(\lambda)},\dapp{\uu}\backslash\dapp{\ww}]|_{\tube_\lambda}.
    \end{gather}
    with equality only if $(\dapp{\uu}\backslash\dapp{\ww})|_{\tube}=\emptyset$.
    \item If $(\dapp{\uu}\backslash\dapp{\ww})|_{\tube}=\emptyset$, then
    \begin{gather}\label{eq:throw4}
    [\sink,\dapp{\ww}\backslash\dapp{\uu}]|_{\vect}\subsetneqq
    [\sink,\dapp{\uu}\backslash\dapp{\ww}]|_{\vect}.
    \end{gather}
\end{enumerate}
\end{lemma}

\begin{proof}
\textbf{For $1^\circ$}:
For any $\ww\in\dapp{\ww}\backslash\dapp{\uu}$, we have $\ww\to\source$ factors through $\dapp{\uu}\to\source$ since $\ww\in\pcts{\uu}$ but $\ww\not\in\dapp{\uu}$.
Then by \eqref{eq:Homvect}, $\dapp{\uu}\in[\uat\sink,\infty)\cup\tube(Q)$ guarantees that $\dapp{\ww}\subset[\uat\sink,\infty)\cup\tube(Q)$.

When $(\dapp{\uu}\backslash\dapp{\ww})|_{\tube_\lambda}\ne\emptyset$,
we may assume  $(\dapp{\ww}\backslash\dapp{\uu})|_{\tube_\lambda}\ne\emptyset$.
It is similar to \Cref{lem:ccross1} that \eqref{eq:throw3} holds.

When $(\dapp{\uu}\backslash\dapp{\ww})|_{\tube}=\emptyset$, suppose $(\dapp{\ww}\backslash\dapp{\uu})|_{\tube}\ne\emptyset$. Then for any $\ww\in(\dapp{\ww}\backslash\dapp{\uu})|_{\tube}$, we have
\[
\Hom(\ww,\dapp{\uu}\backslash\dapp{\ww})=\Hom(\ww,(\dapp{\uu}\backslash\dapp{\ww})|_{\tube})=0,
\]
a contradiction.
Hence \eqref{eq:throw3} becomes equal.

\textbf{For $2^\circ$}:
If $(\dapp{\uu}\backslash\dapp{\ww})|_{\tube}=\emptyset$,
then $(\dapp{\ww}\backslash\dapp{\uu})|_{\tube}=\emptyset$ as above.
Hence the two differences $\dapp{\uu}\backslash\dapp{\ww}$ and $\dapp{\ww}\backslash\dapp{\uu}$
are both in $\vect(Q)$.
Note that $\pcts{\uu}\not\subset\Star(\sink)$, we have $\II(\pcts{\uu})\ne\emptyset$ and thus $\dapp{\uu}\backslash\dapp{\ww}\ne\emptyset$.
Applying the proof of \Cref{lem:ccross1} again we obtain \eqref{eq:throw4}.
\end{proof}

Choose $\pcts{\vv}_0$ to be a cell intersecting $L$ nontrivially with $\dapp{\vv}_0\subset[\uat\sink,\infty)\cup\tube(Q)$.
Then by $1^\circ$ of \Cref{lem:ccross3}, the consecutive cells $\pcts{\vv}_1,\pcts{\vv}_2,...$ intersecting $L$ nontrivially along $\dflow{\sink}$ satisfy
\[
\dapp{\vv}_k\subset[\uat\sink,\infty)\cup\tube(Q).
\]
Assume $\II|_{\tube}(\pcts{\vv}_0)\ne\emptyset$.
We claim that there exist $N_1\ge0$ such that
\[
\II|_{\tube}(\pcts{\vv}_{N_1})\subsetneqq\II|_{\tube}(\pcts{\vv}_{0}).
\]
Suppose not, then $\II|_{\tube}(\pcts{\vv}_{k+1})=\II|_{\tube}(\pcts{\vv}_{k})$ for any $k\ge0$.
Hence
\[
(\dapp{\vv}_{k+1}\backslash\dapp{\vv}_k)|_{\tube}=\emptyset
\]
for any $k\ge0$.
By $2^\circ$ of \Cref{lem:ccross3}, we have \eqref{eq:throw4}.
Applying \Cref{lem:dd} to $\II|_{\vect}$ we deduce that
$\II|_{\vect}(\pcts{\vv}_k)=\emptyset$ for $k\gg0$.
Hence
\[
(\dapp{\vv}_{k+1}\backslash\dapp{\vv}_k)|_{\vect}=\emptyset
\]
for $k\gg0$.
In such a case, $\dapp{\vv}_{k+1}\backslash\dapp{\vv}_k=\emptyset$.
Since there is a cell-crossing, $\dapp{\vv}_{k+1}\ne\dapp{\vv}_{k}$
then $\dapp{\vv}_{k+1}\subsetneqq\dapp{\vv}_{k}$.
Hence $\dapp{\vv}_{k}=\emptyset$ when $k\gg0$ due to the finiteness of $\dapp{\vv}_0$.
This contradiction proves our claim.
Repeating the process, we obtain a descending chain
\[
\II|_{\tube}(\pcts{\vv}_{0})\supsetneqq\II|_{\tube}(\pcts{\vv}_{N_1})\supsetneqq
\II|_{\tube}(\pcts{\vv}_{N_2})\supsetneqq\ldots
\]
and deduce that $\II|_{\tube}(\pcts{\vv}_k)=\emptyset$ when $k\gg0$.
From the proof above, we also have
$\II|_{\vect}(\pcts{\vv}_k)=\emptyset$ when $k\gg0$.
$L$ finally enters $\Star(\sink)$ by \eqref{criterion3} and then ends at $\sink$.

As a consequence, any $L$ intersecting $\pcts{\vv}$ nontrivially with $\dapp{\vv}\subset[\sink,\infty)\cup\tube$ has endpoint $\sink$.
Dually, when $\dapp{\vv}\subset\tube\cup(-\infty,\source]$,
$L$ has endpoint $\source$.
Then each $\sink$-leaf $L$ in the $\sink$-foliation is either compact or semi-compact by \eqref{eq:ext},
i.e. the $\sink$-foliation is semi-compact.

%=========================================================

%=========================================================
\end{document}